\documentclass[11pt]{amsart}

\usepackage{enumitem, amssymb}
\usepackage{fontawesome}
\usepackage{hyperref}
\usepackage{cleveref}
\usepackage{mathabx}
\usepackage{tikz}
\usepackage{tikz-cd}
\usetikzlibrary { decorations.pathmorphing, decorations.pathreplacing, decorations.shapes, } 
\usepackage{pgfplots}
\usepackage[all,cmtip]{xy}
\usepackage{capt-of}

\numberwithin{equation}{section}

    \DeclareFontFamily{U}{wncy}{}
    \DeclareFontShape{U}{wncy}{m}{n}{<->wncyr10}{}
    \DeclareSymbolFont{mcy}{U}{wncy}{m}{n}
    \DeclareMathSymbol{\Sh}{\mathord}{mcy}{"58} 

\usepackage{graphicx}
\usepackage{color}
\usepackage{scalerel}
\usepackage{etoolbox}

\usepackage{mathtools}
 
\newcommand{\lthen}{\rightarrow}

\newcommand{\Index}[1]{\mathbb{#1}}
\newcommand{\Ideal}[1]{\mathcal{#1}}

\newcommand{\Bool}[2]{\mathcal{P}(\Index{#1})/\Ideal{#2} }

\newcommand\vartextvisiblespace[1][.5em]{%
  \makebox[#1]{%
    \kern.07em
    \vrule height.3ex
    \hrulefill
    \vrule height.3ex
    \kern.07em
  }
}

\DeclareMathOperator{\MA}{MA}
\DeclareMathOperator{\OCAT}{OCA_T}

\DeclareMathOperator{\SL}{SL}
\DeclareMathOperator{\PSL}{PSL}

\DeclareMathOperator{\diag}{diag}

\DeclareMathOperator{\Aut}{Aut}

\DeclareMathOperator{\tp}{tp} 
\DeclareMathOperator{\qftp}{qftp} 
 
\DeclareMathOperator{\Fin}{Fin}
\DeclareMathOperator{\alg}{alg}

\DeclareMathOperator{\dom}{dom}

\newcounter{my_enumerate_counter}

\DeclareMathOperator{\Th}{Th}

\newcommand{\cP}{\mathcal P}

\newcommand{\cN}{\mathcal N}

\newcommand{\Sym}{\mathfrak S}

\newtheorem{thm}{Theorem}
\newtheorem{coro}[thm]{Corollary}

\newcommand{\bbN}{{\mathbb N}}

\newcommand{\bbJ}{{\mathbb J}}
\newcommand{\bbI}{{\mathbb I}}

\newcommand{\cA}{{\mathcal A}}
\newcommand{\calL}{{\mathcal L}}

\newcommand{\cU}{{\mathcal U}}

\DeclareMathOperator{\Id}{Id}

\newcommand{\cstar}{$\mathrm{C}^*$}

\DeclareMathOperator{\eq}{eq}

\newcommand{\bbP}{\mathbb P}

\DeclareMathOperator{\supp}{supp}

\newcommand{\bbZ}{{\mathbb Z}}
\newcommand{\bbT}{\mathbb T} 
 
\newcommand{\bbQ}{\mathbb Q}
\newcommand{\bbR}{\mathbb R}
\newcommand{\cJ}{{\mathcal J}}
\newcommand{\cI}{{\mathcal I}}

\newcommand{\calH}{{\mathcal H}}

\newcommand{\fC}{\mathfrak C}
\newcommand{\fD}{\mathfrak D}
\newcommand{\rs}{\restriction}

\newcommand{\cC}{\mathcal C}
\newcommand{\cF}{\mathcal F}

\newcommand{\cG}{\mathcal G}

\newcommand{\cB}{\mathcal B}
\newcommand{\cK}{\mathcal K}
\newcommand{\cM}{\mathcal M}

\newcommand{\cPN}{\cP(\bbN)}

\newcommand{\cPNF}{\cP(\bbN)/\Fin}

\let\strokeL\L
\DeclareRobustCommand{\L}{\ifmmode\mathbf{L}\else\strokeL\fi}

\renewcommand{\phi}{\varphi}

\DeclareMathOperator{\GL}{GL} 

\makeatletter
\newcommand{\customlabel}[2]{%
	\protected@write \@auxout {}{\string \newlabel {#1}{{#2}{\thepage}{}{}{}}}}
\makeatother

\newtheorem{theorem}{Theorem}[section]
\newtheorem*{theorem*}{Theorem}
\newtheorem{proposition}[theorem]{Proposition}
\newtheorem{problem}[theorem]{Problem}
\newtheorem*{proposition*}{Proposition}
\newtheorem{lemma}[theorem]{Lemma}
\newtheorem*{lemma*}{Lemma}
\newtheorem{corollary}[theorem]{Corollary}
\newtheorem*{corollary*}{Corollary}
\newtheorem{fact}[theorem]{Fact}
\newtheorem*{fact*}{Fact}
\theoremstyle{definition}
\newtheorem{definition}[theorem]{Definition}

\newtheorem*{definition*}{Definition}
\newtheorem{claim}[theorem]{Claim}
\newtheorem*{claim*}{Claim}

\newtheorem*{conjecture*}{Conjecture}

\newtheorem{question}[theorem]{Question}

\theoremstyle{remark}
\newtheorem{example}[theorem]{Example}
\newtheorem*{example*}{Example}
\newtheorem{remark}[theorem]{Remark}
\newtheorem*{remark*}{Remark}

\newtheorem*{note*}{Note}
\newtheorem*{question*}{Question}

\newtheoremstyle{TheoremNum}
        {\topsep}{\topsep}              
        {\itshape}                      
        {}                              
        {\bfseries}                     
        {.}                             
        { }                             
        {\thmname{#1}\thmnote{ \bfseries #3}}
\theoremstyle{TheoremNum}

\DeclareMathOperator{\head}{head}
\DeclareMathOperator{\tail}{tail}

\usepackage{ulem}
\usepackage{todonotes}
\setuptodonotes{noinline}

\colorlet{KyleColor}{yellow}
\colorlet{PierreColor}{cyan}
\colorlet{IlijasColor}{red}

\DeclareMathOperator{\cw}{cw}

\author[Farah, I.]{Ilijas Farah}

\address{Department of Mathematics and Statistics\\
	York University\\
	4700 Keele Street\\
	North York, Ontario\\ Canada, M3J
	1P3 and Matemati\v{c}ki Institut SANU, Kneza Mihaila 36, Belgrade 11001, Serbia}

\urladdr{https://ifarah.mathstats.yorku.ca}
\email{ifarah@yorku.ca}

\author[Gannon, K.]{Kyle Gannon}
\address{ Beijing International Center for Mathematical Research (BICMR) \\ Peking University \\ Beijing, China.}
\urladdr{http://faculty.bicmr.pku.edu.cn/~kyle/}
\email{kgannon@bicmr.pku.edu.cn}

\author[Touchard, P.]{Pierre Touchard}
\address{Institut für Algebra, Technische Universität Dresden, 01062 Dresden, Germany}
\email{pierre.touchard@tu-dresden.de}

\title[Coordinate recognition]{Coordinate recognition: General theory, Groups, and other surprises}
\date{\today}

\pgfplotsset{compat=1.18}

\begin{document}
\begin{abstract}
A class of structures \emph{recognizes coordinates} if any reduced product of structures from said class witnesses a certain kind of rigidity phenomenon. We provide several equivalent characterizations of this property. This property has (at least) two remarkable consequences, one set-theoretic and one model-theoretic, for reduced products of structures of the said class. First, under appropriate set-theoretic assumptions every isomorphism between such reduced products associated with the Fr\'echet ideal lifts (modulo a finite change) to an isomorphism between products of the original structures. Second, with an additional mild assumption, it implies a strong quantifier elimination result.  Of note, we show that a class recognizes coordinates if and only if an individual formula witnesses a certain syntactic property. We also consider many concrete classes of structures and determine whether or not they recognize coordinates. We place heavy emphasis on well-known classes of groups, such as permutation groups, acylindircally hyperbolic groups, quasisimple groups, free products, and graph products,  but we also discuss other classes of structures.  
\end{abstract}
	\maketitle
	\section{Introduction}
	This work is motivated by problems in both set theory and model theory. With regard to set theory, the central motivation arises from problems involving rigidity of quotient structures. 
	 In 1979, Shelah described a forcing extension of the universe where all automorphisms of the Boolean algebra $\cP(\bbN)/\Fin$ are induced by a bijection $f_\Phi$ between two cofinite subsets of~$\bbN$ (\cite{Sh:Proper}). Another way to state Shelah's result is that there is a Boolean algebra endomorphism $\Phi_*$ of $\cPN$ such that the following diagram  commutes (vertical arrows correspond to quotient maps). 
	 	 \begin{figure}[h]
	 	\begin{tikzpicture}
	 		\matrix[row sep=1cm,column sep=1.5cm] 
	 		{
	 			& & \node (M1) {$\cPN$}; & \node (M2) {$\cPN$};&
	 			\\
	 			& & \node (Q1) {$\cPN/\Fin$}; & \node (Q2) {$\cPN/\Fin$} ;
	 			\\
	 		};
	 		\draw (M1) edge [->] node [above] {$\Phi_*$} (M2);
	 		\draw (Q1) edge [->] node [above] {$\Phi$} (Q2);
	 		\draw (M1) edge [->]   (Q1);
	 		\draw (M2) edge [->] (Q2);
	 	\end{tikzpicture}
	 \end{figure}
	 
	 In other words, in Shelah's model  every automorphism\footnote{Whenever the category of morphisms is not mentioned explicitly, it is understood that it is the category at hand.} $\Phi$ of $\cPNF$ can be lifted by an endomorphism $\Phi_*$ of $\cPN$. 
	 	 Such automorphisms are called \emph{trivial}. By a 1956 result of W. Rudin, the Continuum Hypothesis implies that there are $2^{2^{\aleph_0}}$ nontrivial automorphisms of $\cPNF$. This is an immediate consequence of the fact that $\cPNF$ is $\aleph_1$-saturated, a notion not yet isolated at the time of Rudin's work. 

Shelah's conclusion held in a very specific forcing extension of the universe (Shelah's oracle-cc forcing notion was invented for the purpose of finding this extension!), but  it was soon proved to follow from forcing axioms\footnote{We refer the reader to \cite{moore2010proper} for background on forcing axioms.} (\cite{ShSte:PFA, Ve:OCA}).
Analogous lifting results from forcing axioms were proven for quotient algebras of the form $\cPN/\cI$ for numerous Borel ideals $\cI$ on $\bbN$ ($\cPN$ is given the Cantor set topology). These early results are summarized in  \cite{Fa:Rigidity,farah2022corona}. 
	
	New impetus to the study of set-theoretic rigidity theory was given by the solution to a prominent 1977 problem in the theory of operator  algebras (\cite{BrDoFi:Unitary}) asking whether the Calkin algebra has outer automorphisms. The Calkin algebra is the quotient of the algebra $\cB(H)$ of all bounded linear operators on the separable, infinite-dimensional, complex Hilbert space modulo the ideal of compact operators, and it is generally considered to be the noncommutative analog of $\cPNF$. The analogs of Rudin's and Shelah's results were proved in \cite{PhWe:Calkin} and \cite{Fa:All}, respectively, surprisingly showing that the answer to the Brown--Douglas--Fillmore question is independent from ZFC. (The fact that the proof that all automorphisms of the Calkin algebra are inner  owes a lot to the ideas presented in \cite{Fa:Rigidity} may be even more surprising.) A far-reaching generalization of~\cite{Fa:All}  showing rigidity of coronas of other separable \cstar-algebras was obtained in \cite{vignati2018rigidity}. 
	
	The algebra $\cPNF$ is isomorphic to the reduced product of two-element Boolean algebras modulo $\Fin$, and 
		it was generally believed by the experts that the chances of a sweeping extension of Shelah's rigidity result to reduced products of other structures were slim. However, in \cite{de2023trivial} it was shown that in categories of linear orders, trees, and sufficiently random graphs, forcing axioms imply that isomorphisms between reduced products are, with the appropriate definition, trivial. In  \cite{de2024saturation} it was proven that stable (in the model-theoretic sense) reduced products associated with $\Fin$ are automatically $2^{\aleph_0}$-saturated, giving the analog of Rudin's result. In the domain of continuous structures, the first results were on rigidity of coronas of \cstar-algebras (\cite{Fa:All}, \cite{Gha:FDD},  \cite{mckenney2020forcing}, \cite{vignati2022rigidity}; see \cite[\S 17]{Fa:STCstar} for an overview), and they have recently been extended to symmetric  groups with respect to the Hamming metric (\cite{de2024automorphism}). In the context  of metric structures the situation is made more complicated (or rather, more interesting) because all considerations necessarily involve Ulam stability of approximate homomorphisms (\cite{mckenney2018ulam}; see \cite{Ul:Problems}). 
			For the state of the art of rigidity of quotient structures, both discrete and metric, see \cite{farah2022corona} and in particular \S 5 for a discussion of the role of Ulam stability. 
	
	In the present paper we continue the investigation of set-theoretic rigidity. 
	 Fundamental questions from this perspective include the following: 
	\begin{enumerate}
		\item How does the isomorphism type of a reduced product $\prod_\cI\cM_{i}$ depend on the indexed family of structures $(\cM_i)_i$ and the ideal $\cI$?
		\item Can one isolate the \emph{right} notion of a trivial isomorphism between reduced products, and prove that forcing axioms imply all isomorphisms are trivial for appropriate classes of reduced products? 
	\end{enumerate}
	 Analogous rigidity questions for reduced products of groups were studied in \cite{KanRe:Ulam} and \cite{KanRe:New}. These results are concerned with rigidity of those isomorphisms between quotients that have Borel-measurable liftings (`topologically trivial' in the terminology of \cite{farah2022corona}) and are closely connected to the well-studied Ulam-stability of approximate homomorphisms (see \cite[\S 4]{farah2022corona}). This assumption is not necessary in our context as our results are applicable to arbitrary isomorphisms.  
	
    Both of the above questions admit partial, yet substantive answers when considering whether or not our reduced product is composed of structures from a class which \emph{recognizes coordinates}. The notion of a class of structures of the same language recognizing coordinates (Definition~\ref{D.recognizes-coordinates}) was isolated in  \cite{de2023trivial}. There, it was proved that if a class $\fC$ recognizes coordinates, then forcing axioms imply that isomorphisms between reduced products of structures from $\fC$ over the Fr\'echet ideal $\Fin$ are `trivial'. By trivial, we mean that every such isomorphism between reduced products can be lifted by a bijection $\pi$ between cofinite sets and bijections $f_i\colon \cN_i\to \cM_{\pi(i)}$ (\cite[Theorem~7]{de2023trivial}). If the language is finite, then in addition we may require that  each $f_i$ is an isomorphism. If the language is infinite, this is no longer true (see e.g., \cite[Example 2.3]{de2023trivial}). Instead, for every finite fragment of the language all but finitely many of the $f_i$ are isomorphisms. 
    
    The use of additional set-theoretic axioms is necessary because the Continuum Hypothesis implies that reduced products over $\Fin$ are saturated and so isomorphism between reduced products reduces to elementary equivalence. 
	 On the other hand, some classes of structures do not recognize coordinates. In \cite[Theorem~1]{de2024saturation}	 	
	 	 it was proved that if a reduced product  of countable structures over $\Fin$ has a stable theory, then it is saturated \emph{provably in ZFC}.  In particular, such a reduced product has $2^{2^{\aleph_0}}$ automorphisms and (since there are at most  $2^{\aleph_0}$ trivial automorphisms for any reasonable definition of `trivial') the underlying class of structures cannot recognize coordinates.  
	 
	 Consequently, whether a class of structures recognizes coordinates determines the rigidity behavior associated with reduced products from that particular class. We continue the study of this property and our results proceed in two directions: general theory and examples with an emphasis on groups. First,  we prove that recognizing coordinates is necessarily witnessed by a first-order condition. 
In the introduction to \cite[\S 2.2]{de2023trivial} it was pointed out that a satisfactory proof that a theory recognizes coordinates would proceed by proving the equivalence of \eqref{1.A} and \eqref{2.A} of Theorem~\ref{T.A} below. The equivalence of recognizing coordinates with \eqref{3.A}  provides an even more palpable (and necessary) criterion  for recognizing coordinates  (all natural classes of structures are full,  see Definition~\ref{Def.Full};  for the definition and relevance of $h$-formulas see~\S\ref{S.h-formulas}). 
			 
			 \begin{thm} \label{T.A} 

                 	For a full class $\fC$ of $\mathcal{L}$-structures, the following are equivalent. 
    	\begin{enumerate}
    		\item \label{1.A} $\fC$ recognizes coordinates.
    		\item \label{2.A} For every index set $\Index I$ and every ideal $\cI$ on $\Index I$, a reduced product $\cM:=\prod_{\Ideal I} \cM_i$ of structures from $\fC$ (uniformly) interprets the Boolean algebra $\cP(\Index I)/\cI$, the system of quotient structures $\cM\rs S$, and the quotient maps $\pi_S$ (parametrized by $S\in \cP(\Index I)/\cI$). 
    		\item \label{3.A} The formula $x=x' \rightarrow y=y'$ is equivalent to an $h$-formula in the common theory of structures from $\fC$, denoted $\Th(\fC)$.
    	\end{enumerate}
			 \end{thm}

Theorem~\ref{T.A} is part of the more detailed Theorem~\ref{T.eq}. It implies that if a class $\fC$ of structures of the same language recognizes coordinates, then so does the class of all models of its theory (Corollary~\ref{C.Th(C)}). 

We also remark that recent connections with operator algebras have invigorated the interest in reduced products associated with the Fr\'echet filter. Such reduced powers are called (asymptotic) sequence algebras and the interplay between them, the original algebra, and its ultrapower, plays a very important role in classification theory (see e.g., \cite[\S 6]{white2023abstract}). These connections resulted in new results about such reduced products, such as splitting of  the exact sequence $0\to \prod_{\Fin} A\to \prod_{\cU} A \to 0$ and transfer of information between the reduced power and the ultrapower (see the introduction to \cite{farah2022between}.

Second, we provide many new classes of structures which recognize (or do not recognize) coordinates. Surprisingly,  very little structure suffices for recognizing coordinates. For instance, linear orders and sufficiently random graphs all recognize coordinates by  \cite[Proposition~2.7]{de2023trivial}. In this paper, we place particular emphasis on groups, as they provide a rich class of structures which witness versatile behaviors with respect to coordinate recognition. 	Additionally, there is a historical precedent of studying automorphism groups of products of groups.  Our research has applications to this line of research
(see \S\ref{S.Products} for details).

	The following is proved below as \Cref{T.RC}.

	\begin{thm} \label{T.RecognizesCoordinates}
		Each of the following classes of groups recognizes coordinates.
		\begin{enumerate}[label=(\alph*)]
			\item The class of all non-abelian simple groups.
			\item The class of all $\Sym_n$, for all $n\geq 3$; i.e., the symmetric groups on finite sets of size greater than or equal to 3. 
			\item The class of all dihedral groups $D_{2n+1}$, for $n\geq 1$.
			
			\item The class of all groups $\SL(n,F)$, for all $n\geq 2$ and every field $F$ with more than four elements. 
			
			\item $\{G\}$ where $G$ is any nontrivial free product. 
			
			\item The class of all graph products $\Gamma\mathcal{G}$, where $\mathcal{G} = \{G_v\}_{v\in V(\Gamma)}$ is a family of groups indexed by the vertices of $\Gamma$, such that the complement graph $\bar\Gamma$ is connected and $|G_v|\geq 3$ for at least one vertex $v$ (see \S\ref{S.GraphProducts}).
		\end{enumerate}
	\end{thm}
	
	 Together with the results of \cite{de2023trivial}, Theorem~\ref{T.RecognizesCoordinates} gives numerous corollaries; the following is a consequence of Theorem~\ref{T.Rigidity}. 
	
	\begin{coro} \label{C.Rigidity} Forcing axioms imply the following. 
		\begin{enumerate}
			\item Every automorphism of $\prod_{n\in \mathbb{N}} \Sym_n/\Fin$ lifts to an automorphism of $\prod_n \Sym_n$. 
			\item For every sequence without repetitions of the form $(m_i, F_i)$, for $i\in \bbN$ such that $m_i\in \bbN$ and $F_i$ is  a finite field, every automorphism of $\prod_{\Fin} \SL(m_i,F_i)$ lifts to an automorphism of $\prod_i \SL(m_i,F_i)$. 
		\end{enumerate}
	\end{coro}
	
One can also find an infinite $X\subseteq \bbN$ such that for every infinite subset $Y\subseteq X$ with $X\setminus Y$ infinite,  the assertion  $\prod_{n\in X}\Sym_n /\Fin \cong \prod_{n\in Y} \Sym_n/\Fin$  is independent of ZFC (this is a very special case of Corollary~\ref{C.Ghasemi}).  


	 In the opposite direction, we also isolate classes of groups which do not recognize coordinates.
The following is Theorem~\ref{T.eDNR}. 

\begin{thm}\label{T.example:DNR} Any class of groups that contains any of the following does not recognize coordinates. 
	\begin{enumerate}[label=(\alph*)]
		\item \label{DNR.abelian} Any group that has  a nontrivial direct summand.
		\item \label{DNR.decomposable} Any group that admits a nontrivial homomorphism into its center. More generally, if the class contains two groups such that there exists a nontrivial homomorphism from one into the center of the other, then it does not recognize coordinates. \label{DNR.S3SL25} In particular, any class of groups that contains both $\Sym_3$ and $\SL(2,5)$ does not recognize coordinates. 
		\item \label{DNR.GLn} The group $\GL(n,F)$ for $n\geq 2$ and any field $F$.
		\item  \label{DNR.Q8} The group $Q_{8} = \langle -1,i,j,k: (-1)^{2} = e, i^{2} = j^{2} = k^{2} =ijk = -1  \rangle$. 
		\item \label{DNR.D2n} The dihedral group $D_{2n}$  of the $2n$-gon for $n\geq 1$. 
		\item \label{DNR.nilpotent} Any nilpotent group. 
		\item \label{DNR.Graph} Any nontrivial graph product $\Gamma\cG$ such that the complement graph $\bar\Gamma$ is not connected (see \S\ref{S.GraphProducts}). 
	\end{enumerate} 
\end{thm} 
	 
	  In the initial stages of this work we had hoped to isolate a clear-cut characterization of classes of groups that recognize coordinates. This revealed itself to be a difficult task, as the union of two classes which recognize coordinates does not necessarily recognize coordinates: often, different groups recognize coordinates for different reasons. For example,  while each one of $\Sym_3$ and $\SL(2,5)$ recognizes coordinates by itself, no class containing both of these groups recognizes coordinates (this is a consequence of Theorem~\ref{T.example:DNR}~\ref{DNR.decomposable}).  The following gives an even stronger obstruction to the existence of a clear-cut characterization of classes of groups that recognize coordinates (see Theorem~\ref{T.Limiting}).

	\begin{thm}\label{T.Limiting.0}
		There is a family $\fD$  of groups that  does not recognize coordinates, but every finite subset of $\fD$ does.  
	\end{thm}
    
    Classical model-theoretic considerations provide another important motivation for this work. Reduced products of $\mathcal{L}$-structures are themselves naturally equipped with an $\mathcal{L}$-structure, and the celebrated Feferman--Vaught theorem provides a natural expansion $\mathcal{L}^+$ where these products eliminate quantifiers. It is however not clear whether (or when) this expansion $\mathcal{L}^+$ is a definable expansion of the original language $\mathcal{L}$ in the reduced product. It is surprising that this natural question has remained unanswered. A key property turns out to be the definability of the support function. To our knowledge, this property was first identified by Medvedev and Van Abel in \cite{MvA18}, where they prove a general theorem about (non-reduced) products and apply it to products of finite fields. 
    We take the occasion to expand on their work and extend their analysis to arbitrary reduced products. In turn, we prove the following theorem (see Theorem~\ref{C.QuestionDecomp.1}; since this is a consequence of folklore results, we prove it in the preliminary section). 
    
	\begin{thm}\label{C.QuestionDecomp}
    Let $\mathcal{M}:=\prod_\Ideal{I} \mathcal{M}_i$ be a reduced product, and assume that $\mathcal{M}$ interprets the relative support function $\supp_=$. Assume there exists a fundamental set $\Phi$ of $h$-formulas for the pair $(\mathcal{M}_i)_{i \in \Index{I}}$ and $\Ideal{I}$. Then $\mathcal{M}$, as an $\mathcal{L}$-structure, already interprets all functions of the language $\calL^+$ and $\mathcal{M}$ eliminates quantifiers relative to $\Bool{I}{I}$ in the language $\calL_\Phi^+$.
      \end{thm}
    The language $\calL_\Phi^+$ is introduced in \Cref{C.FV.reduced}, and is a fragment of the Feferman--Vaught language $\mathcal{L}^+$. See Definition~\ref{Def.h-formulas} for $h$-formulas and Definition~\ref{D.Fundamehtal} for fundamental set of $h$-formulas. 
    As an application, we give an explicit definable expansion of a reduced power of a finite symmetric group $\Sym_n$ for  $n\geq 4$ with $n\neq 6$ which eliminates quantifiers (see \Cref{lem:cyclesformula}).
    
    The paper is outlined as follows: In Section~\ref{S.Preliminary}, we recall some basic preliminaries concerning reduced products. In Section~\ref{S.Main}, we prove that recognizing coordinates is equivalent to interpreting a definable support function. Sections~\ref{S.GroupsRecognizing} and \ref{S.GroupsNotRecognizing} are focused on classes of groups. In Section~\ref{S.GroupsRecognizing}, we prove that certain classes of groups recognize coordinates while in Section \ref{S.GroupsNotRecognizing} we provide several examples of classes of groups which do not recognize coordinates. Section~\ref{S.Other}  focuses on showing other families of structures recognize coordinates, including a fan favorite: the non-associative magma colloquially known as \emph{rock-paper-scissors}. Section~\ref{S.Limiting} focuses on \emph{limiting examples} or examples which imply that the general theory of recognizing coordinates is quite complicated. The final section contains concluding remarks, open questions, and some ties to loose ends.  

	\tableofcontents
    \subsection*{Acknowledgments} 
We are grateful to Patrick Lutz for pointing out the relevance of $\SL(2,5)$ (see \S\ref{S.Perfect}) and to Forte Shinko for his useful comment on the case $n=6$ of \Cref{T.RC}\ref{2.Cplus}, and for answering some questions in a preliminary draft of our paper. We thank the anonymous referee for making numerous suggestions that helped improve the paper. We also thank Marcus Greferath for finding and correcting several errors in an earlier version of this article. We are grateful to Montserrat Casals--Ruiz for pointing out to us the relevance of domains (in the sense of \cite{baumslag1999algebraic}) to our results. This remark substantially increased the class of known groups that recognize coordinates (see Corollary~\ref{C.Montse} and Theorem~\ref{T.GraphProduct.Montse}). We are also indebted to the organizers of the 2026 Beijing Model Theory conference for providing a pleasant and stimulative environment in which this interaction took place.
In the proof of  Theorem~\ref{T.Limiting}, Grok by xAI was used to verify the well-known fact that for every abelian group $G$ there is a nontrivial homomorphism from $G$ into $\bbQ/\bbZ$. 
IF was partially supported by NSERC. KG was supported by the Fundamental Research Funds for the Central Universities, Peking University, grant no.\ 7100604835 and by the National Natural Science Fund of China, grant no.\ 12501001. 
		\section{Preliminaries}\label{S.Preliminary}
        Our notation regarding first-order logic is standard.
	This section contains preliminaries concerning reduced products, recognizing coordinates, $h$-formulas, ultraproducts of reduced products, and the Feferman--Vaught theorem. We fix a language $\calL$ throughout.   
	
   \subsection{Reduced products}
	A class of $\mathcal{L}$-structures will often be denoted by~$\fC$. We use the symbols $\Index{I}, \Index{J}$ to denote index sets and the symbols $\Ideal{I},\Ideal{J}$ to denote ideals on those index sets. We recall that an ideal $\Ideal{I}$ on $\Index{I}$ is a non-empty collection of subsets of $\mathbb{I}$ which is downward closed, closed under finite unions, and does not contain $\mathbb{I}$. Given an ideal $\Ideal{I}$ on $\Index{I}$, one constructs the quotient Boolean algebra $\mathcal{P}(\Index{I})/\Ideal{I}$  via the following identification: 
    \[ A\simeq B \ \Leftrightarrow \ A \vartriangle B \in \Ideal{I}. \]
If $\Index{I}$ is an index set, then we let $\Fin(\Index{I})$ denote the ideal of all finite subsets of $\Index{I}$. When there is no possibility of confusion, we simply write $\Fin$ in place of $\Fin(\Index{I})$. We begin by recalling the definition of the reduced product. 
\begin{definition}[Reduced product]
Let $(\mathcal{M}_i)_{i\in \Index{I}}$ be a collection of $\mathcal{L}$-struc\-tu\-res. We let $\prod_{i \in \Index{I}} \mathcal{M}_i/\Ideal{I}$, or, more frequently,  $\prod_\Ideal{I}\mathcal{M}_i$, denote the quotient of the product $\prod_{ i \in \Index{I}} \mathcal{M}_i$ by the following equivalence relation: if $(a_i)_i$ and $(b_i)_i$ are in $ \prod_{i \in \Index{I}}\mathcal{M}_i$, then  
	\[ (a_i)_i\simeq (b_i)_i \ \Leftrightarrow \{i\in \Index{I} \mid a_i \neq b_i\} \in \Ideal{I}.\]
        Formally, we let $[(a_i)]_{\Ideal{I}}$ denote the equivalence class of $(a_i)_{i}$. This quotient is naturally equipped with an $\mathcal{L}$-structure. For any $n$-ary relation symbol $R\in \mathcal{L}$ and $[(a^{1}_i)]_{\Ideal{I}},\dots,[(a^{n}_i)]_{\Ideal{I}}$ in $\prod_\mathcal{I}\mathcal{M}_i$, 
	\[\models R([(a^{1}_i)]_{\Ideal{I}},\dots,[(a^{n}_i)]_{\Ideal{I}}) ~\Leftrightarrow ~ \{i\in \Index{I} \mid \mathcal{M}_i\models \neg R(a^{1}_i,\dots,a^{n}_i)\}\in \Ideal{I}, \]
	and for any $n$-ary function symbol $f \in \mathcal{L}$ and  $[(a^{1}_i)]_{\Ideal{I}},\dots,[(a^{n}_i)]_{\Ideal{I}}$ in $\prod_{\Ideal{I}}\mathcal{M}_{i}$,
	\[f([(a^{1}_i)]_\Ideal{I},\dots,[(a^{n}_i)]_\Ideal{I})= [f(a^{1}_i,\dots,a^{n}_i)]_\Ideal{I} .\]
    The structure $\prod_{\mathcal{I}} \mathcal{M}_i$ is called the \emph{reduced product}.
    We often abuse notation and identify elements in $\prod_{\Ideal{I}} \mathcal{M}_i$ with the elements in $\prod_{i \in \Index{I}} \mathcal{M}_i$ when there is no possibility of confusion or error.     
	If all $\cM_i$ are isomorphic to a fixed structure $\cN$, then the reduced product $\prod_\cI \cM_i$ is called a \emph{reduced power} of~$\cN$. 

    We remark that in the case where $\cI=\{\emptyset\}$, the reduced product reduces (!) to the product. 
\end{definition}

\begin{definition}
	An ideal $\Ideal{I}$ on an index set $\Index{I}$ is \emph{atomless} if the quotient Boolean algebra $\cP(\Index I)/\Ideal I$ is atomless. By metonymy, we also call a reduced product $\prod_\Ideal{I}\mathcal{M}_i$ \emph{atomless} if $\Ideal{I}$ is atomless.
\end{definition}

The following example illustrates a situation that we would like to avoid; where information is lost. Settings like the one below can lead to pathological behavior and annoying/unrewarding case work.   
\begin{example}\label{Ex.NotFull}
    Let $\mathcal{L}:=\{ \leq, R\}$ be a language with two binary relation symbols. Notice that both graphs $(G,R)$ and linear orders $(L,\leq) $ can be viewed as $\mathcal{L}$-structure by interpreting remaining symbol as empty. Fix a nontrivial linear order and a nontrivial graph relation on a set $A$, denoted by $\leq^{A}$ and $R^{A}$ respectively. Now consider the product $\mathcal{M}=\prod_{n \in \mathbb{N}} \mathcal{M}_n$ where each $\mathcal{M}_n$ has universe $A$ and 
    \[\mathcal{M}_n =\begin{cases}
        R^{\mathcal{M}_n} = \emptyset; \hspace{.2cm} \leq^{\mathcal{M}_n} = \leq^{A},  & \text{ if $n$ is even, }\\
        R^{\mathcal{M}_n} = R^{A}; \hspace{.2cm} \leq^{\mathcal{M}_n} = \emptyset,  & \text{ if $n$ is odd. }
    \end{cases}\]
    By the definition of the reduced product (with respect to the ideal $\{\emptyset\}$), both relations in the structure $\mathcal{M}$ are empty, i.e., $R^\mathcal{M}=\leq^\mathcal{M} =\emptyset$.
\end{example}
In order to avoid the pathological situation described in Example~\ref{Ex.NotFull}, we will consider only classes $\fC$ of structures satisfying the following. 
    
\begin{definition}\label{Def.Full}
    A class $\fC$ of $\mathcal{L}$-structures is called \emph{full} if 
    \begin{enumerate}
        \item For every predicate $P$ in $\mathcal{L}$ and all $\mathcal{M}$ in $\fC$, $P^{\mathcal{M}} \neq \emptyset$. 
        \item Every $\mathcal{M}$ in $\fC$ has at least three elements. 
    \end{enumerate}
    
\end{definition}
We remark that there is no issue in assuming that there exists some structure $\mathcal{M}$ in $\fC$ and some $n$-ary predicate $P$ such that $P^{\mathcal{M}}$ is $\mathcal{M}^{n}$.  
Notice that if $\mathcal{L}$ is a functional language, then all classes of $\mathcal{L}$-structures automatically satisfy condition $(1)$ above. 

    \subsection{Recognizing coordinates} Here we recall the definition of recognizing coordinates. Definition~\ref{D.recognizes-coordinates} below is \cite[Definition 2.5 and Definition 2.6]{de2023trivial} (see \S\ref{S.Products} for the variant of this notion for direct products). In order to state the definition, we introduce some notation which will be used throughout the text. 
    
Assume that $\Ideal{I}$ is an ideal on a set $\Index I$ and $(\mathcal{M}_i)_{i \in \Index{I}}$ is an indexed family of $\mathcal{L}$-structures. Consider the reduced product $\mathcal{M} := \prod_{\Ideal{I}} \mathcal{M}_i$. For 
$S\subseteq \Index I$ we write $\cM\rs S$ for the quotient $(\prod_{i\in S} \mathcal{M}_i)/(\Ideal I\rs S)$ of $\cM$. The quotient map is denoted~$\pi_S$. If $S\in \Bool II$ then  we write $\cM\rs S$ for $(\prod_{i\in \tilde S} \mathcal{M}_i)/(\Ideal I\rs \tilde S)$ and $\pi_S$ for $\pi_{\tilde S}$,  where $\tilde S\subseteq \Index I$ is any set that satisfies  $[\tilde S]_{\Ideal I}=S$. Clearly $\cM\rs S$ does not depend on the choice of representative $\tilde S$ for $S$.  


	\begin{definition}\label{D.recognizes-coordinates}
		 An isomorphism $\Phi$ between reduced products $\cM:=\prod_\Ideal{I} \cM_i$ and $\cN:=\prod_\Ideal{J} \cN_j$ is \emph{isomorphically coordinate respecting} if  there is an isomorphism $\alpha\colon \Bool II\to \Bool JJ$ such that for all $S\in \Bool II$ we have a function 
		\[
		\Phi_S\colon \cM\rs S\to \cN \rs \alpha(S),
		\]
		defined by 
		\[
		\Phi_S(\pi_S(a))=\pi_{\alpha(S)}(\Phi(a)), 
		\]
		making the following diagram commute:
		\begin{center}
			\begin{tikzpicture}
				\matrix[row sep=1cm,column sep=1cm]
				{
					& & \node (M1) {$\cM$}; && &\node (M2) {$\cN$};&
					\\
					& & \node (Q1) {$\cM\rs S$}; &&& \node (Q2) {$\cN\rs \alpha(S)$} ;
					\\
				};
				\draw (M1) edge [->] node [above] {$\Phi$} (M2);
				\draw (Q1) edge [->] node [above] {$\Phi_S$} (Q2);
				\draw (M1) edge [->] node [left] {$\pi_S$} (Q1);
				\draw (M2) edge [->] node [right] {$\pi_{\alpha(S)}$} (Q2);
			\end{tikzpicture}
		\end{center}
			A first-order theory $T$ is said to \emph{recognize coordinates} if every isomorphism between reduced products of models of $T$ is isomorphically coordinate respecting. 

		More generally, if $\fC$ is a class of structures of the same language (not necessarily axiomatizable), then $\fC$ is said to \emph{recognize coordinates} if every isomorphism  between arbitrary reduced products of structures from $\fC$ (possibly with repeated structures) is isomorphically coordinate respecting. 
		
		Moreover, a class $\fC$ is said to \emph{recognize coordinates in $\bbN$} \footnote{ In \cite{de2023trivial}, the notion of ``recognizing coordinates in $\mathbb{N}$" was simply called ``recognizing coordinates". We choose to change the terminology to clarify our presentation.} if every isomorphism  between arbitrary reduced products of the form $\prod_{n\in \bbN} \cM_n/\Ideal I$ of structures from $\fC$ (again, possibly with repeated structures) is isomorphically coordinate respecting. 
        In \Cref{T.eq}, we will show in particular that this distinction is superfluous and that a class recognizes coordinates in $\bbN$ if and only if it recognizes coordinates.  
	\end{definition}

	    \subsection{h-formulas}\label{S.h-formulas}

	In this section we first recall the definition of an $h$-formula. We then recall some analysis of $h$-formulas in reduced products by Palmgren and Omarov. Finally, we conclude with some analysis of our own concerning $h$-formulas and \emph{definable support functions}. We prove that a full class of structures admits such functions if and only if a certain formula (involving only logical symbols and equality) is equivalent to an $h$-formula. The following definition is taken from \cite{palmgren}, \cite{palyutin1}, \cite{palyutin2} (see \cite[\S 2]{de2024saturation}). 
	
	\begin{definition}[The class of $h$-formulas] \label{Def.h-formulas} 
		The class of \emph{$h$-formulas} is the smallest class of formulas $\mathcal C$ containing all atomic formulas and if $\varphi$ and~$\psi$ belong to $\mathcal C$, then so do
		\[
		\varphi\wedge\psi,\,\,(\exists x)\varphi,\,\,(\forall x)\varphi,\text{ and }(\exists x)\varphi\wedge (\forall x)(\varphi\rightarrow\psi).
		\]
        \end{definition}

It is not difficult to see that the class of formulas that contains all atomic formulas and is closed under the construct $(\exists x)\varphi\wedge(\forall x)(\varphi\rightarrow\psi)$ coincides with the class of $h$-formulas. 
The following fact will be used tacitly and frequently without being explicitly mentioned. 

\begin{fact} 
	 If $x$ and $\bar{y}$ are the variables occurring freely in $\phi$ and if theory $T$ implies  $(\forall \bar y) (\exists x)\varphi(x,\bar y)$ then $T$ implies that  $(\forall x)(\varphi\to \psi)$ is equivalent to an $h$-formula. 
\end{fact}

In the situation described by the fact above, we will slightly abuse terminology and say that the formula $(\forall x)(\varphi\to \psi)$ is an $h$-formula. The following is an analog of \L o\' s's theorem for reduced products (see \cite{palyutin2} for details). 
	
	\begin{theorem}\label{T.Palyutin}
		If $\cI$ is an ideal on an index set $\mathbb{I}$, $(\mathcal{M}_i)_{i \in \mathbb{I}}$ is an indexed family of $\mathcal{L}$-structures, $\mathcal{M}=\prod_\Ideal{I} \mathcal{M}_i$, $\bar{a}$ is an element of  $\mathcal{M}^{|\bar{x}|}$ with representative $(\bar{a}_i)$, and $\varphi(\bar x)$ is an $h$-formula, then
		\begin{align*}
			\mathcal{M}&\models \varphi(\bar a)\text{ if and only if } \{i\in \mathbb{I} \mid \mathcal{M}_i\models \neg \varphi(\bar a_i)\}\in \cI.
		\end{align*}
	\end{theorem}
	
    Proposition~\ref{lemma:Palmgren} below is a uniform version of \cite[Lemma 3]{palmgren}. There it was credited to Palyutin. It is a consequence of \cite[Lemma 2]{palmgren} which implies that every atomless reduced product satisfies the so-called simple cover property, but we sketch a proof for the reader's convenience. This fact is a relative to the conclusion of McKinsey's Lemma, \cite[Lemma 9.1.7]{Hodg:Model}. 
			
			\begin{proposition}\label{lemma:Palmgren}
				For every $\mathcal{L}$-formula $\phi(\bar{x})$, there is a Boolean combination of $h$-formulas $\psi(\bar{x})$ such that 
				\[(\forall \bar{x})~ (\phi(\bar{x}) \leftrightarrow \psi(\bar{x})),\]
				holds in all atomless reduced products. 
			\end{proposition}
			
			\begin{proof}
				In \cite[Lemma 2]{palmgren}, it was proved that every atomless reduced product has the simple cover property, and in \cite[Lemma 3]{palmgren} it was proved that the simple cover property implies $\varphi$ is equivalent to a Boolean combination of $h$-formulas. Although the latter asserts that the theory of the reduced product implies $\varphi$ is equivalent to a Boolean combination of $h$-formulas, the proof gives that the simple cover property implies this, and that the Boolean combination depends only on  $\varphi$. 
			\end{proof}

			By definition, a class $\fC$ recognizes coordinates if and only if every reduced product of models in $\fC$ does so. We will show that there is only one reduced product that matters, and recognizing coordinates in with respect to this model implies every other reduced product of structures from $\fC$ recognizes coordinates. 
			
			\begin{definition}\label{D.K(C)} For a class $\fC$ of $\calL$-structures, we define a reduced product $\cK(\fC)$ of structures in $\fC$. Let 
				\begin{align*}
				\bbT=\bbT(\fC)&=\{\Th(\cN): \cN\in \fC\},\\
				\bbI=\bbI(\fC)&=\bbN\times \bbT(\fC),  
				\end{align*}
and for each $(j,T)\in \bbI$, fix $\cN_{j,T}\in \fC$ such that $\cN_{j,T}\models T$. Let 
\[
\cK(\fC)=\prod_{(j,T)\in \bbI} \cN_{j,T}/\Fin(\bbI). 
\]
			\end{definition}

			The following proposition is a direct consequence of \cite[Lemma 4]{Oma91}. 

			\begin{proposition}[Omarov]\label{Prop:Omarov}
				Let $\fC$ be a class of $\mathcal{L}$-structures and let $\phi(\bar{x})$ be a satisfiable $\mathcal{L}$-formula with the following property: there is an $\mathcal{L}$-formula $\psi(\bar{x})$ such that for all atomless reduced products $\mathcal{M}=\prod_\Ideal{I}\mathcal{M}_i$ of structures from $\fC$ and for all $\bar{a} \in \prod_\Ideal{I}\mathcal{M}_i$, \begin{equation}\label{eq.Omarov}
					\prod_\Ideal{I}\mathcal{M}_i\models \phi(\bar{a}) \Leftrightarrow \{i \in \Index{I} \mid \mathcal{M}_i\models\neg \psi(\bar{a}_i)\} \in \Ideal{I}.
					\end{equation}
					\begin{enumerate}
						\item \label{1.Omarov} Then $\phi(\bar{x})$ is equivalent to an $h$-formula $\Phi(\bar{x})$ in the common theory of all atomless reduced products from $\fC$.
						\item \label{2.Omarov} Also, $\Phi(\bar{x})$ and $\psi(\bar{x})$ are equivalent modulo $\Th(\fC)$.  
					\end{enumerate}
					Moreover, if \eqref{eq.Omarov} holds for $\cK(\fC)$, then it holds in every reduced product of structures from $\fC$. 
	            \end{proposition}

            We had difficulty recovering the proof directly from the English translation (\cite{Oma91}). We also had difficulty understanding the proof from the original Russian version (\cite{omarov1991syntactical}). These proofs seem to be substantially different. Thus, we take the opportunity to write a proof using the work of Palmgren and Palyutin (\Cref{lemma:Palmgren}). 
%
	
			\begin{proof}
					For simplicity, we denote the ideal $\Fin(\bbI)$ by $\Fin$ when the index set~$\bbI$ is clear from the context. 
				
				We first assume that \eqref{eq.Omarov} holds for $\cK(\fC)$ and prove \eqref{1.Omarov} and \eqref{2.Omarov}. 
				
				Note that if $\cK(\fC)$ satisfies \eqref{eq.Omarov}, then so does $\cK(\fC)^m$ (identified with the reduced product $\prod_{(i,j,T)\in m\times \bbN\times \bbT(\fC)} \cN_{j,T}/\Fin$) for every $m\geq 1$, because $\cK(\fC)\cong \cK(\fC)^m$ by an isomorphism that 
				lifts to an isomorphism between $\prod_{\xi\in T(\bbI)} N_\xi$ and $\prod_{\xi\in \bbI(\fC)\times m} N_\xi$. To be precise, this isomorphism corresponds to a bijection $\alpha\colon T(\bbI)\to T(\bbI)\times m$ such that 
				\[
				\alpha[\{(i,T): i\in \bbN\}]=\{ ((i,T),j): i\in \bbN, j<m\}
				\] 
				 for every $T\in T(\fC)$.
				
				For simplicity of notation, we work with formulas in a single variable. Throughout we let $T$ be the common theory of all atomless reduced products of structures from $\mathfrak{C}$ and all implications are modulo this theory. We first prove Statement \eqref{1.Omarov}. 
				 \Cref{lemma:Palmgren} implies that  there are $h$-formulas $A_1, \dots, A_n$ such that $\Th(\cK(\fC))$ implies that $\phi$ is equivalent to a Boolean combination $B$ of $A_1, \dots, A_n$.
				We will show that $\phi$ is equivalent to a subconjunction of $A_1,\dots, A_n$. Let $l_1<\dots<l_m \leq n$ such that $B \rightarrow A_{l_1} \wedge \cdots \wedge A_{l_m}  $ and assume that $\{l_1,\dots,l_m\}$ is a maximal collection of indices (under inclusion) with this property, in the sense that  for all $l_{m+1}\leq n$ with $l_{m+1}\neq l_j$, $j\leq m$, 
                \[B \not \rightarrow \bigwedge_{j\leq m+1}  A_{l_j}.\]
				To simplify notation, by re-indexing we may assume $l_j=j$ for all $j\leq m$. Therefore,  for all $j>m$, the formula $B \wedge \neg A_j$ is consistent. We will show that $B$ is equivalent to $A_1 \wedge \cdots \wedge A_m$. 
				For each  $m< j\leq n$, let $\cK_j$ be an isomorphic  copy of $\cK(\fC)$; denote $\bbI_j$ its index set, and let $a_j$ be an element of $\cK_j$ satisfying $B \wedge \neg A_j$. Let $\cK_{m}$ be another isomorphic copy of $\cK(\fC)$, with index set $\bbI_m$, and let $a$ be an element of $\cK_m$ satisfying $A_1 \wedge \cdots \wedge A_m$. As $B$ is consistent, there is an element $b$ in $\cK_{m}$ satisfying $B$ (and thus it also satisfies $A_1 \wedge \cdots \wedge A_m$). The structure 
                \[
                \cK=\prod_{m\leq j\leq n} \cK_j
                \] 
                is isomorphic to $\cK(\fC)^{n-m+1}$ and (as mentioned at the beginning of this proof) also to $\cK(\fC)$, via an isomorphism that respects the reduced product structure. Specifically, with $\bbJ= (n-m+1)\times \bbN\times \bbT(\fC) = \bbI_m  \sqcup \cdots \sqcup \bbI_n $ we have that $\prod_{(i,j,T)\in \bbI}\cN_{j,T}/\Fin$ is naturally isomorphic to $\cK$. For convenience we write $\bbI=(n-m+1)\times \bbN\times \bbT(\fC)$ and $\cM_l$ for $\cN_{(j,T)}$ if $l=(i,j,T)$, so that 
                \[
                \cK=\prod_{l\in \bbI} \cM_l/\Fin. 
                \]  
                Since the isomorphism respects the reduced product structure, \eqref{eq.Omarov} holds for this representation of $\cK$. 
                
     
                In $\cK$
                 the element $b^\frown a_{m+1}^\frown \cdots ^\frown a_n$ satisfies $B$ (by condition \eqref{eq.Omarov}).
                 Moreover, since the $A_i$'s are $h$-formulas, multiple applications of \Cref{T.Palyutin} imply that both $a^\frown a_{m+1}^\frown \cdots ^\frown a_n$ and $b^\frown a_{m+1}^\frown \cdots ^\frown a_n$ satisfy, \[A_1\wedge \cdots \wedge A_m \wedge \neg A_{m+1} \wedge \dots \wedge \neg A_n.\]

\begin{figure}[h]
    \centering
                        
\begin{tikzpicture}
  \begin{scope}
    \draw (-1,1) ellipse (2.2 and 1.4);
    \fill[opacity=0.1] (-1,1) ellipse (2.2 and 1.4);
    \node at (-3.2,2) {$A_1$};
  \end{scope}

  \begin{scope}
    \draw (1,1) ellipse (2.2 and 1.4);
    \node at (3.2,1.8) {$A_2$};
  \end{scope}

  \begin{scope}
    \draw (0,2.2) ellipse (2.2 and 1.4);
    \node at (0,3.8) {$A_3$};
  \end{scope}

  \begin{scope}
    \draw (0,-0.2) ellipse (2.2 and 1.4);
    \node at (0,-1.8) {$A_4$};
  \end{scope}

    \filldraw (0,1) circle (1pt) (0,1)  node [text=black,right] {$b$};
    \filldraw (-1.5,1.8) circle (1pt) (-1.5,1.8)  node [text=black,above] {$a_2$};
    \filldraw (0,-0) circle (1pt) (0,-0)  node [text=black,above] {$a_3$};
    \filldraw (0,2) circle (1pt) (0,2)  node [text=black,below] {$a_4$};
    \filldraw (-2.4,0.7) circle (1pt) (-2.4,0.7)  node [text=black,left] {$a^\frown a_{2}^\frown a_3 ^\frown a_4$};
    \filldraw (-2.4,1.5) circle (1pt) (-2.4,1.5)  node [text=black,left] {$b^\frown a_{2}^\frown a_3 ^\frown a_4$};
    \node at (5,1) {for $n=4$, $m=1$};
\end{tikzpicture}
    \end{figure}
                                 Therefore, as $B$ is a Boolean combination of the $A_i$'s, $a^\frown a_{m+1}^\frown \cdots ^\frown a_n$ also satisfies $B$. By condition \eqref{eq.Omarov}, if we let $(c_i)_i\in \prod_{i\in \bbI} \cM_i$ be a representative of $a^\frown a_{m+1}^\frown \cdots ^\frown a_n$, then 
                    \[
                    \{i\in \Index{J}\mid \mathcal{M}_i\models \neg \psi(c_i) \} \in \Fin(\bbJ).
                    \]
It follows that
\[
                 \{i\in \Index{I}_m \mid \mathcal{M}_i\models \neg \psi(c_i) \}\in \Fin(\bbI_m).
                 \]
                By Condition \eqref{eq.Omarov} again, 
                $\cK=\prod_\Ideal{I}\mathcal{M}_i $ satisfies $\phi(a).$
                As the element $a$ was chosen arbitrarily to satisfy $A_1 \wedge \cdots \wedge A_m$, this shows that, 
                \[A_1 \wedge \cdots \wedge A_m \rightarrow B,\]
                and thus we conclude that $\phi$ is equivalent to an $h$-formula (namely, the conjunction of the formulas $A_1,\dots,A_m$). We let $\Phi$ be this $h$-formula.
                We now prove Statement \eqref{2.Omarov}. By our hypothesis and Theorem~\ref{T.Palyutin}, we have for any $(a_i)_i$, 
                \[\tag{*}\{i \in \Index{I} \mid \mathcal{M}_i \models \neg \psi(a_i)\}\in \Ideal{I} \ \Leftrightarrow \ \{i \in \Index{I} \mid \mathcal{M}_i \models \neg \Phi(a_i)\} \in \Ideal{I}. \]
                Assume towards a contradiction that there is an $\Ideal{I}$-positive set $J \subseteq \Index{I}$ such that, for any $i\in J$, we have, 
                \[\mathcal{M}_i \models (\exists x ) \neg (\psi(x)  \leftrightarrow \Phi(x)). \]
                For $i\in J$, let $ a_i$ be an element in $\mathcal{M}_i$ such that $\neg (\psi(a_i)  \leftrightarrow \Phi(a_i))$ holds. At the cost of restricting to a smaller positive set and swapping the role of $\psi$ and $\Phi$, we may assume that for any $i\in J$, $\neg \psi(a_i) \wedge \Phi(a_i)$ holds in $\mathcal{M}_i$. For $i \in \Index{I}\setminus J$, we choose $a_i$ in $\mathcal{M}_i$ 
            such that $\Phi(a_i)$ holds (if it exists), otherwise take $a_i$ to be anything. Note that since $\Phi$ is a satisfiable $h$-formula, the set of indices for which $\Phi$ is not satisfiable belongs to $\Ideal{I}$. Then we conclude that, 
                 \[\{i \in \Index{I} \mid \mathcal{M}_i \models \neg \psi(a_i)\}\notin \Ideal{I}, \]
                 and 
                 \[\{i \in \Index{I} \mid \mathcal{M}_i \models \neg \Phi(a_i)\}\in \Ideal{I}, \]
                 contradicting $(*)$.

                 This proves that if $\cK(\fC)$ satisfies \eqref{eq.Omarov}, then $\varphi$ and $\psi$ are equivalent to the same $h$-formula $\Phi$ modulo $T(\fC)$. By Palyutin's theorem (Theorem~\ref{T.Palyutin}), \eqref{1.Omarov} and \eqref{2.Omarov} hold in every reduced product of structures in $\fC$. This concludes the proof.  
			\end{proof}
%

            We will now focus on particular kinds of $h$-formulas, those only involving logical connectives and equality. As stated in the introduction, it turns out that whether or not a certain class of structures recognizes coordinates is equivalent to whether or not a particular formula is logically equivalent to an $h$-formula. To be more formal, we must first define a theory relative to an indexed family of structures and a fixed ideal.            

	\begin{definition}\label{Def:ThI}
        Let $\Index{I}$ be an index set, $\Ideal{I}$ an ideal on $\Index{I}$, and $(\mathcal{M}_i)_{i \in \Index{I}}$ an indexed family of $\mathcal{L}$-structures. We let $\Th_\Ideal{I}(\mathcal{M}_i, i\in \mathbb{I})$ denote the common theory of $\mathcal{M}_i$ for all but $\Ideal{I}$-many $i\in \Index{I}$.  In other words,
		\[
		\Th_\Ideal{I}(\mathcal{M}_i, i \in \mathbb{I}):=\{\phi\mid  \{i \in \Index{I} \mid \mathcal{M}_i \models \neg \phi\} \in \Ideal{I}\}.
		\] 
        We sometimes write $\Th_{\Ideal{I}}(\mathcal{M}_i, i \in \Index{I})$ as $\Th_{\Ideal{I}}(\mathcal{M}_i)$ when there is no possibility of confusion.  
	\end{definition}

If $\cI$ is a proper ideal (which will always be the case), then $\Th_\cI(M_i)$ is a consistent theory because the conjunction of any finite set of sentences in it is satisfied in all but $\cI$ many $\cM_i$. A glance at Definition~\ref{D.K(C)} is sufficient to verify the following. 

\begin{lemma}\label{L.ThFinK(C)}
For any class $\fC$, with $\bbI(\fC)$ and $\cN_{j,T}$ for $(j,T)\in \bbI(\fC)$ as in Definition~\ref{D.K(C)}, we have that $\Th_{\Fin(\bbI(\fC))} (\cN_{j,T}:(j,T)\in \bbI(\fC))=\Th(\fC)$. \qed 
\end{lemma}
        Two variants of $h$-formulas are useful for our purposes. These $h$-formulas should be thought of as \emph{comparing elements in the reduced product on coordinates}. The next proposition shows that if the \emph{relative comparison formula} is an $h$-formula then the \emph{comparison formula} is an $h$-formula. 
        
        \begin{proposition}\label{prop:interpretingAllZsupport} 
		Let $\Index{I}$ be an index set, $\Ideal{I}$ an ideal on $\Index{I}$, and $(\mathcal{M}_i)_{i \in \Index{I}}$ be an indexed family of $\mathcal{L}$-structures. Suppose that for all but $\Ideal{I}$-many $i\in \Index{I}$, $\mathcal{M}_i$ has at least $3$ elements. Furthermore, suppose the formula $x=z \rightarrow y=z$ is equivalent to an $h$-formula modulo $\Th_\Ideal{I}(\mathcal{M}_i, i\in \mathbb{I})$. Then the formula  $x=z \rightarrow y=w$ is  equivalent to an $h$-formula modulo $\Th_\Ideal{I}(\mathcal{M}_i, i\in \mathbb{I})$. 
			\end{proposition}
	
	\begin{proof} 
		The proposition is the immediate consequence of the following two claims. 
		
		\begin{claim}\label{Claim.psineq}
			With respect to the theory of any structure $\mathcal{M}$ with at least two elements, the formula $x \neq y$ is equivalent to $$\psi_{\neq}(x,y):= (\forall z) \ (x=y \rightarrow z=y),$$ and is therefore equivalent to an $h$-formula.
		\end{claim}
            \begin{proof}
                By assumption $(x=y \rightarrow z=y)$ is equivalent to an $h$-formula $\phi(x,y,z)$. Therefore, from the definition of $h$-formulas, $(\forall z) \phi(x,y,z)$ is an $h$-formula equivalent to $\psi_{\neq}(x,y)$. Since $\mathcal{M}$ has at least two elements, the latter is clearly equivalent to $x\neq y$.
            \end{proof}

				\begin{claim}
					 With respect to the theory of any structure $\mathcal{M}$ with at least three elements, the formula $x=z \rightarrow y=w$ is equivalent to the following formula:
					\begin{align}\label{eq.Zsupp}
						\begin{split}(\exists u)(\exists x')(\exists y') &  \bigl(\psi_{\neq}(u,x)\wedge \psi_{\neq}(u,y)\bigr.
						\\
						\wedge   &(x=z \leftrightarrow x'=z)\\ \wedge  &(\forall x'') \left ( (x=z \leftrightarrow x''=z) \rightarrow (x''=u \rightarrow x'=u) \right)\\
						\wedge  & (y=w \leftrightarrow y'=w)\\ \wedge &(\forall y'') \left ( (y=w \leftrightarrow y''=w) \rightarrow (y''=u \rightarrow y'=u)\right)\\
						 \bigl.\wedge  &y'=u \rightarrow x'=u\bigr).
						\end{split}
					\end{align}
				\end{claim}
                \begin{proof}
                    To facilitate the argument, we view the variables $x,y,z,w,\dots $ as functions from a set $A$ to the structure $\mathcal{M}$. The claim then follows if we take $A$ to be a singleton and identify  $\mathcal{M}$ with the set of functions from $A$ to~$\mathcal{M}$.
                    The first line of (\ref{eq.Zsupp}) introduces $u$, a function nowhere equal to $x$ and $y$. The second line defines $x'$ as the function which coincides with $x$ whenever $x$ coincides with $z$, and coincides with $u$ everywhere else. The third line defines similarly $y'$ as the function which coincides with~$y$ whenever $y$ coincides with $w$, and coincides with $u$ everywhere else. The fourth line compares where $x'$ and $y'$ coincide with $u$. 
                    
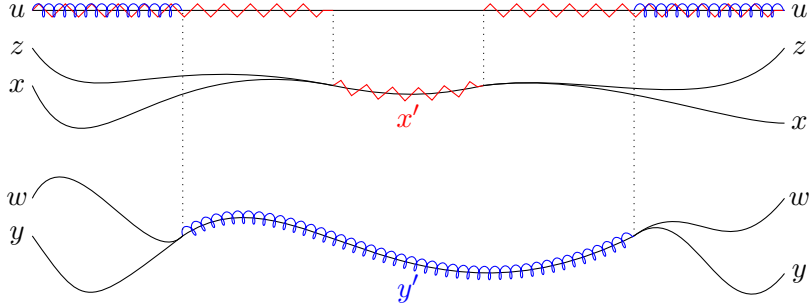
                            \begin{tikzpicture}
            	\draw (0,4) node{$\phantom{.}$};
            	\draw (0,3) -- (10,3);
            	\draw (-0.2,3) node{$u$};
            	\draw (10.2,3) node{$u$};
            	
            	\draw[decorate,decoration={zigzag},color=red]   (0,3) -- (4,3);\draw[decorate,decoration=zigzag,color=red]   (6,3) -- (10,3);
            	
            	\draw[decorate,decoration=coil,segment length=4pt,color=blue] (0,3) -- (2,3);
            	\draw[decorate,decoration=coil,segment length=4pt,color=blue] (8,3) -- (10,3);
            	
            	\draw (0,0) ..  controls (0.7,-1) .. (2,0);
            	\draw (0,0.5) ..  controls (0.5,1.5) and (1.5,-0.5) .. (2,0);
            	\draw (-0.2,0) node{$y$};
            	\draw (-0.2,0.5) node{$w$};
            	
            	\draw (2,0) ..  controls (3.5,1) and ((5,-1.5)  .. (8,0);
            	\draw[decorate,decoration=coil,segment length=4pt,color=blue] (2,0) ..  controls (3.5,1) and ((5,-1.5)  .. (8,0);
            	\draw[color=blue] (5,-0.7) node{$y'$};
            	
            	\draw (8,0) ..  controls (8.8,0.6) and (9.2,-1.5) .. (10,-0.5);
            	\draw (8,0) ..  controls (8.8,0.6) and (9.2,-0.5) .. (10,0.5);
            	\draw (10.2,-0.5) node{$y$};
            	\draw (10.2,0.5) node{$w$};
            	\draw (0,2) ..  controls (0.7,0.5) and (1.5,2.5) .. (4,2);
            	\draw (0,2.5) ..  controls (0.7,1.5) and (1.5,2.5) .. (4,2);
            	\draw (-0.2,2) node{$x$};
            	\draw (-0.2,2.5) node{$z$};
            	
            	\draw (4,2) ..  controls (5,1.8) and (5.5,1.9)  .. (6,2);\draw[decorate,decoration=zigzag,color=red]   (4,2) ..  controls (5,1.8) and (5.5,1.9)  .. (6,2);
            	\draw[color=red] (5,1.6) node{$x'$};
            	
            	\draw (6,2) ..  controls (7.5,2.2) and (9.2,1.5) .. (10,1.5);
            	\draw (6,2) ..  controls (7.5,2.2) and (9.2,1.5) .. (10,2.5);
            	\draw (10.2,2.5) node{$z$};
            	\draw (10.2,1.5) node{$x$};
            	\draw[dotted] (2,0)-- (2,3);
            	\draw[dotted] (8,0)-- (8,3);
            	\draw[dotted] (4,2)-- (4,3);
            	\draw[dotted] (6,2)-- (6,3);
            \end{tikzpicture}

                    To see that \eqref{eq.Zsupp} is equivalent to the original formula, notice that $x'$ coincides with $u$ more often than $y'$ does  if and only if $y$ coincides with $w$ more often  than $x$ coincides with $z$. 
                \end{proof}

                We proceed to prove Proposition~\ref{prop:interpretingAllZsupport}.  We may assume that $\cM_i$, for $i\in \bbI$, are structures with at least three elements each and that $x=z\to y=z$ is equivalent to an $h$-formula modulo $\Th_\cI(\cM_i, i\in \bbI)$. In order to prove that $x=z \rightarrow y=w$ is equivalent to an $h$-formula modulo $\Th_\Ideal{I}(\mathcal{M}_i, i\in \mathbb{I})$ it suffices to prove that each of the conjuncts in the body of the formula~\eqref{eq.Zsupp} is equivalent to an $h$-formula.  This is true for the first line by Claim~\ref{Claim.psineq}. 
                
                Note that 
                $x=z \leftrightarrow x'=z$ is an abbreviation for
                \[
                x=z \rightarrow x'=z \wedge x'=z \rightarrow x=z,
                \]
                and is therefore equivalent to an $h$-formula. This takes care of the second, fourth, and sixth conjuncts. 

                Since the formula
                \[(\exists x') ~(x=z \leftrightarrow x'=z ),\]
                trivially holds for all $x$ and $z$, we deduce that the remaining two conjuncts in  \eqref{eq.Zsupp} are equivalent to $h$-formulas.  
					\end{proof}

            We now connect the syntax of $h$-formulas with the semantics of reduced products. The \emph{comparison formulas} from the previous proposition corresponds to certain kinds of support functions. The \emph{relative support function of} $\prod_\Ideal{I}\mathcal{M}_i$, denoted $\supp_=$,  is the map which associates to two elements $a$ and $b$ with representations $(a_i)_i$ and $(b_i)_i$, the element of $\Bool{I}{I}$ where these two elements coincide: 
    \[
    ((a_i)_i,(b_i)_i) \buildrel{\supp_=}\over\longmapsto \big[\{ i \mid a_i=b_i\}\big]_\Ideal{I}.  
    \]
    Intuitively, by the variant of \strokeL o\'s's theorem for reduced products, the definability of the \emph{relative support function} in all reduced products should be an intrinsic property of the class of structures. This is what we show in \Cref{lemma:RecognizingCoordinateEquivalence} and \Cref{T.InterpretingSupport} below (see Definition~\ref{Def:ThI} for $\Th_\cI$).

			\begin{lemma} \label{lemma:RecognizingCoordinateEquivalence}
				Let $\Index{I}$ be an index set, $\Ideal{I}$ be an ideal on $\Index{I}$, and $(\mathcal{M}_i)_{i \in \Index{I}}$ be an indexed family of $\mathcal{L}$-structures. Then the following Condition (\ref{item:RCE1}) implies Condition (\ref{item:RCE2}). 
				
				\begin{enumerate}
					\item\label{item:RCE1} The formula,     \[ x=x' \rightarrow y=y',\]
					is equivalent to an $h$-formula $\Phi(x,x',y,y')$ in $\Th_\Ideal{I}(\mathcal{M}_i, i\in \mathbb{I})$.
					\item\label{item:RCE2} The Boolean algebra $\Bool{I}{I}$ and the relative support function,
                    \[\supp_=: \prod_\Ideal{I} \mathcal{M}_i \rightarrow \Bool{I}{I} \text{ via } 
                     ((g_i)_i, (g_i')_i) \rightarrow \big[\{ i \in \Index{I} \mid  g_i = g_i'\}\big]_\Ideal{I},\] 
                    are interpretable in $\prod_\Ideal{I} \mathcal{M}_i$. 
				\end{enumerate}
			\end{lemma}
			\begin{proof}
            Assume Condition \eqref{item:RCE1}. Fix $a,a',b$, and $b'$ in  $\prod_\Ideal{I} \mathcal{M}_i$ with representatives $(a_i)_{i},(b_i)_{i},(a'_i)_{i},$ and $(b'_i)_{i}$ respectively. Then,
				\begin{align*}
					\supp_=(a,a') \subseteq \supp_=(b,b') & \Leftrightarrow  \{i \in \Index{I} \mid a_i = a_i' \wedge b_i \neq b_i'\} \in \Ideal{I}\\
					& \Leftrightarrow  \{i \in \Index{I}  \mid \neg (a_i = a_i' \rightarrow b_i = b_i')\} \in \Ideal{I}\\
					&  \Leftrightarrow  \{i \in \Index{I}  \mid \mathcal{M}_i \models \neg \Phi(a_i,a_i',b_i,b_i')\} \in \Ideal{I} \\
					& \Leftrightarrow   \prod_\Ideal{I} \mathcal{M}_i \models \Phi(a,a',b,b').
				\end{align*}
				It follows that $\Bool{I}{I}$ is interpretable in $\prod_\Ideal{I} \mathcal{M}_i$  as a poset. Moreover, all Boolean operations are interpretable in this poset. For example, $A\cap B$ is the maximal (under~$\leq$) $C$ such that $C\leq A$ and $C\leq B$, and $A\setminus B$ is the maximal $C$ such that $C\leq A$ and $C\cap B=\emptyset$. 
				 Therefore, $\Bool{I}{I}$ is interpretable as a Boolean algebra.  The natural projection is the relative support function, therefore Condition~(\ref{item:RCE2}) holds.
			\end{proof}
			\begin{theorem}\label{T.InterpretingSupport}
				Let $\fC$ be a full class of $\mathcal{L}$-structures. 
				The following conditions are equivalent: 
				\begin{enumerate}
					\item\label{item:IntSupp1} The relative support function $\supp_=$ is interpretable in all reduced products of structures from $\fC$.
					\item\label{item:IntSupp2} The relative support function $\supp_=$ is interpretable in $\cK(\fC)$ (see Definition~\ref{D.K(C)}). 
					\item\label{item:IntSupp3} The formula $x=x' \rightarrow y=y'$ is equivalent to an $h$-formula in the common theory $\Th(\fC)$ of $\fC$.  
				\end{enumerate} 
			\end{theorem}

			\begin{proof}
				The implication $\eqref{item:IntSupp1}\Rightarrow \eqref{item:IntSupp2}$ is obvious and $\eqref{item:IntSupp3}\Rightarrow \eqref{item:IntSupp1}$ follows from ~\Cref{lemma:RecognizingCoordinateEquivalence}.
				It remains to show that $\eqref{item:IntSupp2}$ implies $\eqref{item:IntSupp3}$. Assume that $\eqref{item:IntSupp2}$ holds and consider $\cK(\fC)$ as in Definition~\ref{D.K(C)}. Let $\phi(x,x',y,y')$ be the $\mathcal{L}$-formula defining the relation $\supp_=(a,a') \subseteq \supp_=(b,b')$ in $\prod_\Ideal{I} \mathcal{M}_i$. It follows that for all $a,a',b,b' \in \prod_\Ideal{I} \mathcal{M}_i$, with respective representatives $(a_i)_i,(a'_i)_i,(b_i)_i,(b'_i)_i$:
              \[
				\prod_\Ideal{I} \mathcal{M}_i \models \phi(a,a',b,b') \Leftrightarrow \{i \in \Index{I} \mid \neg (a_i = a_i' \rightarrow b_i = b_i')\} \in \Ideal{I}.
				\]
				By \Cref{Prop:Omarov} \eqref{1.Omarov}, in $\Th(\cK(\fC))$ the formula $\phi(x,x',y,y')$ is equivalent to an $h$-formula $\Phi(x,x',y,y')$, and by Theorem~\ref{T.Palyutin} we have 
                \[
                 \prod_\Ideal{I} \mathcal{M}_i \models \Phi(a,a',b,b') \Leftrightarrow \{i \in \Index{I} \mid \neg \Phi(a_i,a_i',b_i,b_i')\} \in \Ideal{I}.
                \]
                By \Cref{Prop:Omarov} \eqref{2.Omarov}, for all but  finitely many $i\in \Index{I}$ we have that  
                \[
                \mathcal{M}_i \models \forall x,x',y,y'~\left(\Phi(x,x',y,y') \leftrightarrow (x=x'\rightarrow y=y')\right).
                \]
                Since $\Th_{\Fin(\bbI(\fC))} (\cN_{j,T}:(j,T)\in \bbI(\fC))=\Th(\fC)$ (Lemma~\ref{L.ThFinK(C)}), we have the equivalence in $\Th(\fC)$, as required, and this concludes the proof.
                \end{proof}

			\subsection{Ultraproducts of reduced products}
			
			An ultraproduct is a special kind of reduced product. The purpose of this subsection is to show that the class of (atomless) reduced products constructed from a fixed class of $\mathcal{L}$-structures $\fC$ is closed under ultraproducts. In the following proposition, it is understood that $\cI_j$ is an ideal on an index set $\bbI_j$, and similarly for $\cJ$ on $\bbJ$. 
			
			
			\begin{proposition}\label{P.ultraproduct}
				Given a class $\fC$ of $\mathcal{L}$-structures (elementary or not), the class of all reduced products of structures in~$\fC$ is closed under taking ultraproducts. Also, the class of all atomless reduced products of structures from $\fC$ is closed under taking ultraproducts. 
			\end{proposition}

			\begin{proof}  
				Let $\cU$ be an ultrafilter on a set $\bbJ$ and let $\cI_j$, for $j\in \bbJ$, be a family of ideals. 
				We will produce an ideal $\cJ$ such that the ultraproduct $\prod_\cU (\prod_{\cI_j} \mathcal{M}_{ij})$  is isomorphic to the reduced product $\prod_\cJ \mathcal{M}_{ij}$ for every (appropriately indexed) family of structures $\mathcal{M}_{ij}$. We will also show that if all $\cI_j$ are atomless then so is $\cJ$, thus proving both parts of Proposition~\ref{P.ultraproduct}.  
				
				It will be convenient to assume that all $\cI_j$ are ideals on the same index set~$\bbI$. Towards this let $\bbI=\bigsqcup_{j\in \bbJ} \bbI_j$ (the disjoint union of $\bbI_j$)  and replace~$\cI_j$ with the ideal on $\bbI$ generated by $\cI_j$ and the set $\bigcup_{k\in \bbJ, k\neq j} \bbI_k$ for every $j$. Clearly this does not affect $\prod_{\cI_j} \mathcal{M}_{ij}$ (take $\mathcal{M}_{ij}$ to be arbitrary for $i\in \bbI\setminus \bbI_j$).

				Let $\cU_*$ be the ideal on $\bbJ$ dual to $\cU$ and note that for any index family of $\mathcal{L}$-structures $(A_j)_{j \in \mathbb{J}}$, the ultraproduct $\prod_{\cU} A_j$ is literally the same as the reduced product $\prod_{\cU_*} A_j$. 	
				For $X\subseteq \bbJ\times \bbI$ and $j\in \bbJ$, let $X_j=\{i\mid (j,i)\in X\}$. 
				On the set $\bbJ\times \bbI$ define $\cJ$ to be the set of all $X\subseteq \bbJ\times \bbI$ such that, 
				\[
				\{j\mid X_j\notin \cI_j\}\in \cU_*. 
				\]
				\begin{claim}\label{Claim.Identity}
					For all structures $\mathcal{M}_{ij}$, $i\in \bbI$, $j\in \bbJ$ of the same language the structures $\prod_\cU (\prod_{\cI_j} \mathcal{M}_{ij})$ and $\prod_\cJ \mathcal{M}_{ij}$ are isomorphic. Moreover, the identity map on $\prod_{i\in \bbI, j\in \bbJ} \cM_{ij}$ is a lifting of an isomorphism. 
				\end{claim}
				
				\begin{proof} Each one of the reduced products  in the statement of this claim is a quotient of $\prod_{i\in \bbI, j\in \bbJ} \cM_{ij}$.  Let $(a_{ij})$ and $(b_{ij})$ be two sequences indexed by $\bbJ\times \bbI$. Then the set $X:=\{(j,i)\mid a_{i,j}\neq b_{i,j}\}$ belongs to $\Ideal{J}$ if and only if $(\cU j) X_j\in \cI_j$ (where $(\cU j)$ is the quantifier `for $\cU$-many $j\in \bbJ$').  Therefore $(a_{ij})$ and $(b_{ij})$ are  equal modulo $\cJ$ if and only if  for $\cU$-many $j$, $a_{ij}$ and $b_{ij}$ are equal modulo $\cI_j$. This implies that the identity map on $\prod_{i\in \bbI, j\in \bbJ} \cM_{ij}$ is a lifting of an isomorphism between $\prod_\cU (\prod_{\cI_j} \mathcal{M}_{ij}))$ and $\prod_\cJ \mathcal{M}_{ij}$. 
					The same argument shows that this bijection preserves relation symbols, and (since we may interpret functions and constants by relations) is therefore an isomorphism as required. 
				\end{proof}
				
				\begin{claim}\label{C.atomless} 
					 The ideal $\cJ$ is atomless if and only if $\cI_j$ is atomless for $\cU$-many~$j$.
				\end{claim}
				
				\begin{proof} Assume that $\cI_j$ is atomless for $\cU$-many $j$. 
					Let $X\subseteq \bbJ\times \bbI$ be a $\cJ$-positive set. Then the set $Y=\{j\mid X_j\notin \cI_j\}$  belongs to $\cU$. Since $\cI_j$ is atomless, there is a partition $X_j=X_j^0\sqcup X_j^1$ into $\cI_j$-positive sets. Then the sets $X^k=\bigcup\{\{j\}\times X_j^k\mid j\in Y\}$, for $k=0,1$, form a partition of $X$ into $\cJ$ positive sets. Since $X$ was arbitrary, the conclusion follows. 
					
					For the converse, since $\cU$ is an ultrafilter we may assume that $\cI_j$ has an atom for $\cU$-many $j$. The union of these atoms gives an atom in $\cJ$. 
				\end{proof}

				The conclusion of the first part of Proposition~\ref{P.ultraproduct} follows by the two claims. 
			\end{proof}
            \begin{remark}
It is clear that the isomorphism between             $\prod_\cU (\prod_{\cI_j} \mathcal{M}_{ij})$ and $\prod_\cJ \mathcal{M}_{ij}$    as in Claim~\ref{Claim.Identity} respects the Boolean algebra and  $\supp_=$, and we therefore have  
\[
                      \textstyle (\prod_\cJ \mathcal{M}_{ij}, \cP(\Index{J}\times \Index{I})/\Ideal{J},\supp_=)\simeq \prod_\cU (\prod_{\cI_j} \mathcal{M}_{ij}, \cP(\Index{I}_j)/\Ideal{I}_j, \supp_=) .
                      \]
            \end{remark}

	\subsection{A folklore reformulation of Feferman--Vaught}\label{section:FV}
	
	As the title of this subsection suggests, we discuss a folklore variant of the Feferman--Vaught theorem in the context of reduced products. In broad strokes, this theorem explains how to understand the theory of the reduced product by understanding the theory of the indexed models along with the quotient Boolean algebra. As usual, we let $\Index{I}$ be an index set, $\Ideal{I}$ be an ideal on $\Index{I}$, and $(\mathcal{M}_i)_{i \in \Index{I}}$ be an indexed family of $\mathcal{L}$-structures. We denote by~$\mathcal{L}_{\textrm{Bool}}$ the language of Boolean algebra:
	\[\mathcal{L}_{\textrm{Bool}}:=(\cap ,\cup, 0, 1, ^\complement).\]

    \begin{definition}
	Let $\phi(\bar{x})$ be a formula with free variables $\bar{x}$ and let $\bar{a} \in (\prod_\Ideal{I} \mathcal{M}_i)^{\vert \bar{x}\vert}$. The $\phi$-support of $\bar{a}$, denoted by 
	$\supp_{\phi(\bar{x})}(\bar{a})$, is  
	the element, 
    \[\big[\{ i\in \Index{I} \mid \mathcal{M}_i \models \phi(\bar{a}_i)\} \big]_\Ideal{I}  \in \Bool{I}{I}.\]
	We call the function, 
	\[\supp_{\phi(\bar{x})}: \left(\prod_\Ideal{I} \mathcal{M}_i\right)^{\vert \bar{x} \vert } \rightarrow \Bool{I}{I}, \text{ via } (\bar{a}_i) \mapsto \supp_{\phi(\bar{x})}(\bar{a}) 
	\]
	the \emph{$\phi$-support function} of $\prod_\Ideal{I} \mathcal{M}_i$.
    If $\theta$ is an $\mathcal{L}$-sentence, we denote by $c_\theta$  the following element of $\Bool{I}{I}$:
	\[c_\theta:=\big [\{ i\in \Index{I}  \mid \mathcal{M}_i \models \theta \}\big ]_\Ideal{I}.\] 

\end{definition}
     The theorem of Feferman--Vaught can be adapted to reduced products (see e.g. \cite[Theorem 6.3.2]{CK90}). We give below a folklore reformulation of this theorem in terms of relative quantifier elimination. We refer to \cite[Appendix A]{Rid17} for definitions of relative quantifiers elimination and related concepts.
     
	\begin{proposition}[Feferman--Vaught for reduced products] \label{P.FV.reduced}
		The reduced product~$\mathcal{M}\coloneqq \prod_\Ideal{I} \mathcal{M}_i$ eliminates quantifiers relative to  $\Bool{I}{I}$ in the following two-sorted language $\mathcal{L}^+$: 
    \[\mathcal{L}\cup \left\{\supp_{\phi(\bar{x})} : \phi(\bar{x}) \text{ an }  \mathcal{L}\text{-formula} \right\} \cup \mathcal{L}_{\textrm{Bool}}\cup \{c_\theta : \theta \text{ an }  \mathcal{L}\text{-sentence} \}.\]
	\end{proposition}
    The Feferman--Vaught theorem also asserts that there is an algorithmic procedure for quantifier elimination in this setting. Since we are not concerned with decidability, we will never (ever) mention this algorithm again. 
		
		
		
	
	\begin{proof}
		We first observe that the sort,
        \[
        B:= (\Bool{I}{I};\cap ,\cup, 0, 1, \{c_\theta : \theta \text{ an } \mathcal{L}\text{- sentence} \} ),
        \]
        is \emph{closed} in the following sense: in the language, there are neither functions from $B$ to the main sort $\mathcal{M}$, nor predicates in $B^n\times \mathcal{M}^m$ for strictly positive integers $n,m$. Any (parameter-free) $\mathcal{L}^{+}$-formula is equivalent, in the theory of the reduced product, to a Boolean combination of formulas of the form:
		\begin{itemize}
			\item $\psi_B(\supp_{\phi}(\Bar{x}))$
			\item $\phi(\Bar{x})$
		\end{itemize}
		where $\phi$ is an $\mathcal{L}$-formula, and $\psi_B$ is a formula  in the language $\mathcal{L}_{\textrm{Bool}} \cup \{c_\theta: \theta \text{ an } \mathcal{L}\text{-sentence} \}$.
		By \cite[Theorem 6.3.2]{CK90}, a formula of the second kind can be expressed as a formula of the first kind. Therefore, we have eliminated all quantifiers in the main sort $\prod_\Ideal{I} \mathcal{M}_i$. Since the sort is closed, relative quantifier elimination follows by \cite[Remark A.8]{Rid17}. 
	\end{proof}

	This language for quantifier elimination can sometimes be optimized by using the notion of a \emph{fundamental} set of formula. 
	
	\begin{definition}\label{D.Fundamehtal}
	We call a set $\Phi$ of $\mathcal{L}$-formulas  \emph{fundamental} for the pair $(\mathcal{M}_i)_{i \in \Index{I}}$ and $\Ideal{I}$ if every $\mathcal{L}$-formula is equivalent, in $\Th_{\cI}(\mathcal{M}_i)$, to a Boolean combination of formulas in $\Phi$.
	\end{definition}
	
 The following (easy) proposition implies that we may obtain quantifier elimination in a fragment of $\mathcal{L}^{+}$ which contains the support functions $\supp_\phi$ where $\phi$ ranges over a fundamental set $\Phi$. 
	
	\begin{proposition}\label{P.boolfund}
		Let $\phi(\bar{x})$ and $\psi(\bar{x})$ be two $\mathcal{L}$-formula. Then every $\bar{a}$ in $ \left(\prod_\Ideal{I}\mathcal{M}_i\right)^{|\bar{x}|}$ satisfies the following: 
\begin{align*}
		\supp_{\neg\phi}(\bar{a})& = \supp_{\phi}(\bar{a})^\complement,\\
\supp_{\phi\wedge\psi}(\bar{a})&=\supp_\phi(\bar{a})\cap\supp_\psi(\bar{a}),\\
\supp_{\phi\vee\psi}(\bar{a})&=\supp_\phi(\bar{a})\cup\supp_\psi(\bar{a}). 
\end{align*}
	\end{proposition}

\begin{corollary}\label{C.FV.reduced} If $\Phi$ is a fundamental set of $\mathcal{L}$-formulas for the pair $(\mathcal{M}_i)_{i \in \Index{I}}$ and $\Ideal{I}$ then 
	the reduced product $\mathcal{M}\coloneqq \prod_\Ideal{I} \mathcal{M}_i$ eliminates quantifiers relative to  $\Bool{I}{I}$ in the following two-sorted language $\mathcal{L}^+_\Phi$: 
	\[
	\mathcal{L}\cup \left\{\supp_{\phi(\bar{x})} : \phi(\bar{x})\in \Phi \right\} \cup \mathcal{L}_{\textrm{Bool}}\cup \{c_\theta : \theta\in \Phi \text{ is a sentence} \}.
	\]
\end{corollary}
\begin{proof}
It is enough to show that every function of the form $\supp_{\psi(\bar x)}$, where $\psi(\bar x)$ is any $\mathcal{L}$-formula, is equivalent to a quantifier-free $\mathcal{L}^+_\Phi$-formula.
    Fix an $\mathcal{L}$-formula $\psi(\bar x)$. By assumption, $\psi(\bar x)$ is equivalent to a Boolean combination of formulas from $\Phi$, say $\phi_1(\bar x),\dots,\phi_k(\bar x)$. By \Cref{P.boolfund}, $\supp_{\psi(\bar x)}$ is the same Boolean combination of $\supp_{\phi_i(\bar x)}$, $i\leq k$. Therefore, it is expressed without quantifiers in the language $\mathcal{L}^+_\Phi$, as desired.    
\end{proof}

A priori, the language $\mathcal{L}^{+}$ in the Feferman--Vaught theorem can be much larger than the language needed for quantifier elimination. However, it can also be already definable in our original language. Thus the following question naturally arises:

    \begin{question}
        For which indexed families of $\mathcal{L}$-structures $(\mathcal{M}_i)_{i \in \mathbb{I}}$ and ideals $\mathcal{I}$ on $\mathbb{I}$ does the reduced product $\prod_\Ideal{I} \mathcal{M}_i$ interpret the Boolean algebra $\Bool{I}{I}$ and the support functions $\supp_{\phi(\bar{x})}$ in the language $\mathcal{L}$? In other words, in which reduced products is the language $\mathcal{L}^+$ a definable expansion of the language $\mathcal{L}$?
    \end{question}
    We provide a positive answer to the question above if our reduced product satisfies two conditions: (1)~the relative support function is interpretable and (2)~there exists a fundamental set of $h$-formulas.

First, we need the following lemma:

\begin{lemma}\label{lemma:PhiSupport}
	Fix an index set $\Index{I}$, and ideal $\Ideal{I}$ on $\Index{I}$ and a family of $\mathcal{L}$-structures $(\mathcal{M}_i)_{i \in \Index{I}}$. Assume that the reduced product $\prod_\Ideal{I} \mathcal{M}_i$ interprets the Boolean algebra $\Bool{I}{I}$ and the relative support function:
	\[\supp_= \colon (a_i)_i,(b_i)_i \mapsto \big [\{ i \mid a_i=b_i\} \big]_\Ideal{I}.\]
	Then, for every satisfiable $h$-formula $\phi(\bar x)$, the $\phi$-support function $\supp_{\phi(\bar{x})}$ is interpretable in $\prod_\Ideal{I} \mathcal{M}_i$. 
\end{lemma}
\begin{proof} We may assume that $\bar x$ is a single variable, as the proof changes only notationally. 
	Fix $b\in \prod_\Ideal{I} \mathcal{M}_i$ satisfying $\phi(x)$. For every element $a \in \prod_\Ideal{I} \mathcal{M}_i$ and $S\in \Bool{I}{I}$, there is a unique element $a_S$ which coincides with $a$ on $S$ and with $b$ on $S^\complement$. Formally, $a_S$ is the unique element such that:
	\begin{itemize}
		\item $\supp_=(a_S,a) \supseteq S$,
		\item $\supp_=(a_S,b) \supseteq S^\complement$.
	\end{itemize}
	Then, $\phi(a_S)$ holds if and only if $\cM_i\models \phi(a_S(i))$ for all but $\cI$-many~$i$'s, and this is equivalent to $S\setminus \supp_{\phi(\bar{x})}(\bar{a}) 
	\in \cI$. 
	Therefore $\supp_{\phi(\bar{x})}(\bar{a})$  
	is the largest $S \in \Bool{I}{I} $ with respect to the inclusion such that $\phi(a_S)$ holds. Since $S$ does not depend on the choice of $b$, we get that the $\phi$-support function $\supp_{\phi(x)}$ is interpretable without parameters.
\end{proof}

We can now prove Theorem~\ref{C.QuestionDecomp}.

	\begin{theorem}\label{C.QuestionDecomp.1}		    Assume that a reduced product $\mathcal{M}:=\prod_\Ideal{I} \mathcal{M}_i$ interprets the relative support function $\supp_=$. Assume that there exists $\Phi$ a fundamental set of satisfiable $h$-formulas for  the pair $(\mathcal{M}_i)_{i \in \Index{I}}$ and $\Ideal{I}$. Then $\mathcal{M}$, as an $\mathcal{L}$-structure, already interprets all functions of the language $\calL^+$ and $\mathcal{M}$ eliminates quantifiers relative to $\Bool{I}{I}$ in the language $\calL_\Phi^+$.
	\end{theorem}
    \begin{proof} We fix an interpretation of the Boolean algebra $\Bool{I}{I}$ and of the map $\supp_=$.
    By \Cref{lemma:PhiSupport}, for any $h$-formula $\phi(x)$ the function $\supp_{\phi(\bar{x})}$ is definable in the language,  \[
	\mathcal{L} \cup \mathcal{L}_{\textrm{Bool}}\cup \left\{\supp_{=}\right\}.
	\]
    Relative quantifier elimination follows directly from \Cref{C.FV.reduced}.
    \end{proof}

If $\mathcal{M}$ is an atomless reduced product, then it admits automatically a fundamental set of $h$-formulas (i.e., \Cref{lemma:Palmgren}). 
    
For non-reduced products $\prod_{i\in\Index{I}} \cM_i$, the interpretability of $\supp_{\phi(\bar x)}$, for any $\mathcal{L}$-formula $\phi(x)$, follows directly from the interpretability of the relative support function, and there is therefore no need of a fundamental set of $h$-formulas.
This was observed by Macintyre and Derakhshan, at least in the case of products of connected rings. See \cite[Corollary 7.2]{DM22}.

    \section{Proof of Theorem \ref{T.A}}


As mentioned in the introduction, in \cite[\S 2.2]{de2023trivial} it was pointed out that 
    from the model-theoretic point of view, a \emph{morally} satisfactory proof that a theory recognizes coordinates would proceed by exhibiting a copy of $\cP(\bbN)/\cI$ as well as the projections $\pi_S$, for $S\in \Bool {\bbN} \cI$ inside every reduced product $\prod_{\Fin} \cM_n$ of models of $T$. In this section, we prove that a theory recognizes coordinates if and only if every reduced product interprets both the appropriate Boolean algebra along with the appropriate coordinate projections. Moreover, we prove that recognizing coordinates is equivalent to a simple characterization using the relative support function. This section provides much of substance of the proof of \Cref{T.A}. We gather the results together at the end of the section (see \Cref{T.eq}).
    
    \subsection{Interpreting supports implies recognizing coordinates}
    We show in this subsection that if all reduced products from a class of structures interpret the relative support function, then said class recognizes coordinates. The proof breaks nicely into two steps: (1) if a reduced product $\prod_\cI \mathcal{M}_i$ interprets the relative support function, then it also (uniformly) interprets the quotient structures $\prod_\cI \mathcal{M}_i\restriction S$ for any $S \in \Bool{I}{I}$ and (2) if all reduced products from a class of structures interpret the appropriate Boolean algebra and quotients (coherently), then the class recognizes coordinates. We prove Step (1) and then Step (2). 
    
    Consider a reduced product $\mathcal{M}:=\prod_\cI \mathcal{M}_i$ and nonzero $S \in \Bool{I}{I}$, with a representative $\tilde S\subseteq \Index{I}$. Then  $\Ideal{I}$ induces the ideal $\Ideal{I}_S:=\{\tilde S \cap J \mid J \in \Ideal{I}\}$ on~$\tilde S$ and $\mathcal{M} \restriction S$ denotes the reduced product $\prod_{\Ideal{I}_S} \mathcal{M}_i$. We let $a\rs S$ denote the natural projection of an element $a$ of $\mathcal{M}$ to $\mathcal{M} \restriction S$. 

    We first show Step (1), i.e., interpreting the relative support function implies that the relevant family of quotients is also interpretable. Before proving the following sequence of lemmas, we recall that the concept of a \emph{full class} was defined in Definition~\ref{Def.Full}. 
	\begin{lemma}\label{Claim.RestrictionInterpretable}Let $\fC$ be a full class, $\Index{I}$ be an index set, $\Ideal{I}$ be an ideal on $\Index{I}$, and $(\cM_i)_{i \in \Index{I}}$ a fixed family of $\mathcal{L}$-structures from $\fC$. Let $\mathcal{M} \coloneq \prod_\Ideal{I} \mathcal{M}_i$ be the reduced product. Assume that   
	$\mathcal{M}$ interprets the Boolean algebra $\Bool{I}{I}$ and the relative support function:
	\[\supp_= \colon (a_i)_i,(b_i)_i \mapsto \big[\{ i \mid a_i=b_i\}\big]_\Ideal{I}.\]
        Then both the restriction $\mathcal{M} \restriction S$, with its natural $\mathcal{L}$-structure, and the natural projection $\pi_S: \mathcal{M} \rightarrow \mathcal{M} \restriction S$ via  $a\mapsto a\restriction S$ are (uniformly) interpretable with parameter $S$. 
	\end{lemma}
	
	\begin{proof}
		We first recover the base set of $\mathcal{M} \restriction S$ as the quotient of $\mathcal{M}$ by the equivalence relation $\simeq$ given by: for all $a,b \in \mathcal{M}$, 
		\[a\simeq b  \Leftrightarrow \supp_=(a,b) \supseteq S.\]
		It remains to interpret the $\mathcal{L}$-structure of $\mathcal{M} \restriction S$. Let $R$ be a predicate in~$\mathcal{L}$. Since $\fC$ is full, $R$ is a satisfiable $h$-formula. By \Cref{lemma:PhiSupport}, the support function $\supp_{R(\bar{x})}$ is interpretable in $\mathcal{M}$. Then, for any $\bar{a}\restriction S$ in  $\left(\mathcal{M}\restriction S\right)^{|\bar{x}|}$, we have that $\mathcal{M}\restriction S \models R( \bar{a}\restriction S )$  if and only if $\supp_{R(\bar{x})} (\bar{a}) \supseteq S$.
		The interpretations of function symbols $f$ from $\mathcal{L}$ in $\mathcal{M}\restriction S$ are also straightforward: for any $\bar{a}\restriction{S}\in \left(\mathcal{M}\restriction S\right)^{|\bar{x}|}$, set 
        \begin{equation*}f(\bar{a}\restriction{S})=f(\bar{a})\restriction{S}. \qedhere
        \end{equation*}
	\end{proof}

    \begin{lemma}\label{lemma:equiv} Let $\fC$ be a full class, $\Index{I}$ be an index set, $\Ideal{I}$ be an ideal on $\Index{I}$, and $(\cM_i)_{i \in \Index{I}}$ a fixed family of $\mathcal{L}$-structures from $\fC$. Let $\mathcal{M} \coloneq \prod_\Ideal{I} \mathcal{M}_i$ be the reduced product. Then the following are equivalent: 
\begin{enumerate}
    \item The formula $\phi(x,x',y,y') := \supp_{=}(x,x') \subseteq \supp_{=}(y,y')$ is an  $\emptyset$-definable subset of $\mathcal{M}^{4}$. 
    \item The Boolean algebra $(\mathcal{P}(\Index{I})/\mathcal{I},\subseteq)$ and the relative support function $\supp_{=}: \mathcal{M}^{2} \to \mathcal{P}(\Index{I})/\mathcal{I}$ are interpretable in $\mathcal{M}$. 
\end{enumerate}
\end{lemma}

\begin{proof} $(1) \to (2)$. Suppose that $\phi(x,x',y,y')$ is a definable set. Consider the equivalence relation given by $E := \phi(x,x',y,y') \cap \phi(y,y',x,x')$. Then $\mathcal{M}^{2}/E$ is naturally isomorphic to $\Bool{I}{I}$. Let $\pi_{E}:\mathcal{M}^{2} \to \mathcal{M}^{2}/E$ be the quotient map. Notice that if $S, T \in \Bool{I}{I}$ then $S \subseteq T$ if and only if $\exists a,b,c,d$ in $\mathcal{M}$ such that $\pi_{E}(a,b) = S$, $\pi_{E}(c,d) = T$ and $\phi(a,b,c,d)$. Hence the relation $\subseteq$ is interpretable. Notice that $\pi_{E} = \supp_{=}$ which concludes this direction.   

$(2) \to (1)$. By construction, the set $\phi(x,x',y,y')$ is a definable subset of $\mathcal{M}^{4}$ in $\mathcal{M}^{\eq}$. Since there are no new definable subsets of $\mathcal{M}^{n}$ in $\mathcal{M}^{\eq}$, we see that $\phi(x,x',y,y')$ is already definable in $M$. 
\end{proof}

\begin{lemma}\label{L.RecognizingCoordinates} Let $\fC$ be a full class. Suppose that there exists an $\mathcal{L}$-formula $\phi(x,x',y,y')$ such that for any index set $\Index{I}$, ideal $\Ideal{I}$ on $\Index{I}$, and family $(\cM_i)_{i \in \Index{I}}$ of $\mathcal{L}$-structures from $\fC$, the reduced product $\mathcal{M} := \prod_{\Ideal{I}} \mathcal{M}_i$ satisfies,
\begin{equation*}
\mathcal{M} \models \phi(a,b,c,d) \Longleftrightarrow \supp_{=}(a,b) \subseteq \supp_{=}(c,d).
\end{equation*}
Then the class $\fC$ recognizes coordinates. 
\end{lemma}
\begin{proof}
Assume that  $\mathcal{M}= \prod_\Ideal{I}\mathcal{M}_i$ and $\mathcal{N}=\prod_\Ideal{J}\mathcal{N}_j$ are reduced products of structures from $\fC$ and let $\Phi\colon \mathcal{M}\to \mathcal{N}$ be an isomorphism. Then $\Phi$ extends to an isomorphism $\mathcal{M}^{\eq}\rightarrow \mathcal{N}^{\eq}$ and by \Cref{lemma:equiv}, it induces an isomorphism $\alpha: \Bool{I}{I} \rightarrow \Bool{J}{J}$. Then, we can name the parameter $S$, and $\Phi$ gives rise to an isomorphism between $\cM\rs S$ and $\cN \rs{\alpha(S)}$ by \Cref{Claim.RestrictionInterpretable}. 
    Hence the following diagram commutes
	
	\begin{center}
		\begin{tikzpicture}
			\matrix[row sep=1cm,column sep=1cm]
			{
				& & \node (M1) {$\cM$}; && &\node (M2) {$\cN$};&
				\\
				& & \node (Q1) {$\cM\rs S$}; &&& \node (Q2) {$\cN \rs {\alpha(S)}$} ;
				\\
			};
			\draw (M1) edge [->] node [above] {$\Phi$} (M2);
			\draw (Q1) edge [->] node [above] {$\Phi_S$} (Q2);
			\draw (M1) edge [->] node [left] {$\pi_S$} (Q1);
			\draw (M2) edge [->] node [right] {$\pi_{\alpha(S)}$} (Q2);
		\end{tikzpicture}
	\end{center}
	This shows that $\fC$ recognizes coordinates.
\end{proof}

\begin{remark}
Uniformity comes for free: assume otherwise, that for all formulas $\varphi(x, x', y, y')$, there is a reduced product $\prod_{\mathcal{I}_\varphi} \mathcal{M}_i$ where $\supp_=(x, x') \subseteq \supp_=(y, y')$ is not equivalent to $\varphi(x, x', y, y')$. Then $\supp_=(x, x') \subseteq \supp_=(y, y')$ is not $\emptyset$-definable in the reduced product $\prod_{\mathcal{I}} \mathcal{M}_i$, where $\Ideal{I} \coloneq\bigoplus_\varphi \mathcal{I}_\varphi$.
\end{remark}

\subsection{Recognizing coordinates implies interpretability of the support}
In this section, we prove the main implication (\ref{1.A} $\Rightarrow$ \ref{2.A}) of Theorem~\ref{T.A}. We show that if a class of structures recognizes coordinates, then all reduced products from said class interpret the relative support function. In fact we show that recognizing coordinates in $\mathbb{N}$ is enough. This is at the cost of some additional work and requires a forcing argument.  For full classes of structures see Definition~\ref{Def.Full}. 

\begin{proposition}\label{P.eq.countable}   Let  $\calL$ be a countable language and let 
	$\fC$ be a full class of $\mathcal{L}$-structures. If~$\fC$ recognizes coordinates in $\bbN$, then any reduced product  of structures from~$\fC$ defines the relative support function. 
\end{proposition}

The proof of this proposition requires two lemmas. In one of the lemmas, we use a Keisler--Shelah style argument to compare isomorphic ultrapowers.
Since there is a forcing extension in which two elementarily equivalent countable structures have no isomorphic ultrapower associated with an ultrafilter on $\bbN$ (\cite{shelah1992vive}), our proof involves set theory. More precisely, we need a slight modification of a standard theorem, first proven in \cite{platek1969eliminating}, asserting that if $\varphi$ is a projective statement provable in ZFC+CH, then $\varphi$ is provable in ZFC.


For  a reduced product $\cM:=\prod_\cI \cM_i $ consider the structure
\[
\cM^+:=\textstyle (\cM, \cP(\bbN)/\cI, \supp_=)
\]
in the language of $\cM$ expanded by a sort for $\cP(\bbN)/\cI$ equipped with its natural Boolean algebra structure. 

\begin{lemma}\label{L.absoluteness}
	Let  $\calL$ be a countable language and let 
	$\fC$ be a full class of $\mathcal{L}$-structures that is closed  under elementary equivalence. 	For a formula $\psi(\bar x)$ of the expanded language  (possibly with parameters from $\cM$), consider the statement,
	\begin{itemize}
		\item [$\theta_{\psi}$:] 	
		For every reduced product $\cM$ of structures in $\fC$  associated with an ideal $\cI$ on~$\bbN$, the set definable by $\psi$ in~$\cM^+$ is first-order definable in $\cM$ (with the same parameters). 
	\end{itemize}
	Then   ZFC+CH implies $\theta_{\psi}$ if and only if ZFC implies $\theta_{\psi}$. 
\end{lemma}

\begin{proof}Only the direct implication requires a proof. Suppose that ZFC and CH together imply $\theta_{\psi}$. Fix $\cM_n$, for $n\in \bbN$, in $\fC$ and an ideal $\cI$ on~$\bbN$.  
	Let $\kappa:=\max(\sup_{n\in \bbN} |\cM_n|, 2^{\aleph_0})$, and let~$\bbP$ be the forcing notion whose conditions are functions $p\colon \gamma_p\to \kappa$, where $\gamma_p$ is a countable ordinal, ordered by the reverse extension (hence $p\leq q$, i.e., $p$ is a stronger condition than $q$, if $\gamma_p\geq \gamma_q$ and $p\rs \gamma_q=q$). 
	
	This forcing notion is the L\'evy collapse of $\kappa$ to $\aleph_1$. We need two standard facts about forcing with $\bbP$. 
	First,  in the generic extension there is a surjection from $\aleph_1$ to $\kappa$ (\cite[Lemma 15.21]{jech2006set}).  
	Second,  $\bbP$ is $<\aleph_1$-closed (see the last two lines of the proof of \cite[Lemma 15.21]{jech2006set}). 
	Therefore $\bbP$ does not add any 	bounded subsets of $\kappa$ (\cite[Lemma 15.8]{jech2006set}) and in particular it does not add any new subsets of $\cP(\bbN)$, or any new elements of $\prod_n \cM_n$. 
	
	These two facts imply that CH holds in the forcing extension and that each one of the structures $\prod_{n} \cM_n$, $\cM$, and $\cM^+$ is unchanged.  Moreover, since $\fC$ is assumed to be closed under elementary equivalence, each $\cM_n$ still belongs to $\fC$ in the forcing extension. 
	Since ZFC+CH imply $\theta_{\psi}$, in the forcing extension  there exists a formula $\varphi$ such that for all $\bar a$ in $\cM$ we have that $\cM^+\models \psi(\bar a)$ if and only if $\cM\models \varphi(\bar a)$. Note that $\varphi$ belongs to the ground model, as it is a finite sequence of symbols in $\calL$.

	Since $\prod_n \cM_n$ is unchanged by forcing, and since $\cI$ is the same (since $\cP(\bbN)$ is unchanged by forcing, $\cI$ remains an ideal on $\bbN$ in the forcing extension),~$\cM$ is unchanged by forcing. Therefore  $\varphi$ and $\psi$ define the same set in the ground model, as required.  
\end{proof}
Lemma~\ref{L.absoluteness}  will  be applied in situations when $\psi$ is the formula 
\[\psi(x_1,x_2,y_1,y_2)\coloneq \supp_=(x_1,x_2) \subseteq \supp_=(y_1,y_2).\]

\begin{lemma}
	\label{L.eq.ctble} Let  $\calL$ be a countable language and let 
	$\fC$ be a class of $\mathcal{L}$-structures. 
	If every atomless reduced product $\prod_{i\in \bbN}\cM_i/\Ideal I$ of structures from~$\fC$ defines $\supp_=$, then every atomless reduced product $\prod_{i\in \Index I}\cM_i /\Ideal I$ of structures from $\fC$ defines  $\supp_=$. 
\end{lemma}

\begin{proof} 
	Assume the contrary, that there are an atomless ideal $\cI$ on some index set $\Index I$ and  reduced product $\cM:=\prod_{i\in \Index I} \cM_i/\Ideal I$ of structures from $\fC$ such that  no formula  $\varphi$ defines $\supp_=$ in $\cM$. 
	Let $\cF_4$ be the set of all $\calL$-formulas with at most four free variables. 
	
	For every formula $\varphi\in \cF_4$ fix $a^\varphi_1,a^\varphi_2,b^\varphi_1,b^\varphi_2$ in $\cM$ such that
	\begin{equation}
		\label{eq.def.supp-not=}
		\supp_=(a^\varphi_1,a^\varphi_2)\subseteq \supp_=(b^\varphi_1,b^\varphi_2)\qquad \not\Leftrightarrow \qquad 
		\cM\models \varphi(a^\varphi_1,a^\varphi_2,b^\varphi_1,b^\varphi_2). 
	\end{equation}		
	Consider the Boolean subalgebra $\cA$ of $\cP(\Index I)$ generated by the sets 
	\[
	\zeta^\varphi_\psi:=\{i\in \Index I\vert \cM_i\models \psi(a^\varphi_1,a^\varphi_2,b^\varphi_1,b^\varphi_2)\}
	\]
	and the sets 
	\[
	\supp_=(a^\varphi_1,a^\varphi_2)\setminus \supp_=(b^\varphi_1,b^\varphi_2),
	\]
	where $\varphi$ and $\psi$ range over $\cF_4$. 
	
	Let $\theta$ be a regular cardinal sufficiently large to have all of the objects mentioned so far---$\calL$, $\bbI$, $\cI$, $\langle \cM_i: i\in \bbI\rangle$, $\cF_4$, the assignment $\varphi\mapsto (a^\varphi_1,a^\varphi_2,b^\varphi_1,b^\varphi_2)$ belong to the set $H_\theta$ of all sets whose hereditary closure has cardinality less than $\theta$ (see e.g., \cite[Definition~I.13.27]{Ku:Set}). 

Let $N\preceq H_\theta$ be a countable elementary submodel that contains all of the objects listed below. Since the Boolean algebra $\cA$ is countable, it is included in $N$. Define
\begin{align*}
	\bbI_0&=N\cap \bbI, \\
	\cI_0&=\{A\cap \bbI_0: A\in N\cap \cI\}. 
\end{align*}
	By elementarity, $\cI_0$ is a proper atomless  ideal on~$\Index I_0$. 
		Let $\cM_0:=\prod_{i\in \Index I_0} \cM_i/\cI_0$. Since $\Index I_0$ is countable, our assumption implies that there is a formula $\varphi$ such that 
	\begin{equation*}
		\supp_=(a_1,a_2)\subseteq \supp_=(b_1,b_2)\qquad \Leftrightarrow \qquad 
		\cM_0\models \varphi(a_1,a_2,b_1,b_2)
	\end{equation*}
	for all $a_1,a_2,b_1,b_2$ in $\cM$. In particular, this holds for $a_1=a_1^\varphi\rs \Index I_0$ (i.e., $a_1=\pi_{\Index I_0}(a_1^\varphi)$), and analogously defined $a_2,b_1$, and $b_2$. 
	
	By our choice of $\Index I_0$ and $\cI_0$, we have that 
	\begin{equation}\label{eq.supp=}
		\supp_=(a_1,a_2)\subseteq \supp_=(b_1,b_2)\quad \Leftrightarrow\quad \supp_=(a_1^\varphi,a_2^\varphi)\subseteq \supp_=(b_1^\varphi,b_2^\varphi). 
	\end{equation}
	By the Feferman--Vaught theorem as stated in \cite[Theorem 6.3.2]{CK90}, there are a finite list $\psi_j$, for  $1\leq j\leq k$, of $\calL$-formulas and a formula $\Theta(z_1,\dots, z_k)$ in  the language of Boolean algebras such that 
	\[
	\cM\models \varphi(a^\varphi_1,a^\varphi_2,b^\varphi_1,b^\varphi_2)\quad\Leftrightarrow\quad\cP(\Index I)/\cI\models \Theta([\zeta^\varphi_{\psi_1}]_{\cI},\dots, [\zeta^\varphi_{\psi_k}]_{\cI})
	\]
	and also 
	\[
	\cM_0\models \varphi(a_1,a_2,b_1,b_2)\quad\Leftrightarrow\quad\cP(\Index I_0)/\cI_0\models \Theta([\zeta^\varphi_{\psi_1}\cap \Index I_0]_{\cI_0},\dots, [\zeta^\varphi_{\psi_k}\cap \Index I_0]_{\cI_0}). 
	\]
	Both $\Bool II$ and $\mathcal{P}(\Index{I}_0)/\Ideal{I}_0$ are atomless Boolean algebras and we may therefore assume $\Theta$ is quantifier-free\footnote{Atomless Boolean algebra admit quantifier elimination, e.g., \cite[Theorem 6.21]{poizat2000course}}.Thus the truth of $\Theta(z_1,\dots, z_k)$ depends only on knowing which of the intersections $\bigcap_{j=1}^k z_j^{\xi(j)}$ are trivial, where $\xi(j)\in \{\vartextvisiblespace ,\complement\}$, with $z_j^{\vartextvisiblespace } := z_j$. The  choice of $N\preceq H_\theta$, the set $\Index I_0$, and the ideal $\cI_0$ assure that  for all such $\xi$ we have $\bigcap_{j=1}^k (\zeta^\varphi_{\psi_k})^{\xi(j)}\in \cI$ if and only if $\bigcap_{j=1}^k (\zeta^\varphi_{\psi_k}\cap \Index I_0)^{\xi(j)}\in \cI_0$ (where the complement is evaluated with respect to~$\Index I_0$). 
	
	Therefore $	\cM\models \varphi(a^\varphi_1,a^\varphi_2,b^\varphi_1,b^\varphi_2)$ if and only if $\cM_0\models \varphi(a_1,a_2,b_1,b_2)$. Together with \eqref{eq.supp=} this implies the negation of \eqref{eq.def.supp-not=}; contradiction. 
\end{proof}

\begin{proof}[Proof of Proposition~\ref{P.eq.countable}]  
	Suppose that $\calL$ is a countable language and $\fC$ is a full class of $\mathcal{L}$-structures which recognizes coordinates in $\bbN$, and fix a reduced product $\mathcal{M}\coloneqq \prod_{i\in \Index I}\cM_i/\Ideal I$ of structures from $\fC$. 
	We need to prove that $\cM$  defines the relative support function $\supp_=$, quotients $\cM_{S}$, and quotient maps $\pi_S\colon \cM\to \cM\rs S$, for all $S\in \Bool II$. By Theorem~\ref{T.InterpretingSupport} we may assume that the ideal $\cI$ is atomless.
	By Lemma~\ref{L.eq.ctble} it suffices to prove this in the case when~$\cM=\prod_{i\in \bbN}\cM_i/\Ideal I$ is a reduced product associated with an atomless ideal $\Ideal I$ on~$\bbN$. 
	
	We use Beth's definability theorem. For $\supp_=$, it suffices to prove that for every reduced product $\cM$ of structures from $\fC$, the support function on $\cM$ is implicitly interpretable.  This means that for  every elementary extension $\mathcal{N}$ of $\mathcal{M}$, and any two `support' functions $s_1$ and $s_2$ on $\cN$ such that, 
	\[
	(\mathcal{N},s_1) \equiv (\mathcal{M},\supp_=) \equiv (\mathcal{N},s_2),  
	\]
	we have that for all $a,a',b,b' \in \mathcal{N}$, we have $s_1(a,a')\subseteq s_1(b,b') $ if and only if  $s_2(a,a')\subseteq s_2(b,b')$.
	
	Fix $\cM$, $\cN$, $s_1$, and $s_2$ and let $\cU$ be a nonprincipal ultrafilter on $\bbN$. If the Continuum Hypothesis (CH) holds and all $\cM_i$ as well as $\cN$ have cardinality no greater than $2^{\aleph_0}$,  then the ultrapowers  
	$\prod_\mathcal{U}(\mathcal{N},s_1)$,  $\prod_\mathcal{U}(\mathcal{M},\supp_=)$, and $\prod_\mathcal{U}(\mathcal{N},s_2)$ have cardinality $2^{\aleph_0}$ and are $\aleph_1$-saturated.  Since they are also elementarily equivalent, we have  two isomorphisms $\sigma: \prod_\mathcal{U}\mathcal{M} \rightarrow \prod_\mathcal{U}\mathcal{N} $ and
	$\tau: \prod_\mathcal{U}\mathcal{M} \rightarrow \prod_\mathcal{U}\mathcal{N}$
	that moreover respect $\supp_=$, $s_1$,  and $s_2$ so that the following holds,
	\begin{equation}\label{eq.eq.countable}
		\textstyle (\prod_\mathcal{U}\mathcal{N},s_1) \overset{\sigma}{\simeq} (\prod_\mathcal{U}\mathcal{M},\supp_=)\overset{\tau}{\simeq}(\prod_\mathcal{U}\mathcal{N},s_2).
	\end{equation}
	We will complete the proof under this assumption and then show how it can be removed.

	By \Cref{P.ultraproduct}, there is an ideal $\Ideal{J}$ on the index set $\Index J=\bbN\times \bbN$ and an enumeration $\mathcal{M}_{ij}$, for $(i,j)\in \bbN\times \Index \bbN$ (with many repetitions), such that $\prod_\mathcal{U}(\mathcal{M}, \supp_=) = (\prod_\Ideal{J}\mathcal{M}_{ij},\supp_=)$. 
	
	Let $\rho= \tau^{-1}\circ \sigma$. This map is an automorphism of $\prod_\Ideal{J}\cM_{ij}$. We need to show that this is an automorphism of $(\prod_\Ideal{J}\cM_{ij}, \supp_=)$.

	Since $\fC$ recognizes coordinates, there is an automorphism $\alpha: \Bool{J}{J} \rightarrow \Bool{J}{J}$ such that for all $S\in \Bool{J}{J}$  the following diagram commutes (see Definition~\ref{D.recognizes-coordinates}):
	
	\[\begin{tikzcd}
		{(\mathcal{N},s_1)} &&& {(\mathcal{N},s_2)} \\
		{\prod_\mathcal{U}(\mathcal{N},s_1)} & {(\prod_\Ideal{J}\cM_{ij},\supp_=)} & {(\prod_\Ideal{J}\cM_{ij},\supp_=)} & {\prod_\mathcal{U}(\mathcal{N},s_2)} \\
		& {\prod_\Ideal{J}\cM_{ij}\upharpoonright S} & {\prod_\Ideal{J}\cM_{ij}\upharpoonright \alpha(S)} & 
		\arrow["\diag", from=1-1, to=2-1]
		\arrow["\diag"', from=1-4, to=2-4]
		\arrow["\sigma"', from=2-2, to=2-1]
		\arrow["\rho", from=2-2, to=2-3]
		\arrow["{\pi_S}", from=2-2, to=3-2]
		\arrow["\tau", from=2-3, to=2-4]
		\arrow["{\pi_{\alpha(S)}}"', from=2-3, to=3-3]
		\arrow["{{\rho_{S}}}", dashed, from=3-2, to=3-3]
	\end{tikzcd}\]
	
	To show that $(\rho,\alpha)$ is an automorphism of $(\prod_\Ideal{J} \cM_{ij}, \supp_=)$, we only need to show that it respects the support function. Take a pair $(g,g')$ of elements of $\prod_\mathcal{J} \cM_{ij}$ and set $S=\supp_=(g,g')$. Then $\pi_S(g)=\pi_S(g')$ and since the diagram commutes, we have,
	\[
	\pi_{\alpha(S)}(\rho(g))= \pi_{\alpha(S)}(\rho(g')). 
	\]
	It follows that $  \alpha(\supp_=(g,g'))  \subseteq \supp_=(\rho(g),\rho(g'))$.   
	By considering the isomorphism $\rho^{-1}$ and $\alpha^{-1}$ instead of $\rho$ and $\alpha$, by the analogous argument we have that, 
	\[ \alpha^{-1}(\supp_=(\rho(g),\rho(g')) \subseteq
	\supp_=(\rho^{-1}(\rho(g)),\rho^{-1}(\rho(g'))),
	\]
	which gives the other inclusion: 
	\[
	\alpha(\supp_=(g,g')) \subseteq \supp_=(\rho(g),\rho(g')).
	\]
	Therefore, for all $g,g'$, we have $\alpha(\supp_=(g,g')) = \supp_=(\rho(g),\rho(g'))$ and  $(\rho,\alpha)$ is an automorphism of $(\prod_\Ideal{J}\cM_{ij},\supp_=)$.

	The rest of the proof is straightforward: let $a,a',b,b' \in \mathcal{N}$, such that $s_1(a,a') \subseteq s_1(b,b')$. Then in the ultraproduct, we have, 
	\[\prod_\mathcal{U}(\mathcal{N},s_1) \models s_1(\diag(a),\diag(a')) \subseteq s_1(\diag(b),\diag(b'))\]
	where $\diag$ is the canonical embedding of $\mathcal{N}$ in the ultrapower. Since $\Id =\sigma^{-1}\circ \rho \circ \tau$ is a composition of isomorphisms preserving the support, we also have,
	\[
	\prod_\mathcal{U}(\mathcal{N},s_2)\models s_2(\diag(a),\diag(a')) \subseteq s_2(\diag(b),\diag(b')),
	\]
	and therefore,
	\[
	(\mathcal{N},s_2) \models s_2(a,a') \subseteq s_2(b,b').
	\]
	By the Beth definability theorem, this proves that the relation $\supp_=(x,x')\subseteq \supp_=(y,y')$ is first-order definable in $\cM$ if CH holds and all $\cM_n$ and $\cN_n$ have cardinality not greater than the continuum.

	The Feferman--Vaught  theorem implies that if $\cM_i\equiv \cN_i$ for all $i$ and $\supp_=$ is first-order definable in  $\prod_\cI\cM_i$, then $\supp_=$ is first-order definable in $\prod_\cI \cN_i$, by the same definition.\footnote{The proof of this assertion is very similar to the use of the Feferman--Vaught theorem at the end of the proof of Lemma~\ref{L.eq.ctble} and is therefore omitted.} Therefore we may apply 
	Lemma~\ref{L.absoluteness} to the closure of $\fC$ (under elementary equivalence). By  applying it to $\psi$ chosen to be  $\supp_=(x_1,x_2)\subseteq \supp_=(y_1,y_2)$,   this conclusion follows already in ZFC. 		
	Therefore $\supp_=$ is definable in reduced products of structures in~$\fC$ associated with ideals on~$\bbN$. As pointed out at the beginning of the proof, this implies the general case. 
\end{proof}	

\subsection{Proof of the main result}\label{S.Main}
%
%

If $\fC$ is a class of $\mathcal{L}$-structures, we let $\Th(\fC)$ denotes the common theory of the structures in $\fC$, that is, the set of all $\calL$-sentences that are true in all structures in $\fC$. The following theorem implies Theorem~\ref{T.A}.

\begin{theorem}\label{T.eq}
	For a  countable language $\calL$ and a full class $\fC$ of $\mathcal{L}$-structures, the following are equivalent. 
	\begin{enumerate}
		
		\item \label{1.T.eq} $\fC$ recognizes coordinates.
		\item \label{1.countable.T.eq} $\fC$ recognizes coordinates in $\bbN$. 
		\item \label{2b.T.eq} For every ideal $\Ideal{I}$ on an index set $\Index I$, a reduced product $\cM:=\prod_{\Ideal I} \cM_i$ of structures from $\fC$ interprets both the Boolean algebra $\cP(\Index I)/\cI$  and the system of  quotient structures $\cM\rs S$ and quotient maps $\pi_S$ (pa\-ra\-me\-tri\-zed by $S\in \cP(\Index I)/\cI$).
		\item \label{2.T.eq} For every ideal $\Ideal{I}$ on an index set $\Index I$, a reduced product $\cM:=\prod_{\Ideal I} \cM_i$ of structures from $\fC$ interprets the Boolean algebra $(\Bool{I}{I}, \subseteq)$ and the relative support function 
		\[
		\supp_=: \mathcal{M}^2\rightarrow \mathcal{P}(\mathbb{I})/\mathcal{I}, \ ((a_i)_i,(a_i')_i) \mapsto \big[\{i \mid a_i = a_i'\}\big]_\mathcal{I}.
		\]
		\item \label{3.T.eq} The formula $x=x' \rightarrow y=y'$ is equivalent to an $h$-formula in the common theory $\Th(\fC)$ of $\fC$.
	\end{enumerate}
	\begin{enumerate} 
    \setcounter{enumi}{5}
		\item \label{3.b.T.eq} The formula $x=z \rightarrow y=z$ is equivalent to an $h$-formula in the common theory $\Th(\fC)$ of $\fC$.
	\end{enumerate}
	In case when $\calL$ is not necessarily countable, assertions \eqref{1.T.eq} and \eqref{2b.T.eq}--\eqref{3.T.eq} are equivalent. 
\end{theorem}

\begin{proof} The implication \eqref{1.T.eq} $\Rightarrow$ \eqref{1.countable.T.eq} is trivial. \eqref{1.countable.T.eq} $\Rightarrow$ \eqref{2.T.eq} is the main implication (see Proposition~\ref{P.eq.countable}); it is also the only implication that uses the assumption that $\calL$ is countable.  The equivalence between \eqref{2.T.eq} and  \eqref{3.T.eq} is \Cref{T.InterpretingSupport}. As all structures in a full class $\fC$ have at least three elements, the equivalence of \eqref{3.T.eq} and \eqref{3.b.T.eq} is \Cref{prop:interpretingAllZsupport}. The implication \eqref{2.T.eq} $\Rightarrow$ \eqref{2b.T.eq} follows from \Cref{Claim.RestrictionInterpretable}. Finally, \eqref{2b.T.eq} $\Rightarrow$ \eqref{1.T.eq}  is \Cref{lemma:equiv} and \Cref{L.RecognizingCoordinates}. 

\begin{center}
	\begin{minipage}{0.9\linewidth}
	\small
	\xymatrix{
		\eqref{1.T.eq} \ar@{=>}[r]^{\text{triv.}} & \eqref{1.countable.T.eq} \ar@{=>}[d]^{\ref{P.eq.countable}} &\\
		\eqref{2b.T.eq} \ar@{=>}[u]^{\ref{L.RecognizingCoordinates}} &\eqref{2.T.eq}\ar@{=>}[l]^{\ref{Claim.RestrictionInterpretable}} \ar@{<=>}[r]_{\ref{T.InterpretingSupport}}&  \eqref{3.T.eq} \ar@{<=>}[r]_{\ref{prop:interpretingAllZsupport}}&  \eqref{3.b.T.eq} }
	\end{minipage}
\end{center}
	
	In case when $\calL$ is not necessarily countable it suffices to prove that \eqref{1.T.eq} implies \eqref{2.T.eq}. The proof follows the proof of Proposition~\ref{P.eq.countable} closely. The only difference is that instead of $\aleph_1$-saturated ultrapowers $\prod_\mathcal{U}(\mathcal{N},s_1)$,  $\prod_\mathcal{U}(\mathcal{M},\supp_=)$, and $\prod_\mathcal{U}(\mathcal{N},s_2)$ one needs to choose saturated elementary extensions of $(\cN,s_1)$, $(\cM,\supp_=)$, and $(\cN,s_2)$\footnote{Such models need not exist, but another absoluteness argument shows that there is no harm in assuming they do. See \cite[\S 3]{halevi2023saturated} or \cite[\S 8.1]{de2024saturation} for more details.}  of the same cardinality; the remaining part of the proof is identical. 
\end{proof}

It is not enough to just interpret the quotient Boolean algebra. As a limiting example, we consider the group of 8 elements corresponding to the basis of the Hamiltonians, commonly referred to as $Q_8$. We show that the reduced product $\prod_{\Fin} Q_8$ interprets the quotient $\Bool\bbN\Fin$, but the relative support function $\supp_=$ is not interpretable. Consequently, we prove that $Q_8$ does not recognize coordinates (see Example~\ref{C.Q8}).

By the Keisler--Shelah theorem, the class of all models of $\Th(\fC)$ is the class of all elementary submodels of ultraproducts of models in $\fC$. 

\begin{corollary}\label{C.Th(C)}
	
	A full class $\fC$ of $\mathcal{L}$-structures recognizes coordinates if and only if $\Th(\fC)$ recognizes coordinates. 
\end{corollary}

\begin{proof}
	This is immediate, as by \Cref{T.eq} \eqref{3.T.eq}, recognizing coordinates is a property of $\Th(\fC)$.
\end{proof}

Another amusing consequence of Theorem~\ref{T.eq} is the following. Suppose that a class $\fC$  recognizes coordinates and $\cM$ is an (atomless) reduced product of structures in $\fC$. Then every model  $\cN$ of $\Th(\cM)$ interprets an (atomless) Boolean algebra and a system of quotients $\pi_S\colon \cN\to \cN_S$ for $S$ ranging in this Boolean algebra that behave as the maps in the diagram in Definition~\ref{D.recognizes-coordinates}. In other words, every model of $\Th(\cM)$ `thinks' that it is an (atomless)  reduced product.  
\section{Groups recognizing coordinates}\label{S.GroupsRecognizing}

In this section, we shift our attention to classes of groups. We first prove some basic preparatory results to work in this setting. We then show that a large variety of familiar classes of groups recognize coordinates. In later sections, we will consider classes of groups which do not recognize coordinates. 

\subsection{General criteria for recognizing coordinates}\label{S.Recognizing}

We begin by fixing some notation for this section. Suppose that $G$ is a group and $a$ is an element of $G$. We will always use the symbol $\cdot$ for group multiplication and the symbol $e$ for the identity element of $G$. We let $a^{G}$ denote the conjugacy class of $a$ inside the group $G$. We remark that if $z$ is a variable then $x \in z^{G}$ is shorthand for $\exists y (yzy^{-1} = x)$. If $S \subseteq G$, we let $C_{G}(S)$ denote the centralizer of $S$, i.e.
\[
C_{G}(S) := \{ g \in G \mid gs = sg \text{ for all } s \in S \}.
\]
If $S$ is a definable subset of $G$, then so is $C_{G}(S)$. We let $Z(G)$ denote the center of $G$. If $n$ is a natural number, then $a^{n}$ denotes the element $a$ to the $n$-th power while if $g \in G$, then $a^{g}$ denotes the conjugate of $a$ by $g$, i.e. $a^{g} := gag^{-1}$.

It will be convenient to use the set relation~$\subseteq$ when referring to elements of $\Bool{I}{I}$. If $\prod_\Ideal{I} G_i$ is a reduced product of groups, and $a$ is an element of $\prod_\Ideal{I} G_i$ with representative $(a_i)_i$, we define the support of $a$ as,
\[
\supp(a) \coloneqq \big[\{ i \in  \Index{I} \mid a_i\neq e\}\big]_\Ideal{I}. 
\]
In this context, \Cref{T.eq} can be rewritten as follows: 
\begin{theorem}[Groups recognizing coordinates]\label{T.eq.groups}
    Let $\fC$  be a full class of groups in a language $\{\cdot, {}^{-1}, e, \dots\}$ with potentially additional structure. The following are equivalent:
    \begin{itemize}
        \item $\fC$ recognizes coordinates.
        \item The relation $\supp(y) \subseteq \supp(x)$ is definable in all reduced products of structures from $\fC$.\footnote{In order to avoid omitting a triviality that together with other trivialities may add up to an impasse, we should point out that $x=e\rightarrow y=e$ corresponds to the formula $\supp(y)\subseteq \supp(x)$.}
        \item The formula $x=e \rightarrow y=e$ is equivalent to an $h$-formula in the common theory $\Th(\fC)$ of $\fC$. 
    \end{itemize}
 
\end{theorem}
We first note that the formula $x=z \rightarrow y=z$ is equivalent to an $h$-formula if and only if $xz^{-1}=e \rightarrow yz^{-1}=e$ is equivalent to an $h$-formula. The theorem above follows from \Cref{T.eq} and the following lemma:

\begin{lemma}\label{L.RecognizingCoordinatesGroup}
	Let $\fC$ be a class of groups. Fix a reduced product $G=\prod_\Ideal{I} G_i$ of groups from $\fC$. The following are equivalent:
	
	\begin{enumerate}
		\item \label{RC1} The relation, 
		\[
		\supp(a)\subseteq \supp(b).
		\]
		is first-order definable in $G$.
		\item \label{RC2} The support function,
		\[\supp: (a_i)_i \mapsto \big[\{ i \mid a_i \neq e\}\big]_\Ideal{I},\]
		is first-order interpretable in $G$.
		\item \label{RC3} the relative support function,
		\[\supp_= \colon (a_i)_i,(b_i)_i \mapsto \big[\{ i \mid a_i=b_i\}\big]_\Ideal{I},\]
		is first-order interpretable in $G$.
	\end{enumerate} 
	
\end{lemma}
\begin{proof}
	
	$(\ref{RC1} \rightarrow \ref{RC2})$ Assume the relation $a \sqsubseteq b$ defined by $ \supp(a)\subseteq \supp(b)$ for $a,b\in G $ is definable. This relation is a pre-order, and we denote by $\sim$ the associated equivalence relation: for $a,b\in G$, $a\sim b$ holds if and only if $a$ and $b$ have the same support. The ordered quotient $(G/\sim, \sqsubseteq)$ is therefore isomorphic to $(\Bool{I}{I},\subseteq)$ and the natural projection is, after identification, equal to the support function $\supp$.\\
	$(\ref{RC2} \rightarrow \ref{RC3})$
	This implication is immediate: for all $a,b\in G$, we have, 
    \[\supp_=(a,b) = \supp(a\cdot b^{-1})^\complement. \]
    $(\ref{RC3} \rightarrow \ref{RC1})$ This is immediate, since for all $a,b\in G$, $\supp(a)\subseteq \supp(b)$ holds if and only if $\supp_=(b,e) \subseteq \supp_=(a,e)$.
\end{proof}
    %
	

Now that we have finished the preliminaries of this section, we are able to give our first classes of groups which recognize coordinates (very exciting). We will use these first examples to show that more familiar classes of groups recognizes coordinates. 

\begin{theorem}\label{T.RecognizingCoordinates}
	\begin{enumerate}
		\item \label{2.RecognizingCoordinates} Let $\mathcal{G}$ be the class of all groups $G$ such that for every $a \in G\backslash \{e\}$, $C_G(a^G)=\{e\}$. Then $\mathcal{G}$ recognizes coordinates.
        \item \label{1.RecognizingCoordinates} For a prime $p$, let $\mathcal{G}_{p}$ be the class of all groups $G$ which have elements of order $p$ and for every $a\in G$ of order $p$, $C_G(a^G)=\{e\}$. Then $\mathcal{G}_{p}$ recognizes coordinates. 
 
	\end{enumerate}
\end{theorem}

\begin{proof}
(\ref{2.RecognizingCoordinates}) We show that $x=e \rightarrow y=e$ is equivalent to the $h$-formula:
     \begin{itemize}
		\item [(*)]	$(\forall z) ~ \left(y\in C(z^G)\rightarrow  x\in C(z^G)\right)$.
        \end{itemize}
To see this, let $\varphi(x,z)$ be the formula $(\forall t)(xtzt^{-1}=tzt^{-1}x)$, or equivalently, $x\in C_G(z^G)$. 
This is an $h$-formula and since $\varphi(x,e)$ is true in every group, $(\forall z)(\varphi(y,z)\lthen \varphi(x,z))$ is also an $h$-formula. This formula is clearly equivalent to (*). The assumption that $C_G(a^G)=\{e\}$ for every $G\in \cG$ and every $a\neq e$ in $G$ implies that  (*) is equivalent to $x=e\to y=e$. 

         The conclusion follows by \Cref{T.eq.groups} \\

    (\ref{1.RecognizingCoordinates}) Fix a prime $p$. For a group in $\mathcal{G}_p$, we have that $x=e \rightarrow y=e$ is equivalent to 
    \[\tag{**} \label{Eq.1.re}	(\forall z) ~ \left( [x\in C(z^G)\wedge z^p=e] \rightarrow y\in C(z^G) \right).\]
        Indeed, notice that for all $z$ of order $p$ or $1$,  if $x= e$, then $x\in C(z^G)$. If  $x\neq e$, then $x\in C(z^G)$ only if $z=e$. Therefore, $x = e \to y = e$ holds if and only if, 
        \[\{z \mid z^p=e , x\in C(z^G) \} \subseteq \{z \mid z^p=e , y\in C(z^G) \}, \]
        and the equivalence with (\ref{Eq.1.re}) is clear. We observe that (\ref{Eq.1.re}) is an $h$-formula as 
        \[(\exists z) ~ \left( x\in C(z^G)\wedge z^p=e \right)\]
        is always satisfied. The statement holds by \Cref{T.eq.groups}.
   \end{proof}

 A group $G$ satisfies \eqref{2.RecognizingCoordinates} of Theorem~\ref{T.RecognizingCoordinates} if and only if for any nontrivial normal subgroup $H$ of $G$, $C_G(H)=\{e\}$. This is equivalent to being a domain in the sense of \cite[Definition~1]{baumslag1999algebraic}. 
	
	In \cite[Theorem A]{kvaschuk2005algebraic}  it was proved that the first-order theory of a finite product of domains $\prod_{i<n} G_i$ ``determines $n$" and interprets the theories of all $G_i$, for $i<n$. Together with an easy argument using the  Keisler--Shelah theorem, this is equivalent to the assertion that  domains recognize coordinates in finite products. Our methods give the following far-reaching generalization of that result (we will not need to use the definition of an acylindrically hyperbolic group, it suffices to say that this class is extensive and well-studied,  see \cite{osin2016acylindrically}). 
	
	\begin{corollary}\label{C.Montse}
		The class of all domains recognizes coordinates. 
		In particular, the class of all acylindrically hyperbolic groups without a non-trivial finite normal subgroup recognizes coordinates. 
		\end{corollary}
		
		\begin{proof}
						By Theorem~\ref{T.RecognizingCoordinates} \eqref{2.RecognizingCoordinates},  the class of all domains recognizes coordinates. By \cite[Theorem 1.3]{cassella2025first-order}, every acylindrically hyperbolic group without a non-trivial finite normal subgroup is a domain. 
		\end{proof}

\begin{proposition}\label{P.2}
	Suppose that $T$ is a theory of (nontrivial) groups and that for some $m\geq 1$ and  $n\geq 1$, $T$ includes an axiom asserting the following: 
	\begin{enumerate}
		\item \label{1.P.1}For all $a,b$ in $G\setminus\{e\}$ of order dividing $m$ there is some $c$ in $a^G$ such that $bc\neq cb$. 
		\item\label{2.P.1} Every $c$ in $G$ is the product of at most $n$ elements, each one of which has order dividing $m$.   
	\end{enumerate}
	Then $T$ recognizes coordinates
\end{proposition}

\begin{proof}
We proceed as in the proof of Theorem~\ref{T.RecognizingCoordinates}.
We will show that modulo the theory $T$, the formula  $ x=e \rightarrow y=e $ is equivalent to the following (let $G_{[m]}$ denote the set of elements  of $G$ whose order divides $m$). 
\begin{align*}
    \tag{*} \label{Eq.1.P.2}	(\exists a_1,\dots,a_n \in G_{[m]})& \  [y =a_1 \cdot \dots \cdot a_n   \\ \ \wedge  ( & \forall a \in G_{[m]}) ~ ( x\in C(a^G)\cap G_{[m]} \rightarrow \bigwedge_{i\leq n}a_i\in C(a^G) )].  
\end{align*}
Assume that (\ref{Eq.1.P.2}) holds, witnessed by $y =a_1 \cdot \dots \cdot a_n$ with $a_i\in G_{[m]}$, $i\leq n$. Moreover assume $x=e$. Then for every $a\in G_{[m]} \setminus\{e\}$ we have $x\in C(a^G)\cap G_{[m]}$. By \eqref{Eq.1.P.2}, we have for $i\leq n$ that $a_i\in C(a^G)\cap G_{[m]}$.  Since \eqref{1.P.1} implies that this set is equal to $\{e\}$, it follows that $a_i=e$ for all $i$ and therefore $y=a_1\cdot ... \cdot a_n=e$.

Conversely, assume that $x=e \rightarrow y=e $ holds. If $x=y=e$, then (\ref{Eq.1.P.2}) trivially holds, witnessed by $a_1= \dots= a_n=e$. Assume $x\neq e$. By \eqref{2.P.1} there are $a_1,\dots, a_n\in G_{[m]}$ such that $y=a_1\cdot {}\dots{} \cdot a_n$.  We check that the consequent of \eqref{Eq.1.P.2} holds for every $a\in G_{[m]}$.  First,~\eqref{1.P.1} implies that $x\notin C(a^G)\cap G_{[m]}$ for every $a\neq e$.  If  $a=e$, then $a_i\in C(a^G)$ for all $i\leq n$, and  therefore, (\ref{Eq.1.P.2}) also holds.
\end{proof}

\subsection{Classes of groups that recognize coordinates}\label{S.GroupsRecognizingCoordinates}

The following theorem gathers classes of groups which we know recognize coordinates.  For the definition of perfect groups and quasisimple groups, see \S\ref{S.Perfect}. 
For an integer $n\geq 2$, we let $\Sym_n$ denote the symmetric group of $n$ elements.

\begin{theorem} \label{T.RC}
	Each of the following classes of groups recognizes coordinates. 
	\begin{enumerate}[label=(\alph*)]
		\item \label{1.C} The class of all non-abelian simple groups.

        \item \label{2.Cplus} The class of all $\Sym_n$, $n\geq 3$. 
		\item \label{3.C} The class of all dihedral groups $D_{2n+1}$, for $n\geq 1$. 
		
		\item \label{4.1.C} The class of all quasisimple groups of commutator width $\leq n$, for any fixed $n$. 
		
		\item \label{4.2.C} The class of all finite  quasisimple groups.

		\item \label{4'.C}
		The class of groups of the form $\SL(n,F)$, for $2\leq n$ and  $|F|\geq 4$
        .
		
		\item \label{5.C} The class of groups which are free products of two nontrivial groups at least one of which has cardinality at least 3. 
		
		\item \label{6.C} $\{G\}$ where $G$ is any nontrivial free product. 
		
		\item \label{7.C} The class of all graph products $\Gamma\cG$ such that the complement graph $\bar\Gamma$ is connected and $|G_v|\geq 3$ for at least one $v\in V(\Gamma)$ (see \S\ref{S.GraphProducts}). 
		
		\item \label{8.C} The class of all domains (in the sense of \cite{baumslag1999algebraic}) and in particular 	the class of all acylindrically hyperbolic group (\cite{osin2016acylindrically}) without a non-trivial finite normal subgroup. 
	\end{enumerate}
\end{theorem}

\begin{proof} 
	\ref{8.C}  is Corollary~\ref{C.Montse}, included in the present list for completeness. 
	
	We will use Theorem~\ref{T.RecognizingCoordinates}. 
	
	\ref{1.C} It suffices to prove that every non-abelian simple group satisfies \eqref{2.RecognizingCoordinates} of Theorem~\ref{T.RecognizingCoordinates}.  Suppose $G$ is a non-abelian simple group and $a\in G\setminus \{e\}$. Since the subgroup generated by $a^G$ is obviously normal, it is equal to $G$ and in particular $C_G(a^G)=\{e\}$.

    
	
	\ref{2.Cplus} This is Corollary~\ref{C.SymAllExceptSix}. We can remark here that the class $\{\Sym_n\mid n\geq 4\}$ recognizes coordinates by \eqref{2.RecognizingCoordinates} of Theorem~\ref{T.RecognizingCoordinates} since that, if $n\geq 4$, every nontrivial $a\in \Sym_n$ has the property that the centralizer of $a^{\Sym_n}$ is trivial. We also note that the group $\Sym_3$ satisfies  \eqref{1.RecognizingCoordinates} of Theorem~\ref{T.RecognizingCoordinates} with $p=2$ and therefore recognize coordinates. 
	This is because every element of $\Sym_3$ of order $2$ is a transposition, the transpositions generate $\Sym_n$, and they are pairwise conjugate.\footnote{Note that \eqref{2.RecognizingCoordinates} of Theorem~\ref{T.RecognizingCoordinates} fails in $\Sym_3$, because the conjugacy class of a 3-cycle is a proper abelian subgroup.}
	
	\ref{3.C} We  verify \eqref{1.RecognizingCoordinates} of Theorem~\ref{T.RecognizingCoordinates} holds for $D:=D_{2n+1}$ with $p=2$. We identify $D$ with the group of symmetries of a regular $2n+1$-gon. Every element of order $2$ is a reflection and since all reflections are conjugate, we verify easily that $C_{D}(r^{D})=\{e \}$ for all nontrivial reflection $r$. Since every rotation is a composition of two reflections, the conclusion follows.


	\ref{4.1.C} This is Corollary~\ref{C.SL(2.5)} below.

	\ref{4.2.C} 
	By \cite[Corollary 2]{liebeck2011commutators}, in every finite quasisimple group the commutator width is at most 2 and the result follows from Corollary~\ref{C.SL(2.5)} below. 

	\ref{4'.C} is Corollary~\ref{C.SL(n,F)} (a special case of Proposition~\ref{P.Perfect}). 
	
	\ref{5.C} This is Proposition~\ref{P.FreeProduct} below. 
	
	\ref{6.C} The only nontrivial product not covered by \ref{5.C} is $(\bbZ/2\bbZ)*(\bbZ/2\bbZ)$. It recognizes coordinates by Proposition~\ref{P.Z/2Z*Z/2Z}. 
	
	\ref{7.C} This is Theorem~\ref{T.GraphProduct}.
\end{proof}

\subsection{Perfect groups}\label{S.Perfect}
We remark that up to this point in the paper, it is not immediately obvious if there exists groups with nontrivial centers which recognize coordinates. We will see in a later section that if a group admits a \emph{nontrivial homomorphism to its center}, then it cannot recognize coordinates. Yet, not every group with a nontrivial center admits a nontrivial homomorphism into its center. An example is provided by the so-called \emph{perfect} groups. These are the groups $G$ with commutator subgroup $G' = [G,G]$ equal to $G$ itself. A homomorphism from $G$ into some group has an abelian range if and only if its kernel includes the commutator subgroup. In particular,  since for $H\triangleleft G$ we have that $G/H$ is abelian if and only if $H\supseteq G'$, a group is perfect if and only if it has no nontrivial abelian quotients. Hence a perfect group~$G$ with nontrivial center does not admit any nontrivial homomorphism into its center. 

The \emph{commutator width} of a perfect group is the minimal $n$ such that every element of $G$ is the product of $n$ commutators. There are perfect groups with arbitrarily large commutator width, see  \cite[\S 2]{nikolov2004commutator}. 
In an early draft of this paper we asked whether there is a perfect, indecomposable, group $G$ with infinite commutator width and nontrivial center, and whether such $G$ can be chosen so that a group elementarily equivalent to $G$ admits a nontrivial homomorphism into its center. 
Answering part of this question, Forte Shinko proved that $C_{\textrm{Homeo}^+(\bbR)}(\bbZ)$, the centralizer of the natural copy of  $\bbZ$ inside the group of orientation-preserving homeomorphisms of $\bbR$,   is a perfect indecomposable group with infinite commutator width and nontrivial center (\cite{shinko2025recognzing}). A positive answer to the other portion of this question would give a perfect, indecomposable group that does not recognize coordinates.

Note that `$G$ is a perfect group of commutator width $\leq  n$' is axiomatizable, but `$G$ is a perfect group' is not. Indeed, if $(G_n)_n$ is a sequence of perfect groups with arbitrarily large commutator width, then any non-principle ultraproduct $\prod_\cU G_n$ is imperfect.

\begin{proposition}\label{P.Perfect}
	For $m\geq 2$ let $\fC_m$ be the class of all perfect groups $G$ of commutator width $\leq  m$ that satisfy the following condition. 
	\begin{itemize}
		\item [$(\dagger)$] If $H\trianglelefteq G$ is nonabelian, then $C_G(H)=Z(G)$. 
	\end{itemize}
	Then $\fC_m$ recognizes coordinates. 
\end{proposition}

In Theorem~\ref{T.Limiting} we will prove that the class of all perfect groups does not recognize coordinates. 

\begin{proof}
	By \Cref{T.eq.groups}, it suffices to show that $x=e\rightarrow y=e$ is equivalent to an $h$-formula in the common theory $\Th(\fC_m)$ of $\fC_m$. 
	
\begin{claim}\label{L.perfect}
	Suppose that $G$ is in $\mathfrak{C}_{m}$. Then for all $a$ and $b$ in $G$ the following condition holds if and only if $a= e \rightarrow b=e$ holds. 		
	\begin{enumerate}
		\item [(*)] For all $z_i, t_i$, $i\leq m$,  such that $a=\prod_{i\leq m} [z_i,t_i]$ there are $x_i,y_i$, $i\leq m$ such that   $b=\prod_{i \leq m} [x_i,y_i]$, and such that we have 
		\[
		C_G(z_i^G,t_i^G:i\leq m) \subseteq  C_G(x_i^G,y_i^G:i\leq m). 			\]
	\end{enumerate}

\end{claim}	\label{C.4.8}
\begin{proof}
      By ($\dagger$), for every $c\in G\setminus \{e\}$, if $c=\prod_i [x_i,y_i]$ then $C_G(\langle x_i^G, y_i^G: i \leq m\rangle)=Z(G)$. On the other hand, if  $c=e$, then we can choose $x_i=y_i=e$ for all $i$ and therefore $C_G(\langle x_i^G, y_i^G: i \leq m\rangle)=G$. Therefore if $a\neq e$ or $b= e$ then (*) holds. 
Conversely, assume  $a= e$ and that (*) holds. Write $b=\prod_i [x_i,y_i]$. Then $C_G(\langle x_i^G, y_i^G: i \leq m \rangle) \supseteq G$ and $b$ must be equal to $e$.
This concludes the proof.
\end{proof}  

Claim~\ref{C.4.8} implies that the relation $\supp(a)\subseteq \supp(b)$ is definable in reduced products of groups in $\fC_m$, and the conclusion follows by Theorem~\ref{T.eq.groups}. 
\end{proof}

Since the class of perfect groups is closed under direct products, some perfect groups are not indecomposable and therefore do not recognize coordinates. Thus the following question remains open. 
\begin{question}
    Does the class of perfect indecomposable groups recognize coordinates?
\end{question}

A group $G$ is called \emph{quasisimple} if it is perfect and $G/Z(G)$ is simple. This for example includes all groups of the form  $\SL(n,F)$ for $n\geq 2$ and $F$ a field with $|F|\geq 4$. 

\begin{corollary}\label{C.SL(2.5)}
	If for some $n$, $\fC$ is the class of all quasisimple groups $G$ of commutator width $\leq  n$  then $\fC$ recognizes coordinates. 
\end{corollary}

\begin{proof}
	By the Jordan--H\" older theorem (see e.g., \cite[Theorem 13.6]{judson2020abstract}), if $G$ is in $\fC$ then $Z(G)$ is the only nontrivial normal subgroup of $G$ and in particular $G$ has no proper nonabelian normal subgroup. Therefore the conclusion follows by Proposition \ref{P.Perfect}. 
\end{proof}

\begin{corollary} \label{C.SL(n,F)} The class of groups of the form $\SL(n,F)$, for $2\leq n$ and  $|F|\geq 4$ recognizes coordinates. Following a common convention, we write $\SL(n,q)$ for $\SL(n,F_q)$, $q$ a prime power.  
	
	Since the center of $\SL(2,5)$ is $\bbZ/2\bbZ$, this gives an example of a group with nontrivial center that recognizes coordinates. 
\end{corollary}

\begin{proof} The commutator width of each of these groups is at most 2 by \cite{thompson1961commutators}, hence they are all perfect with uniformly bounded commutator width. 
	
	If $G=\SL(n,F)$ then $G/Z(G)$ is $\PSL(n,F)$, which is (assuming $n\geq 2$ and $|F|\geq 4$) well-known to be a simple group (see e.g., \cite{conrad2020simplicity}). Hence this class of groups is a subclass of all quasisimple groups with commutator width bounded by $2$, and Corollary~\ref{C.SL(2.5)} applies.
\end{proof}

\subsection{Free products of groups}
In this subsection we prove \ref{5.C} and \ref{6.C} of Theorem~\ref{T.RecognizingCoordinates}.
\begin{definition}
	Suppose that $G=H_0*H_1$ and both $H_0$ and $H_1$ are nontrivial groups. Then every element of $G\setminus \{e\}$ has the form $x_1 x_2 \dots x_{n}$ for some $n\geq 1$ where $x_i\in (H_0\setminus \{e\})\cup (H_1\setminus\{e\})$ for all $i\leq n$  and $x_i \in H_0$ if and only if $x_{i+1}\in H_1$ for all $i\leq n-1$. With these conventions we define the following ($\lambda$ stands for `length' while $L$' and `$R$' are for `left' and `right') 
	\begin{align*}
		\lambda(x_0x_1 \dots x_{n-1})&=n,\\
		\lambda(e)&=0\\
		L(x_1x_2\dots x_n)&=j\qquad \text{ if } x_1\in H_j\\
		R(x_1x_2\dots x_n)&=j\qquad \text{ if } x_n\in H_j. 
	\end{align*}
\end{definition}
The following lemma is well-known (e.g. \cite{magnus2004combinatorial}) but we include a proof for the reader's convenience. 

\begin{lemma}\label{L.FreeProduct}
	Suppose that $G=H_0*H_1$ is a free product of nontrivial groups and $a,b$ are in $G\setminus \{e\}$. Then we have the following. 
	\begin{enumerate}
		\item $\lambda(a)=\lambda(a^{-1})$. 
		\item $|\lambda(a)-\lambda(b)|\leq \lambda(ab)\leq \lambda(a)+\lambda(b)$. 
		\item $\lambda(a)>\lambda(b)$ implies $L(ab)=L(a)$ and $R(ba)=R(a)$. 
		\item If $R(a)\neq L(b)$ then $L(ab)=L(a)$ and $R(ab)=R(b)$. 
		\item \label{4.FreeProduct} If $|H_0|\geq 3$ then for every $a\in G\setminus \{e\}$ there are $b$ and $c$ in $a^G$ such that $L(b)=R(b)=0$ and $L(c)=R(c)=1$. 
	\end{enumerate}
\end{lemma}

\begin{proof}
	The first four items are immediate consequences of the definition of the free product. To prove the fifth, fix $a$ in $G\setminus \{e\}$. We will first find $b\in a^G$ such that $L(b)=R(b)=0$. If $L(a)=R(a)=0$ then let $b=a$, and if $L(a)=R(a)=1$ then let $b=xax^{-1}$ for some $x\in H_0\setminus \{e\}$. We may therefore assume $a=yd$ or $a=dy$ for some $y\in H_0\setminus \{e\}$ and $d$ such that $L(d)=R(d)=1$ or $d=e$. Since $|H_0|\geq 3$, we can choose  choose $z \in H_0 \backslash\{e,y^{-1}\}$. Then if $a = yd$, let $b=zy dz^{-1}$ and if $a = dy$, let $b=z^{-1} d y z$. Then $b$ is as required. 
	We can now set $c= xbx^{-1}$ for any $x\in H_1\setminus \{e\}$. 
\end{proof}

\begin{proposition}
	\label{P.FreeProduct} The class of all groups that are free products of two nontrivial groups at least one of which has cardinality at least 3 recognizes coordinates.  
\end{proposition}

\begin{proof} By Theorem~\ref{T.RecognizingCoordinates} \eqref{2.RecognizingCoordinates}  it suffices to prove that  for every $a\neq e$, $C_G(a^G)=\{e\}$. This is equivalent to asserting that for all $a$ and $b$ in $G\setminus \{e\}$ there are $a'\in a^G$ and $b'\in b^G$ such that $a'b'\neq b'a'$. Fix $a$ and $b$ in $G\setminus \{e\}$. By Lemma~\ref{L.FreeProduct} \eqref{4.FreeProduct}  there are $a'\in a^G$ with $L(a')=R(a')=0$ and $b'\in b^G$ with $L(b')=R(b')=1$. Then $L(a'b')=L(a')=0$ and $L(b'a')=L(b')=1$, therefore $a'b'\neq b'a'$, as required. 
\end{proof}

\begin{proposition}\label{P.Z/2Z*Z/2Z} The group $(\bbZ/2\bbZ)*(\bbZ/2\bbZ)$ recognizes coordinates. 
\end{proposition}

\begin{proof} By Proposition~\ref{P.2}, with $m=n=2$, it suffices to prove that (i) for any two elements $a$ and $b$ of order 2 there is $a'\in a^G$ such that $a'b\neq ba'$ and (ii) every element of $G\setminus \{e\}$ is a product of at most two elements of order 2. 
	
	Let $G=(\bbZ/2\bbZ)*(\bbZ/2\bbZ)$ and let  $x$ and $y$ denote the generators of the two copies of $\bbZ/2\bbZ$. Then every nontrivial word in $G$ is an alternating sequence of $x$'s and $y$'s. We need two straightforward properties of such words $w$.  First, the inverse of $w$ is obtained by taking $w$ in  the reverse order. Second, the first and the last symbols of $w$  are equal (in the terminology of Lemma~\ref{L.FreeProduct}, $L(w)=R(w)$) if and only if its length is odd. These facts together imply that $w^2=e$ if and only if its length is odd. Assume $a$ and $b$ are of an odd length.  If $L(a)\neq R(b)$ then since $R(a)=L(a)$ and $R(b)=L(b)$, we have $ab\neq ba$. Otherwise, we can take $a'=xax
	$ or $a'=yay$ to obtain $a'\in a^G$ that does not commute with $b$.

	Finally, every nontrivial word of even length is clearly a product of two nontrivial words each one of which has an odd length and Proposition~\ref{P.2}, with $m=n=2$, applies to show that $G$ recognizes coordinates.
\end{proof}

\subsection{Graph products of groups}\label{S.GraphProducts}
Assume $\Gamma=(V(\Gamma),E(\Gamma))$ is a graph (we will write $(V,E)$ when $\Gamma$ is clear from the context) and $\cG=\{G_v\vert v\in V\}$ is a family of groups indexed by its vertices. Then the \emph{graph product} (\cite{green1990graph}) is the group $\Gamma \cG$  defined as the quotient of the free product $\ast_{v\in V} G_v$ modulo the normal subgroup generated by the commutators 
\[
\{aba^{-1} b^{-1}\vert a\in G_v, b\in G_w, \{v,w\}\in E\}. 
\]
Thus in the case when $\Gamma$ is a complete graph, $\Gamma \cG$ is the direct product $\prod_{v\in V} G_v$ while in the case when $\Gamma$ is the null graph $\Gamma\cG$ is the free product $\ast_{v\in V} G_v$.
Two prominent cases of this construction are the right-angled Coxeter groups (when $G_v\cong \bbZ/2\bbZ$ for all $v$) and the right-angled Artin groups (when $G_v\cong \bbZ$ for all $v$).  Although in some of the literature the graph $\Gamma$ is required to be finite, we do not impose any restriction on the cardinality of $\Gamma$ or groups $G_v$.  
We say that a graph product $\Gamma\cG$ is \emph{nontrivial} if  $|V|\geq 2$ and  $|G_v|\geq 2$ for all $v\in V$.  
This is not a loss of generality  since $\Gamma\cG$ is isomorphic to $\Gamma'\cG'$ where $\Gamma'$ is the induced subgraph on $V'=\{v\in V\mid |G_v|\geq 2\}$ and $\cG'=\{G_v\vert v\in V'\}$. 

A remarkable converse to the Feferman--Vaught theorem for graph products was proved in \cite{casals-ruiz2024on}.

Theorems~\ref{T.GraphProduct.Montse} and  \ref{T.GraphProduct} below are stated more naturally in terms of the complement of $\Gamma$, denoted $\bar \Gamma$ (to be specific, if $\Gamma=(V,E)$ then $\bar\Gamma=(V,\bar E)$ where $\bar E$ is the complement of $E$).  Since free products are a special case of graph products associated with the null graph, the following is a generalization of Proposition~\ref{P.FreeProduct}. 
Following \cite{minasyan2019acylindrically}, we  say that a graph $\Gamma$ is \emph{irreducible} if its complement graph $\bar\Gamma$  is connected.

We  learned that  Theorem~\ref{T.GraphProduct.Montse} below follows easily from our results (taken together with some standard facts) from Montserrat Casals--Ruiz.  This theorem confirms a conjecture stated in the original version of this paper, where  a slightly weaker result,  included below as Theorem~\ref{T.GraphProduct} together with a self-contained proof,   was proved.

\begin{theorem}\label{T.GraphProduct.Montse} 
	A nontrivial graph product $\Gamma\cG$ associated with a finite graph~$\Gamma$ with at least two vertices  
	recognizes coordinates if and only if the complement $\bar\Gamma$ of $\Gamma$ is connected.
\end{theorem}

\begin{proof} Suppose that $G=\Gamma\cG$ is a graph product of non-trivial groups with respect to some finite graph $\Gamma$ with at least two vertices whose complement is connected. 
By \cite[Corollary 2.13]{minasyan2019acylindrically}, $G$ is either virtually cyclic or acylindrically hyperbolic (A group is virtually cyclic if it has a cyclic subgroup of finite index).  We claim that the infinite dihedral group  $(\bbZ/2\bbZ)*(\bbZ/2\bbZ)$ is the   only virtually cyclic graph product $G$ such that~$\Gamma$ has at least two vertices and~$\bar \Gamma$ is connected. 
This is well-known to the experts but as non-experts we  feel compelled to sketch a proof.  Assume that $G$ is virtually cyclic.
By \cite[Theorem~12]{epstein1962ends}, $G$ is virtually cyclic if and only if it has exactly two ends (the latter notion will be blackboxed in this proof, but the definition can be found in \cite[p. 115]{epstein1962ends}). 
Since every subgroup of a virtually cyclic group is virtually cyclic, we may assume that all~$G_v$ are finitely generated by replacing every~$G_v$ which is not finitely generated with a nontrivial finitely generated subgroup.  
Graph products of finitely generated groups with exactly two ends are classified in \cite[Proposition~B]{varghese2020on}. By our assumption that $\bar \Gamma$ is connected, (i) of this Proposition does not apply.  Also, the set of all vertices $v$ that are adjacent to all $w\neq v$ is empty, therefore (ii) of this Proposition reduces to $G\cong (\bbZ/2\bbZ)*(\bbZ/2\bbZ)$, as required. This group recognizes coordinates by Proposition~\ref{P.Z/2Z*Z/2Z}.

We may therefore assume that $G$ is acylindrically  hyperbolic, and by Corollary~\ref{C.Montse}  it remains to prove that it has no nontrivial finite normal subgroups. 
Assume otherwise and let $F$ be a finite normal  subgroup of $G$. Then $F$  is conjugate to a subgroup of some vertex group $G_v$, because the normal form of elements of $G$ implies that every element of finite order  is conjugate to an element of some $G_v$. We may assume $F$ is a subgrouop of $G_v$. Let $w$ be any vertex of $\Gamma$ not adjacent to $v$. Then the group $G_v*G_w$ is a homomorphic image of $G$, and $F$ is a normal subgroup of it. Since $G_w$ is nontrivial, this is absurd, and the proof is complete. 
\end{proof}

The first part of Theorem~\ref{T.GraphProduct} below is a slight weakening of Theorem~\ref{T.GraphProduct.Montse} that was present in the original version of our paper. Although it takes up this entire subsection,  our original, self-contained, proof of  this theorem is included because we believe it is of independent interest, and also because it applies to infinite graphs.  We conjecture that the conclusion of   Theorem~\ref{T.GraphProduct.Montse} also holds without the restriction that the graph $\Gamma$ be finite.

\begin{theorem}\label{T.GraphProduct} 
	A nontrivial graph product $\Gamma\cG$ 
	such that $|G_{\tilde v}|\geq 3$ for some $\tilde v\in V$   
	recognizes coordinates if and only if the complement $\bar\Gamma$ of $\Gamma$ is connected.
	Moreover, the class $\cC$ of all nontrivial graph products $\Gamma\cG$ such that~$\bar{\Gamma}$  is connected and $|G_v|\geq 3$ for some $v\in V(\Gamma)$  recognizes coordinates. 
\end{theorem}

The easy direction of this theorem is \Cref{L.graph.disconnected}.  

The assertion that $|G_v|\geq 3$ for some $v$ is equivalent to the assertion that $\Gamma\cG$ is not a right-angled Coxeter group.

\begin{lemma}
	\label{L.graph.disconnected} If $\bar \Gamma$ is not connected then $\Gamma\cG$ is decomposable.   
\end{lemma}

\begin{proof}
	Let $V(\Gamma)=X\sqcup Y$ be a partition of into nonempty sets such that no vertex in $X$ is $\bar\Gamma$-adjacent to a vertex in $Y$. In other words, the complete bipartite graph with bipartition $X,Y$ is a subgraph of $\Gamma$.  Let $\Gamma_X$ and $\Gamma_Y$ be the induced subgraphs of $\Gamma$ on $X$ and $Y$. Then, with $\cG_X=\{G_v\vert v\in X\}$ and $\cG_Y=\{G_v\vert v\in Y\}$  the definition of $\Gamma\cG$ implies that $\Gamma\cG\cong \Gamma_X \cG_X\times \Gamma_Y\cG_Y$. Since each $G_v$ is nontrivial, $\Gamma\cG$ is decomposable as a product of nontrivial groups. 
\end{proof}

Our positive result will require considerably more work and references to the literature.
The \emph{normal form} for words in graph products appears in \cite[Theorem ~3.9, also see Definition~3.5]{green1990graph} and \cite[\S 4]{hsu1999on} for a shorter proof. It asserts that for every $g\in \Gamma\cG\setminus \{e\}$ there are $n\geq 1$, $v(i)\in V$, $g(i)\in G_{v(i)}\setminus \{e\}$, for $1\leq i\leq n$ such that $g=g_1\cdot\ldots \cdot g_n$, and for all $1\leq i\leq k<j\leq n$ such that 
$[g_i,g_l]=e=[g_m,g_j]$ for all $i+1\leq l \leq k$ and all $k+1\leq m \leq j$, we have $v(i)\neq v(j)$. A moment taken to parse the latter condition on commutators reveals that it implies $g=g_1\cdot\ldots\cdot g_{i-1} g_{i+1}\cdot\ldots \cdot g_k g_i g_j g_{k+1}\cdot\ldots \cdot g_{j-1} g_{j+1}\cdot\ldots \cdot g_n$, in which case $v(i)=v(j)$ would imply that $g_i g_j\in G_{v(i)}$ and that $g$ could be presented as a word of length $n-1$.  	The elements $g_i$ in a normal form of $g$ are called \emph{syllables} of $g$.  
Note that the normal form of $g$ is not unique, since if $\{v(i),v(i+1)\}\in E$ then $g_i$ and $g_{i+1}$ can be swapped without changing the value of $g$. However, the set of syllables of $g$ is uniquely determined. 

Following \cite{genevois2019automorphisms}, for $g\in \Gamma\cG$  the \emph{head} of $g$, denoted $\head(g)$, is the set of all first syllables appearing in the reduced words representing $g$. The \emph{tail} of $g$, denoted $\tail(g)$, is the set of all last syllables appearing in the reduced words representing $g$. Clearly these sets depend on $g$ and not on the normal form used to represent it. For $g\in \Gamma\cG$ with the normal form $g=g_1\cdot \ldots \cdot g_n$ we write 
\begin{align*}
	L(g)&=\{v\in V\vert ((\exists a\in \head(g)) a\in G_{v}\},\\
	R(g)&=\{v\in V\vert ((\exists a\in \tail(g)) a\in G_{v}\},\\
	V(g)&=\{v\in V\vert (\exists i) g_i\in G_{v}\}. 
\end{align*}
(A reasonable notation for $V(g)$ would be $\supp(g)$ but in other sections of the present paper $\supp$ has a different meaning.)

Lemma~\ref{L.GraphGeneralities} below collects straightforward facts that will be used in the proof of lemmas leading towards the nontrivial part of Theorem~\ref{T.GraphProduct}. 
\begin{lemma}
	\label{L.GraphGeneralities}
	Suppose that $\Gamma\cG$ is a nontrivial graph product and let  $g=g_1\cdot\ldots \cdot g_m$ and $h=h_1\cdot\ldots \cdot h_n$ be  elements of $\Gamma\cG\setminus \{e\}$ in normal form. 
	\begin{enumerate}
		\item \label{1.GraphGeneralities} If $g_i\in G_{v(i)}\setminus \{e\}$ for $1\leq i\leq n$, then $v\in L(g)$ if and only if there is $i$ such that $v=v(i)$ and  all $1\leq j<i$ satisfy $\{v(j),v(i)\}\in E$. 
		\item \label{2.GraphGeneralities} If $V(g)\cap V(h)=\emptyset$ then $\head(g)\subseteq \head(gh)\subseteq \head(g)\cup \head(h)$. 
		\item \label{3.GraphGeneralities} If $V(g)\cap V(h)=\emptyset$ and for every $v\in L(h)$ there is $w\in V(g)$ such that $\{v,w\}\notin E$ then $L(g)=L(gh)$.
	\end{enumerate}
	Since $\tail(g)=\head(g^{-1})$ assertions analogous to the above  hold for $\tail(g)$ and $\tail(h)$. \qed
\end{lemma}

\begin{lemma}\label{L.GraphConjugacy} In a nontrivial graph product $\Gamma\cG$ such that $\bar\Gamma$ is connected  and $|G_{\tilde v}|\geq 3$ for some $\tilde v\in V$  the following holds. For  every  $v\in V$, every $g\in \Gamma\cG\setminus \{e\}$ is conjugate to some $\tilde g$ such that $L(\tilde g)=R(\tilde g)=\{v\}$. 
\end{lemma}

\begin{proof} 
	Fix $g\in \Gamma\cG\setminus\{e\}$ and let $L(g)=\{v(1), \dots, v(m)\}$. Since the syllables in $\head(g)$ commute, $\{v(i),v(j)\}$ is not an edge of $\bar \Gamma$ for all $i,j\leq m$.
	Since $\bar\Gamma$ is connected, there is $v(*)\in V\setminus L(g)$. 
	Choose a spanning tree $T$ for $\bar \Gamma$. (That is, $T$ is an acyclic connected subgraph of $\bar\Gamma$ with the same vertex set.) By connectivity of the complement, for every $i\leq m$ there is a unique path $P(i)$ in $T$ connecting $v(*)$ and $v(i)$. Consider $T(0)=\bigcup_{i\leq m} P(i)$ as a subgraph of $T$ and let $d$ denote the graph distance on $T(0)$.  Let $\sqsubseteq$ be a linear ordering of the vertices of $T(0)$  such that $x\sqsubseteq y$ implies $d(v(*),x)\leq d(v(*),y)$, and let $x(1), \dots, x(k)$ be the $\sqsubseteq$-increasing list of all vertices of $T(0)$. Then $x(1)=v(*)$. For each $1\leq j\leq k$ pick $h_j\in G_{x(j)}\setminus \{e\}$.

	Let $a=h_1 h_2\dots h_k$.  		If $v(*)=\tilde v$ let $g'=a g a^{-1}$. 
	If $v(*)\neq \tilde v$, then let $x(1),\dots, x(p)$ be a $\bar{\Gamma}$-path from $\tilde v$ to $v(*)$. Choose $h'_i\in G_{x(i)}$  for $1\leq i< p$, let $b=h'_1 h'_2\cdot\ldots \cdot h'_{p-1} $. Then Lemma~\ref{L.GraphGeneralities} \eqref{1.GraphGeneralities} implies $L(ba)=\{\tilde v\}$.  With  $g'=b a g a^{-1} b^{-1}$, Lemma~\ref{L.GraphGeneralities} \eqref{3.GraphGeneralities} implies   $L(g')=\{\tilde v\}$.

	At this point in the construction we don't know what is the relation of $\tilde v$ to the vertices in $R(g')$ but we would like to be able to assume that $\tilde v\in R(g')$. 
	Assume this is not the case. Since $\tilde v\in L(g')$ and $|G_{\tilde v}|\geq 3$, if $a\in \head(g')\cap G_{\tilde v}$ then we can choose $c\in G_{\tilde v}\setminus\{e,a^{-1}\}$.  In this case we have  $L(cg')=L(g')$ and   
	$g''=cg'c^{-1}$ satisfies $L(g'')=\{\tilde v\}$ and $\tilde v\in R(g'')$.

	Let $R(g')=\{w(1),\dots, w(n)\}$ (with $w(1)=\tilde v$). The next portion of the proof, we look for a conjugate $\tilde{g}$ of $g'$ so that  such that  $R(\tilde{g})=L(\tilde{g})=\{\tilde v\} $. This is analogous to previous part of the proof where we were obtained $g'$. Since  $\{w(i),w(j)\}$ is not an edge of $\bar \Gamma$ for all $i,j\leq n$ and $\bar \Gamma$ is connected, there is $w(*)\in V\setminus R(g')$. Find a tree $U(0)$ with root $w(*)$ and leaves  $\{w(1),\dots, w(n)\}$, let $d(\cdot,\cdot)$ denote the graph distance in $U(0)$, fix a linear ordering of its vertices such that  $x\sqsubseteq' y$ implies $d(w(*),x)\leq d(w(*),y)$, and let $y(1), \dots, y(k')$ be the  $\sqsubseteq'$-increasing list of all vertices of $U(0)$. For each $1\leq j<k'$ pick $h'_j\in G_{y(j)}\setminus \{e\}$. 
	Let $a'=h_{k'-1}\cdot\ldots \cdot h'_1$. Then, using Lemma~\ref{L.GraphGeneralities} as before,  $g'''=(a')^{-1} g'' a'$ satisfies $R(g''')=L(g''')=\{w(*)\}$. 
	
	It remains to show that we can assure that $w(*)$ is equal to the distinguished vertex $v$ fixed in the statement of this lemma. If $w(*)\neq v$, choose a path in $\bar\Gamma$ from $w(*)$ to $v$, $w(*)=x(0),x(1), \dots , x(m)=v$, fix $f_i\in G_{x(i)}\setminus \{e\}$, let $d= f_m \cdot f_{m-1}\cdot\ldots \cdot f_1$ and let $\tilde g=d g''' d^{-1}$. Once again,  Lemma~\ref{L.GraphGeneralities} implies  that $L(\tilde g)=R(\tilde g)=\{v\}$. 
\end{proof}

 The reader may be under the impression that the fact that every $g\in \Gamma\cG$ is conjugate to a word $g'$ such that $L(g')\cap R(g')=\emptyset$ (this is \cite[Lemma~3.16]{green1990graph}, where such words $g'$ are called \emph{cyclically reduced}) may be used to remove the assumption that $|G_{\tilde v}|\geq 3$ in Lemma~\ref{L.GraphConjugacy} and  Theorem~\ref{T.GraphProduct}. The proof of  that $(\bbZ/2\bbZ)*(\bbZ/2\bbZ)$ recognizes coordinates (Proposition~\ref{P.Z/2Z*Z/2Z})  is very different from the above proof, suggesting that a common generalization requires additional nontrivial ideas. 
   In particular, using the notation from Proposition~\ref{P.Z/2Z*Z/2Z}, one should note that the element $ab$ of $(\bbZ/2\bbZ)*(\bbZ/2\bbZ)$ (and any other reduced word of even length in this group) does not satisfy the conclusion of Lemma~\ref{L.GraphConjugacy}. 

\begin{lemma}
	\label{L.GraphRecognizing} If $\cC$ is the class of all nontrivial graph products $\Gamma\cG$ such that~$\bar \Gamma$ is connected and   $|G_{\tilde v}|\geq 3$ for some $\tilde v\in V$,   then $\cC$ recognizes coordinates. 
\end{lemma}

\begin{proof} By Theorem~\ref{T.RecognizingCoordinates} \eqref{2.RecognizingCoordinates}, it suffices to show that for every $\Gamma\cG\in \cC$, for all $a$ and $b$ in $\Gamma\cG\setminus \{e\}$, some $\tilde a$ conjugate to $a$ does not commute to some~$\tilde b$ conjugate to $b$. Since $\bar \Gamma$ is connected and it has at least two vertices, we can fix vertices $v$ and $w$ adjacent in $\bar\Gamma$. 
	By Lemma~\ref{L.GraphConjugacy}, there are $\tilde a$ and $\tilde b$ conjugate to $a$  and $b$, respectively, such that $L(\tilde a)=R(\tilde a)=\{v\}$ and $L(\tilde b)=R(\tilde b)=\{w\}$. Since $\{v,w\}$ is not an edge in $\Gamma$,  $L(ab)=L(a)\neq L(b)=L(ba)$, hence $ab\neq ba$.  
\end{proof}

\begin{proof}[Proof of Theorem~\ref{T.GraphProduct}] By Lemma~\ref{L.graph.disconnected} below, if $\bar\Gamma$ is not connected then $\Gamma\cG$ is decomposable and it therefore does not recognize coordinates by Theorem~\ref{Fact:construction}. By Lemma~\ref{L.GraphRecognizing}, the class $\cC$ recognizes coordinates. 
\end{proof}

\subsection{Symmetric groups}

For an integer $n\geq 2$, by $\Sym_n$ we denote the symmetric group on $n$ elements. By $\Sym$ we denote an arbitrary  finite symmetric group. As mentioned in the proof of Theorem~\ref{T.RC} \ref{2.Cplus}, the 
class of finite symmetric groups containing four or more elements recognizes coordinates. In this section we slightly modify this argument to be able to include the symmetric group on $3$ elements. Then, we pursue the model-theoretic motivation described in the introduction and prove a quantifier elimination result. Here, we will however have to exclude $\Sym_6$, because of the existence of a nontrivial outer automorphism.

\begin{lemma}\label{L.Defining2Cycles}
	\begin{enumerate}
		\item   \label{1.Defining2Cycles}      There is an $h$-formula  $\psi(x)$ which uniformly defines the set $C_2\cup \{e\}$ where $C_2$ is the set of transpositions, in all symmetric groups of the form  $\Sym_n$ for $n \neq 2$ and $n \neq 6$. 
		\item \label{2.Defining2Cycles}        For every $k\geq 3$,  there is a formula $\phi_k(x)$ which uniformly defines the set $C_k$ of $k$-cycles in all symmetric group $\Sym_n$ with $n \neq 6$.
	\end{enumerate}
	\end{lemma}

\begin{proof} \eqref{1.Defining2Cycles}
	We claim that the following $h$-formula defines, uniformly in $\Sym_{n}$ for all $n>2$ and $n\neq 6$, the set $C_2 \cup \{e\}$: 
	\[
	\psi(x):=x^2=e \wedge (\forall g)~ (x gxg^{-1})^6=e\wedge (\exists g)~(x gxg^{-1})^3=e. 
	 \]
	 We remark that the last conjunct is required only in $\Sym_4$. 
     
     Since the conjugate of a 2-cycle is a 2-cycle and the product of two distinct 2-cycles is either a 3-cycle or a 2-2-cycle, every 2-cycle satisfies each of the above conjuncts. 
	 
	We now prove the converse. Every element of order 2 is a product of transpositions with disjoint supports (for convenience we will call such permutations $2$-$2$-\dots -$2$-cycles). 
	 
	 The case of $n = 3$ is trivial. Suppose that $\sigma$ is an element of order 2 in $\Sym_n$ for $n\geq 4$ and $n\neq 6$ that is not in $C_2$.
	    
	    If $n=4$ then $\sigma$ is a $2$-$2$-cycle. Any two $2$-$2$-cycles are conjugate, and  in $\Sym_4$, the product of two distinct $2$-$2$-cycles is a $2$-$2$-cycle (e.g.
	 		$(12)(34) \circ (14)(23)=(13)(24)$). Thus in this case $\sigma$  does not satisfy the last conjunct of the formula.
	 
	 Now consider the case when $\sigma$ has a fixed point. Let $\sigma$ be the product of $k\geq 2$ disjoint transpositions, $\sigma=\prod_{j=1}^k (m_{2j-1},m_{2j})$ and let $m_{2k+1}$ be a fixed point of $\sigma$. Then $\sigma':=\prod_{j=1}^k (m_{2j},m_{2j+1})$ is conjugate to $\sigma$, and $\sigma'\circ \sigma$ is the $2k+1$-cycle is the $2k+1$ cycle $(m_1 m_3\dots m_{2k+1} m_{2k} m_{2k-2}\dots m_2)$. Since $2k+1\geq 5$, we have that  $\sigma$ does not satisfy the middle conjunct of the formula. 
	 
	 This proves the claim in the case when $n$ is odd, and when $k \leq 3$, and $n \geq  8$.  We may therefore assume $k\geq 4$; let us assume $k=4$ and $n=8$ for a moment. In the following special case,  we obtain a conjugate $\sigma'$ of sigma such that $\sigma\circ \sigma'$ has order 4 and therefore fails the middle conjunct of the formula: 
	         \[
	         (16)(47)(52)(38)\circ (15)(64)(73)(28)= (1234)(5678).
	         \]   
	      In general, if $k\geq 4$, then by restricting $\sigma$ to an 8-element $\sigma$-invariant subset of its support we can find $\sigma'$ such that $\sigma\circ \sigma'$ has order 4. This concludes the proof of \eqref{1.Defining2Cycles}.    
%
	Prior to embarking on the proof of \eqref{2.Defining2Cycles}, 
	we note that in $\Sym_n$ for all $n\geq 3$, $n\neq 6$,  the set of transpositions $C_2$ is given by the formula
	\[
	\psi(x)\wedge x\neq e.
	\]
      This is however not an $h$-formula. 
     
\eqref{2.Defining2Cycles}
      We need a formula that defines $C_2$ in $\Sym_n$ for all $n\neq 6$. To include the case $n= 2$, we can consider the formula:
    \[
    [\psi(x)\vee (\exists y)(\forall z)~ (z=y\vee z=e) ]\wedge x\neq e .
    \]
The set $C_k$ of $k$-cycles can now be defined as the set of products of $(k-1)$ pairwise distinct $2$-cycles $c_1,\dots, c_{k-1}$ such that $c_ic_j \neq c_jc_i$ if and only if $\vert i-j\vert=1$ for $i,j<k$.  
\end{proof}

Since there is an automorphism of $\Sym_6$ sending $2$-cycles to $2$-$2$-$2$-cycles,  there is no parameter-free definition of the set of $2$-cycles in $\Sym_6$ and Lem\-ma~\ref{L.Defining2Cycles} cannot be improved by including $\Sym_6$ in the set of groups to which it applies. 

Even though we cannot separate $2$-cycles from $2$-$2$-$2$-cycles in $\Sym_{6}$, we can still prove the following:

\begin{corollary}\label{C.SymAllExceptSix}
	The class of symmetric groups $\Sym_n$, $n\geq 3$ recognizes coordinates.
\end{corollary}

\begin{proof} Let $\Sym$ denote an arbitrary $\Sym_n$ for $n\geq 3$ and $n\neq 6$. Let $\psi(x)$ be the $h$-formula defining $\mathcal{C}_2\cup \{e\}$ as in the proof of \Cref{L.Defining2Cycles}\eqref{1.Defining2Cycles}. For $x,y\in \Sym$, let $xRy$ denote the relation ``$x$ commutes with all conjugates of $y$'' used in the proof of Theorem~\ref{T.RecognizingCoordinates} \eqref{2.RecognizingCoordinates}.  Then $xRy$ is an $h$-formula and for all $x,x' \in \Sym$, we have 
\[ 
x'=e \rightarrow x=e \  \Leftrightarrow \ (\forall y\in \mathcal{C}_2\cup \{e\}) ~ \ \left( x'Ry \rightarrow xRy \right). 
\]

The right-hand side formula is an $h$-formula, as it is equivalent to
\begin{align*}
	(\exists y)  ~(\psi(y) ~\wedge \ x'Ry) 
	\wedge (\forall y) ~ \left((\psi(y) ~\wedge \ x'Ry) \rightarrow xRy \right). 
\end{align*}
Notice furthermore that the same equivalence holds in $\Sym_6$. Indeed, the formula $\psi(y)$ defines there the set containing $C_2\cup \{e\}$ and the $2$-$2$-$2$-cycles, and the centralizer of the conjugacy class of a $2$-$2$-$2$-cycles is also $\{e\}$.

The conclusion follows by \Cref{T.eq.groups}.  
\end{proof}

Since $\Sym_2$ is abelian, any class containing it cannot recognize coordinates by \Cref{T.example:DNR}\ref{DNR.decomposable}; this was originally observed in \cite[Example 2.6]{farah2022corona}.

\begin{lemma}\label{L.Sym.xneqe}
In  the common theory of $\Sym_n$, for all $n\geq 3$, the formula $x\neq e$ is equivalent to an $h$-formula. 
\end{lemma}

\begin{proof}
Since both the formulas $\psi(\tau)$ (asserting that $\tau\in C_2\cup \{e\}$ if $n\neq 6$) and $xRy$ used in the proof of \Cref{C.SymAllExceptSix} are equivalent to $h$-formulas and since $eRx$ holds for all $x$, the right-hand side of  
\[
x\neq e \ \Leftrightarrow \ (\forall y) (\forall \tau ) \ ((\psi(\tau)\wedge \tau R y) \rightarrow \tau R x )
\]
 is an explicit $h$-formula. We remark that this Lemma also follows from \Cref{Prop:Omarov}.
\end{proof}

\subsection{Quantifier elimination in products of a symmetric group}

We proceed as discussed in \Cref{section:FV} to define a language interpretable in  $\mathcal{L}_{\textrm{groups}}$ in which reduced products of $\Sym_n$ for a fix $n \geq 4, n\neq 6$ eliminate quantifiers.
The following definition and fact are heavily inspired by a post on Math Stack Exchange (\cite{kruckman2021stack}), from which we borrow the terminology and some of the proofs, which are included for the reader's convenience. 

\begin{definition}\label{D.identifying}
	A pair  of distinct transpositions in some symmetric group which have at least one common element $k$, is said to \emph{identify} $k$, or  that it is \emph{identifying}. The set of all identifying pairs is denoted by $P$.
On $P$ consider the equivalence relation $E$ given by 
\[
(\sigma,\tau) E (\sigma',\tau') \quad   \Leftrightarrow \quad \text{pairs $(\sigma,\tau)$ and $(\sigma',\tau')$ identify the same element $k$.}
\] 
Let $\mathcal{L}_X$ denote  the group language expanded by a sort for the base set $\{1,\dots, n\}$ and the natural projection $\pi_X: P \rightarrow X$.
\end{definition}

Clearly, $E$ is an equivalence relation on the set $P$ and every $\Sym_n$ has a natural expansion to an $\calL_X$-structure. Also, in $\Sym_n$ the quotient $X:=P/E$ is naturally identified with the base set $\{1,\dots,n\}$. It will be convenient to write $\Sym_X$ for the symmetric group on the base set $X$.

\begin{fact}\label{Fact:symmetricgroups}  The following facts hold (uniformly) in the theory of all $\Sym_n$, for $n\geq 4$ and $n\neq 6$.

    \begin{enumerate}
        \item\label{f.sym1} The set $P$ of identifying pairs is (uniformly) $h$-definable, namely there exists an $h$-formula defining $P$ in all $\Sym_n$, for $n \geq 4$ and $n \neq 6$.
        \item\label{f.sym2} The relation $E$ on $P$ is (uniformly) $h$-definable.
          \item\label{f.sym4}       For $k$ and $h$ in the base set  $X$ and $\nu \in \Sym_X$, the relation 
        \[\nu=\begin{cases}
            ~(kh) & \text{ if }k\neq h\\
            ~e & \text{ otherwise}
        \end{cases}\]
        is defined by an $h$-formula $\psi(\nu,k,h)$ in the language $\mathcal{L}_X$. 
        
        \item\label{f.sym5} For $s\geq 2$ and distinct $k_1, \dots, k_s$ 
        in $X$, and $x\in \Sym_X$, the relation:
        \[x=
\begin{pmatrix}
1 & \dots & s \\
k_1 & \dots & k_s
\end{pmatrix},
\]        
is defined by an $h$-formula in the language $\mathcal{L}_X$.
    
    \item\label{f.sym6} For every $n\geq 4$, $n\neq 6$, $k\geq 1$ and a $k$-tuple $(\sigma_1,\dots,\sigma_k)$ of elements of $\Sym_n$,  the type $p=\tp(\sigma_1,\cdots,\sigma_k/\emptyset)$ is isolated by an $h$-formula $\phi_p(x_1,\dots,x_k)$. 
    \end{enumerate}
\end{fact}

We briefly justify  these facts:
\begin{proof}
    (\ref{f.sym1}) As two distinct transpositions move a common element if and only if their product has  order three, we may use the following formula:
       \[\sigma,\tau\in C_2\wedge \sigma\tau^{-1} \neq e  \wedge  \ (\sigma\tau)^3=1.\] 
    This is equivalent to an $h$-formula by Lemma~\ref{L.Sym.xneqe} 
    
    (\ref{f.sym2}) We will prove that the relation $E$ is defined by the following formula:
        \begin{equation}\label{eq.E}
        \begin{aligned}
           (\forall \nu\in C_2\cup\{e\})\quad  
            & [\nu\sigma \nu=\sigma \wedge \sigma \neq \nu ] \rightarrow \\ & [(\tau  \nu \tau'\nu)^3= (\tau  \nu \sigma'\nu)^3=(\sigma  \nu \tau'\nu)^3=(\sigma  \nu \sigma'\nu)^3=1 ]   \\
            \wedge &  [\nu\tau \nu=\tau \wedge \tau \neq \nu ] \rightarrow \\
            &[(\tau  \nu \tau'\nu)^3= (\tau  \nu \sigma'\nu)^3=(\sigma  \nu \tau'\nu)^3=(\sigma  \nu \sigma'\nu)^3=1 ].
        \end{aligned}
        \end{equation}
    Indeed, if both $(\sigma,\tau)$ and $(\sigma',\tau')$ identify an element $k$, then any transposition $\nu$ disjoint from $\sigma$ or $\tau$ does not move $k$. Then, as $\tau'$ and $\sigma'$  move $k$, so do $\nu\tau'\nu$ and $\nu\sigma'\nu$. Therefore each one of $\tau  \nu \tau'\nu,\tau  \nu \sigma'\nu,\sigma  \nu \tau'\nu$ and $\sigma  \nu \sigma'\nu$ has order $1$ or $3$, and formula \eqref{eq.E} holds.
        
    Conversely, assume that  $(\sigma,\tau)$ identifies $k$ and $(\sigma',\tau')$ identifies some $k'\neq k$. We prove that \eqref{eq.E} doesn't hold by considering cases. 
    \begin{itemize}
        \item If at least one of the pairs of transpositions among $(\sigma,\sigma'),(\sigma,\tau'),(\tau,\sigma')$, and $(\tau,\tau')$ is disjoint, then the formula \eqref{eq.E} above fails by taking $\nu=e$.
        \item Now assume that none of these pairs are disjoint. Since $k\neq k'$, there is $k''$ such that  $\tau,\sigma,\tau'$, and $\sigma'$ are transpositions of $k,k'$, and $k''$. By swapping $\sigma$ with $\tau$ and $\sigma'$ with $\tau'$ if needed, we may assume that $\sigma=(kk'),\tau=(kk''),\sigma'=(k'k)$, and $\tau'=(k'k'')$.
        As the base set has at least four  elements, we may take $l\notin \{k,k',k''\}$ and let $\nu:=(lk'')$.  Then $\tau \nu \tau' \nu =(kk'')(k'l)$ has order $2$ and \eqref{eq.E} fails.
    \end{itemize}    
    As $\sigma \neq \nu$ is equivalent to an $h$-formula by Lemma~\ref{L.Sym.xneqe}, one sees that~\eqref{eq.E} is also equivalent to an $h$-formula.

    (\ref{f.sym4}) Let $\psi(\nu,k,h)$ be the following formula (using the $h$-formulas for $P$ and $E$ provided by the above)
    \[(\forall \sigma, \tau \in C_2) ~ \left( [(\sigma,\tau) \in P \wedge \pi_X(\sigma, \tau)=k ] \rightarrow \pi_X(\nu\sigma\nu, \nu\tau\nu)=h \right).\]
    In words, if  $\psi(\nu,k,h)$ holds, then for all $\sigma,\tau\in C_2$,  if $(\sigma,\tau)$ identifies $k$, then $(\nu \sigma \nu ,\nu \tau\nu)$ identifies $h$.
    Clearly, if $k,h$ are distinct elements of $X$, then $\nu=(k,h)$, and if $k=h$, then $\nu=e$ as required.
    
    (\ref{f.sym5}) The permutation $\rho:=\begin{pmatrix}
1 & \dots & s \\
k_1 & \dots & k_s
\end{pmatrix}$ can be written as a product of transposition. Therefore, this relation is expressible as a conjunction of $h$-formulas of the form $\psi(\nu,k,h)$ as in (\ref{f.sym4}).

(\ref{f.sym6}) 
For $i\leq k$, consider the following permutation
    \[\sigma_i:=
\begin{pmatrix}
1 & \cdots & n \\
\sigma_i(1) & \cdots & \sigma_i(n)
\end{pmatrix}. 
\]
Then, the type $p$ is isolated by the following $\mathcal{L}_X$-formula:
    \[\label{formula.Ip}\tag{$I_p$}(\exists ~k_1,\dots,k_n \in X) ~ \bigwedge_{\substack{i,j\leq n\\ i\neq j}} k_i\neq k_j   ~\wedge x_i= \begin{pmatrix}
1 & \cdots & n \\
k_{\sigma_i(1)} & \cdots & k_{\sigma_i(n)}
\end{pmatrix}.\]
    We claim that the formula above is equivalent to an $h$-formula in the language of groups. 
    We use the following facts: 
    \begin{itemize}
        \item Every existential quantifier over $X$ can be replaced by an existential quantifier over the group: 
        \[(\exists k\in X) ~ \phi(X)  \Leftrightarrow \exists \sigma_k,\tau_K  ~ (\sigma_k,\tau_k)\in P \wedge \phi(\pi_X(\sigma_k,\tau_k))\]
        \item We have the equivalence 
        \[\psi( \nu , \pi_X((\sigma_k,\tau_k)), \pi_X((\sigma_h,\tau_h)) ) \Leftrightarrow \nu \in C_2\cup\{e\} \wedge (\nu\sigma_k\nu,\nu\sigma_h\nu)E(\sigma_h,\tau_h). \]
        \item for pairs $(\sigma_h,\tau_h)$ and $(\sigma_k,\tau_k)$ in $P$, the relation   $\pi_X((\sigma_k,\tau_k)) \neq  \pi_X((\sigma_h,\tau_h)) $
        can be expressed in the language of group as follows:
        \[ (\forall \nu)~ [(\nu\sigma_k\nu,\nu\sigma_h\nu)E(\sigma_h,\tau_h)\rightarrow \nu\neq e]. \]
    \end{itemize}
    Using these facts, as well as the facts that the set $P$ and  the relations $E$ and $x\neq e$ are $h$-definable, we can deduce that (\ref{formula.Ip}) can be expressed by an $h$-formula $\phi_p$ in the language of groups. 
\end{proof}

With $\phi_p$ as provided by  \Cref{Fact:symmetricgroups} (\ref{f.sym6}), we obtain the following amusing corollary which is certainly folklore but we could not find it in the literature. 

\begin{corollary}
    Let $n> 6$. The parameter-free type of a permutation $a\in \Sym_n$ is uniformly isolated by the $h$-formula $\phi_p$ in the theory of all $\Sym_N$ for $N\geq n$. \qed 
\end{corollary}

We may finally give, for all $n\geq4$ and $n\neq 6$, an interpretable language where reduced power of the symmetric group $\Sym_n$ eliminates quantifiers.
We denote by $\mathcal{L}_{n}^+$ the $2$-sorted language consisting of the following:
	\begin{itemize}
		\item The language of groups, $(G, \cdot)$. 
		\item The language of Boolean algebras $(\Bool{I}{I}, \subseteq)$ 
		\item $\{\supp_{\phi_{p}}: G^{\vert k \vert}  \rightarrow \Bool{I}{I}: p(x_1,\dots ,x_k)$ is an $\Sym_{n}$-type$\}$,      where $\phi_{p}(x_1,\dots, x_k)$ is as in \Cref{Fact:symmetricgroups} \eqref{f.sym6}.
	\end{itemize}

\begin{corollary}\label{lem:cyclesformula}
    Fix an $n\geq 4$, $n \neq 6$, an index set $\mathbb{I}$, and an ideal on $\mathcal{I}$ on $\mathbb{I}$. Let $G:=\prod_{\Ideal{I}}G_i$ with $G_i= \Sym_n$ for every $i\in \Index{I}$. 
	Then $(G,\mathcal{L}_{n}^+)$ is an interpretable expansion of the group $(G,\cdot)$, which eliminates quantifiers relative to $(\Bool{I}{I}, \subseteq)$.
\end{corollary}

\begin{proof}
    In $\Sym_n$, every formula $\phi(\bar{x})$ is equivalent to a Boolean combination of $h$-formulas of the form $\phi_{p}(\bar{x})$ for some types $p$,  i.e., $\{\phi_{p}(\bar{x}) \ : \ p\in \mathcal{S}(\Sym_{n})\}$ is a fundamental set of satisfiable $h$-formulas in $\Sym_n$. Since $G$ recognizes coordinates, it interprets the support function by \Cref{T.eq.groups}. It follows from \Cref{C.QuestionDecomp} that the language $\mathcal{L}_{n}^+$ is interpretable in the language of groups and that  $G$ eliminates quantifiers relative to $\Bool{I}{I}$ in this language.
\end{proof}
\begin{remark}
    \begin{itemize}
        \item One could slightly improve the above corollary and, for a fixed $N$, eliminate quantifiers in a reduced product $G \coloneq \prod_\Ideal{I} G_i$, where $G_i \in \{\Sym_4,\Sym_5,\Sym_7,\Sym_8,\dots, \Sym_N\}$. 
        \item We don't know a language where an arbitrary reduced products of symmetric groups among $\Sym_n$, for $n\geq 4$, $n\neq 6$,  eliminate quantifiers relative to the corresponding Boolean algebra. This would require a uniform description of definable sets in all symmetric groups, which our analysis  does not provide. For example, we don't know whether the formula $x^2=e$ can be described uniformly as a Boolean combination of $\phi_p$.
         \end{itemize}
\end{remark}

We proceed to analyze quantifier elimination in reduced powers of  $\Sym_3$. For $p\in \{2,3\}$, let $\supp_p$ denote the function $\supp_{x^p=e}: \mathcal{G}\rightarrow \Bool{I}{I}$. Consider the language $\mathcal{L}^{\star} := \{\cdot,  (\Bool{I}{I},\subseteq), \supp_{2},\supp_{3} \}$.

\begin{proposition}\label{P.prodSym3}
	We use a back-and-forth argument with patching.
Let $\mathbb{I}$ be an index set, and $\mathcal{I}$ be an ideal on $\mathbb{I}$. Let $G:=\prod_{\Ideal{I}}\Sym_3$. Then $(G,\mathcal{L}^{\star})$ is an interpretable expansion of the group $(G,\cdot)$, which eliminates quantifiers relative to $(\Bool{I}{I}, \subseteq)$.

\end{proposition}

The proof will be given below. 
Observe that, since every nontrivial element of $\Sym_3$ is either a $2$- or a $3$-cycle, we have $\supp(x)= \supp_2(x)\sqcup \supp_3(x)$.

\begin{lemma}[Patching]
	The following can be expressed by a first-order sentence in $\mathcal{L}^{\star}$-theory of $\mathcal{G}$:
	For all $A,B \in \Bool{I}{I}$ such that $A\cap B = \emptyset$ and $a,b \in G$, there exists $c$ such that $c\restriction{A}=a\restriction{A}$ and $c\restriction{B}=b\restriction{B}$.
\end{lemma}
\begin{proof}
	This condition can be written as follows: 
	\[
	(\exists c\in \mathcal{G}) A^\complement \supseteq \supp(c\cdot a^{-1}) \wedge  B^\complement \supseteq \supp(c\cdot b^{-1}).\qedhere
	\]
\end{proof}

\begin{proof}[Proof of \Cref{P.prodSym3}]
	
	We prove the statement via the standard semantic quantifier elimination argument (see e.g. \cite[Paragraph 2.27]{Cha11}). For convenience, we denote by $\mathcal{P}$ the sort $\Bool{I}{I}$. 
    To prove quantifier elimination relative to $\mathcal{P}$, consider:
    \begin{itemize}
        \item $\mathcal{M}=(\mathcal{G}_\mathcal{M},\mathcal{P}_\mathcal{M})$ and $\mathcal{N}=(\mathcal{G}_\mathcal{N},\mathcal{P}_\mathcal{N})$, two models of $\Th(\mathcal{G})$, such that $\mathcal{N}$ is $\aleph_0$-saturated,
        \item two finitely generated (and therefore finite) substructures $\mathcal{A}=(\mathcal{G}_A,\mathcal{P}_A) \subseteq \mathcal{M}$ and $\mathcal{B}=(\mathcal{G}_B,\mathcal{P}_B) \subseteq \mathcal{N}$,
        \item a partial isomorphism 
    \[
    f=(f_{\mathcal{G}},f_\mathcal{P}):\mathcal{A}\rightarrow \mathcal{B},
    \] 
    such that  $f_\mathcal{P}$ is elementary. 
    \end{itemize}
    We need to show that, for all $a\in \mathcal{M}$, we can extend $f$ to a partial isomorphism $\tilde{f}=(\tilde{f}_{\mathcal{G}},\tilde{f}_\mathcal{P})$ with $\tilde{f}_\mathcal{P}$ elementary, and with domain containing $a$.
	
	We can extend $f_{\mathcal{P}}$ to a full embedding $g: \mathcal{P}_\mathcal{M} \rightarrow \mathcal{P}_{\mathcal{N}}$.
	Notice that the sort $\mathcal{P}$ is closed: there are no function symbol in the language from $\mathcal{P}$ to $\mathcal{M}$. It follows therefore automatically that $f \cup  g: (\mathcal{G}_A,\mathcal{P}_\mathcal{M}) \rightarrow (\mathcal{G}_B,\mathcal{P}_\mathcal{N})$ is a partial isomorphism.
	
	Consider $a\in \mathcal{G}_{\mathcal{M}} \setminus \mathcal{G}_A$. We will give a concrete partition $I_1\sqcup \cdots \sqcup I_7$ of the top element of $\mathcal{P}_\mathcal{M}$, $I_i\in \mathcal{P}_\mathcal{M}$. For $i\leq 7$, we will denote by $a_i$ the element in of $\mathcal{G}_{\mathcal{M}}$ such that:
    \begin{itemize}
        \item $\supp(a_i) \subseteq I_i$,
        \item $\supp(a^{-1}\cdot a_i) \subseteq I_i^\complement$.
    \end{itemize}
    (we can identify $a_i$ with $a\restriction I_i$).
    Then, for each $i\leq 7$, we will find a correct answer ${b}_i$ for $a_i$: for all formulas $\phi(x,a',I) \in \qftp(a_i/\mathcal{A})$, we have $\mathcal{N}\models \phi({b}_i,f(a'),g(I)) $. Then we can conclude by patching that there is a good answer $b$ for $a$ i.e. such that $f(\qftp(a/\mathcal{A})) = \qftp(b/\mathcal{B})$.

	\subsubsection*{Case 1:} Consider $I_1=\supp(\mathcal{G}_A)^\complement \cap \supp_2(a)$ where $\supp(\mathcal{G}_A) = \bigcup_{a'\in \mathcal{G}_A} \supp(a')$. Set $f({a}_1)$ to be any $2$-torsion element $b_1$ of $\mathcal{N}$ such that $\supp(b) = g(I_1)$.

	\subsubsection*{Case 2:} 
	Consider $I_2=\supp(\mathcal{G}_A)^\complement \cap \supp_3(a)$. Set $f({a}_2)$ any $3$-torsion element $b_2$ of $\mathcal{N}$ such that $\supp(b_2)=g(I_2)$.

	\subsubsection*{Case 3:}
	For $a'\in \mathcal{G}_A$, $I_{3,a'}= \supp(a'a^{-1})^c$. This is the part of $\mathcal{P}_\mathcal{M}$ where $a'$ and $a$ coincide. Set ${b}_{3,a'}$ of support $g(I_{3,a'})$ such that $\supp(f(a')^{-1}\cdot {b}_{3,a'}) \subseteq g(I_{3,a'})^\complement$ ( 
    i.e. ${b}_{3,a'}$ can be identified with ${f(a')}\restriction{g(I_{3,a'})}$). Set $I_3= \bigcup_{a'\in \mathcal{G}_A} I_{3,a'}$. 
	 By saturation, compactness and the patching property, we find ${b}_3$ in $\mathcal{N}$ such that for all $a'\in \mathcal{G}_A$, ${b}_3\restriction g(I_{3,a'})={a'}\restriction{g(I_{3,a'})}$.

	\subsubsection*{Case 4:} 
	Let $a' \in \mathcal{G}_A$. Set $I_{4,a'}=\supp_3(a) \cap I_3^\complement \cap \supp_3(a')$. This is a part of $\mathcal{P}$ where $a'$ and $a$ are both $3$-cycle, but do not coincide. In particular, we have ${a}\restriction{I_{4,a'}}= ({a'}\restriction{I_{4,a'}})^2$ and we need to have ${b}\restriction{g(I_{4,a'})}= {f(a')^2}\restriction{g(I_{4,a'})}$ for all such $a'$.
    Set $I_4= \bigcup_{a'\in \mathcal{G}_A} I_{4,a'}$. By saturation, compactness and the patching property, we find ${b}_4$ in $\mathcal{N}$ such that for all $a'\in \mathcal{G}_A$, ${b}_4\restriction {g(I_{4,a'})}={f(a')^2}\restriction{g(I_{4,a'})}$.
	
	\subsubsection*{Case 5:} 
	Let $a',a'' \in \mathcal{G}_A$. Set $I_{5,a',a''}=\supp_2{a} \cap I_3^\complement \cap \supp_2(a') \cap \supp_3(a'') \cap (\supp({a''}a'a))^\complement$.
	This is the part of $\mathcal{P}$ where $a$ is a $2$-cycle, $a'$ is another $2$-cycle, $a''$ is a $3$-cycle  and  $a$ is equal to $a''a'$ . 
    We need that ${b}\restriction{g(I_{5,a',a''})}= {f(a'')}\restriction{g(I_{5,a',a''})}\cdot {f(a')}\restriction{g(I_{5,a',a''})}.$
	
	Set $I_5= \bigcup_{a',a''\in \mathcal{G}_A} I_{5,a',a''}$. By saturation, compactness and patching, we find ${b}_5$ in $\mathcal{N}$ with support included in $g(I_5)$ such that for all $a',a''\in \mathcal{G}_A$, \[ {b}_5\restriction{g(I_{5,a',a''})}={f(a'')}\restriction{g(I_{5,a',a''})}\cdot {f(a')}\restriction{g(I_{5,a',a''})}.\]
	
	\subsubsection*{Case 6:} 
	Let $a',a'' \in \mathcal{G}_A$. Set $I_{6,a',a''}=\supp_2(a) \cap I_3^\complement \cap \supp_2(a') \cap \supp_2(a'') \cap \supp(aa'a''a')^\complement$.
	This is the part of $\mathcal{P}$ where $a$ is a  $2$-cycle, and $a',a''\in \mathcal{G}_A$ are the other two $2$-cycles and $a$ coincides with $a'a''a'^{-1}=a'a''a'$.
	We need to have: 
    \[{b}\restriction{g(I_{6,a',a''})}= f(a')f(a'')f(a')\restriction{g(I_{6,a',a''})}. \]
    for all such $a',a''$.
	Set $I_6= \bigcup_{a',a''\in \mathcal{G}_A} I_{6,a',a''}$. By saturation, compactness and the patching property, we find ${b}_6$ in $\mathcal{N}$ with support included in $g(I_6)$ such that for all $a',a''\in \mathcal{G}_A$, 
    \[ {b}_6\restriction{g(I_{6,a',a''})}=f(a')f(a'')f(a') \restriction {g(I_{6,a',a''})}.\]
	
	\subsubsection*{Case 7:}
	Let $a' \in \mathcal{G}_A$. Set $I_{7,a'}=\supp_2{a} \cap I_3^\complement \cap \supp_2(a') \cap (I_6)^\complement$. This is the part of the support where $a$ is a $2$-cycle, $a'\in \mathcal{G}_A$ is another $2$-cycle but $a'aa'$ is not in $\mathcal{G}_A$.
	Then ${b}\restriction{g(I_{7,a'})}$ needs (and only needs) to be a two torsion element $\alpha_{a'}$ with support $g(I_{7,a'})$ which is nowhere equal element to $f(a')$ on $g(I_{7,a'})$.  
	Set $I_7= \bigcup_{a'\in \mathcal{G}_A} I_{7,a'}$. By saturation, compactness and patching, we find ${b}_7$ in $\mathcal{N}$ such that for all $a'\in \mathcal{G}_A$, ${b}_7 \restriction {g(I_{7,a'})}=\alpha_{a'} \restriction{g(I_{7,a'})}$.
    
\vspace{.1in}

    Set $b$ to be the unique element such that $b\restriction{g(I_i)}=b_i\restriction{g(I_i)}$ for $i\leq 7$. 
    One can observe easily that $\bigcup I_i$ is a partition of $\mathcal{P}_\mathcal{M}$ definable over $a$ and $\mathcal{G}_{A}$. Since we have for all $i\leq 7$, $f(\qftp(a_i/\mathcal{A})) = \qftp(b_i/\mathcal{B})$, we also have $f(\qftp(a/\mathcal{A})) = \qftp(b/\mathcal{B})$. Therefore $f\cup \{(a,b)\}$ extends the partial isomorphism and this concludes the proof.  
\end{proof}

\section{Groups not recognizing coordinates}
\label{S.GroupsNotRecognizing}
\subsection{General criteria for not recognizing coordinates}
\label{S.NotRecognizing}

\begin{theorem}\label{Fact:construction}
	Suppose that $\fC$ is a class of groups that satisfies one of the following two conditions. 
	\begin{enumerate}
		\item \label{1.construction} It contains groups $G$ and $H$ and a nontrivial homomorphism $f\colon G\to Z(H)$.
		\item \label{2.construction} It contains a group that  is decomposable as a direct product of nontrivial groups. 
	\end{enumerate}  
	Then $\fC$ does not recognize coordinates
\end{theorem}

The converse of this theorem does not hold. Namely, there is a family $\fD$ of groups that  fails both \eqref{1.construction} and \eqref{2.construction} of Theorem~\ref{Fact:construction} but it does not recognize coordinates (see  Theorem~\ref{T.Limiting}). 
The sufficiency of the second condition is obvious, and the 
proof  of the sufficiency of the first uses the following obvious lemma. 

\begin{lemma}\label{L.size} Suppose that $T$ recognizes coordinates. Let $(\mathcal{M}_i)_{i \in I}$ and $(\mathcal{N}_j)_{j \in J}$ be indexed families of models of $T$. Consider $\theta: \prod \mathcal{M}_i \to \prod \mathcal{N}_j$ an isomorphism between (non-reduced) products. If $\vert \supp(a) \vert$ is finite, then $|\supp(\theta(a))| = |\supp(a)|$. 
\end{lemma}
\begin{proof}
    By \Cref{T.InterpretingSupport}, the support map is interpretable. In particular, the set $\{a \mid |\supp(a)| =n\}$ is definable and thus preserved by isomorphisms.
\end{proof}

\begin{proof}[Proof of Theorem~\ref{Fact:construction}] 
For $(1)$, consider the map $\tau_{f}: G \times H \to G \times H$ defined via $\tau_{f}((a,b)) = (a,f(a)b)$. Then $\tau_{f}$ is an automorphism of $G \times H$. 
	\begin{enumerate}
		\item Homomorphism: Notice that 
		\begin{align*}
			\tau_{f}((a,b)(c,d)) &= \tau_{f}((ac,bd)) = (ac,f(ac)bd)\\ &= (ac,f(a)bf(c)d) = (a,f(a)b)(c,f(c)d) \\ 
			&= \tau_{f}((a,b)) \tau_{f}((c,d)). 
		\end{align*}
		\item Injective: Suppose that $\tau_{f}((a,b)) = \bar{e}$. Then $(a,af(b)) = (e,e)$. Hence $a = e$ and so $e = f(a)b = f(e)b = b$. It follows that $\tau_{f}$ is injective.
		\item Surjective: Fix $(a,b) \in G \times G$. Then 
		\begin{equation*}
			\tau_{f}(a,f(a)^{-1}b) = (a,b). 
		\end{equation*}
	\end{enumerate} 
	In particular, this allows us to build an automorphism $\Sigma_{f}: \prod_{i\in \bbN} (G\times H)$ via $\Sigma_{f}((a_i,b_i)_{i \in \mathbb{N}}) = (a_i,f(a_i)b_i)_{i\in \bbN}$. Since $f$ is nontrivial, we can find some $a_* \in G$ such that $f(a_*) \neq e$. Then with $a_i=a_*$ for all $i$ and arbitrary $b\in G$ we have 
	\begin{equation*}
		\Sigma_{f}((a_i,b_i)_i)) \neq (a_i,b_i)_{i\in \bbN}. 
	\end{equation*}
	Moreover,  		$\Sigma_f$ is an automorphism of $\prod_{i\in \bbN} (G\times H)$ and its restriction to  $\bigoplus_{i\in \bbN} (G\times H)$ is an automorphism of $\bigoplus_{i\in \bbN} (G\times H)$. It therefore lifts an automorphism of the quotient $\prod_{\Fin} (G\times H)$. 

For $(2)$, the argument is similar to $(1)$. Notice that if $G$ is a decomposable group, then $G \cong H_1 \times H_2$ for nontrivial groups $H_1$ and $H_2$. Consider the map $\rho: G \times G \to G \times G$ defined via $\rho((h_1,h_2),(h_3,h_4)) = ((h_1, h_4), (h_3,h_2))$ where $(h_1,h_2), (h_3,h_4) \in H_1 \times H_2$. Then, 
\begin{align*}
    2 &= |\supp((h_1,e),(e,h_4))| \\ 
    &> |\supp(\rho((h_1,e),(e,h_4)))| = |\supp((h_1,h_4),(e,e))| = 1.  
\end{align*}
As consequence of \Cref{L.size}, $\mathfrak{C}$ cannot recognize coordinates. 
\end{proof}

\subsection{Classes of groups that do not recognize coordinates}

\begin{theorem}\label{T.eDNR} 
 Any class of groups that contains some of the following does not recognize coordinates. 
\begin{enumerate}[label=(\alph*)]
	\item  Any product of nontrivial groups. 
	\item   Groups $G$ and $H$ such that $G$ admits a nontrivial homomorphism into the center of $H$. 
	\item  Group $\GL(n,F)$ for $n\geq 2$ and any field $F$.
	\item   Group $Q_{8} = \langle -1,i,j,k: (-1)^{2} = e, i^{2} = j^{2} = k^{2} =ijk = -1  \rangle$. 
	\item  The dihedral group $D_{2n}$  of the $2n$-gon for $n\geq 1$. 
	\item  Any nilpotent group. 
	\item  Any nontrivial graph product $\Gamma\cG$ such that the complement graph $\bar\Gamma$ is not connected (see \S\ref{S.GraphProducts}). 
	\item  Any class of groups that contains $\Sym_3$ and $\SL(2,5)$.
\end{enumerate} 
\end{theorem}

\begin{proof}
\ref{DNR.abelian} follows by Theorem~\ref{Fact:construction}. 


\ref{DNR.decomposable} This is Theorem~\ref{Fact:construction} \eqref{2.construction}. 

\ref{DNR.GLn} Compose the determinant map with the map that sends scalars to scalar matrices. 

\ref{DNR.Q8}
Recall that $Z(Q_{8}) = \{\pm 1\} \cong \mathbb{Z}/2\mathbb{Z}$. 
For any $l\in \{i,j,k\}$, we have a homomorphism $\alpha_{l}: Q_{8} \to Z(Q_8) $ via
\begin{equation*}
	\alpha_{l}(d)=\begin{cases}
		\begin{array}{cc}
			1 & d=l, -l, 1,-1\\
			-1 & else.
	\end{array}\end{cases}
\end{equation*} 
Note that $Q_{8}$ cannot be written as a nontrivial semi-direct product. 

\ref{DNR.D2n} We have $D_{2n} = \langle r,s : r^{2n} = s^{2} = e, srs = r^{-1} \rangle$. Then if $n$ is even, then the center is $\{e, r^{n/2}\}$. Each element can be written as $s^{\epsilon}r^{k}$ where $\epsilon \in \{0,1\}$ and $k \in \{0,\dots,n-1\}$. Then the map $f:D_{n} \to \mathbb{Z}/2\mathbb{Z}$ via $f(s^{\epsilon}r^{k}) = \epsilon + k \mod 2$ is a homomorphism from $D_{n}$ into its center. 

We remark that $D_{n}$ can be written as a semidirect product. $D_{n} \cong Z_n \rtimes Z_{2}$. 

\ref{DNR.nilpotent} If $G$ is nilpotent, then  Lemma~\ref{L.nilpotent0} below implies that there is a nontrivial homomorphism from $G$ into its center. 

\ref{DNR.Graph} This is the easier half of Theorem~\ref{T.GraphProduct}. 

\ref{DNR.S3SL25} The group $\Sym_3$ has $\bbZ/2\bbZ$ as a quotient, and $\bbZ/2\bbZ$ is the center of $\SL(2,5)$. Therefore Theorem~\ref{Fact:construction} implies that  the class $\{\Sym_3,\SL(2,5)\}$ does not recognize coordinates.
\end{proof}

Regarding \ref{DNR.GLn} in Theorem~\ref{T.example:DNR}, the following may be worth pointing out (e.g., see \cite{ycormathflow}).

\begin{lemma} If $F$ is a field in which every element has a unique $n$-th root then $\GL(n,F)$ is decomposable. In particular $\GL(3,\mathbb{R})$ is decomposable and also $\GL(p,\mathbb{F}^{\alg}_p)$ where $\mathbb{F}^{\alg}_p$ is an algebraically closed field of characteristic~$p$. 
\end{lemma}

\begin{proof} An isomorphism $\GL(n,F)\cong \SL(n,F)\times F^\times$ is given by the map 	$a\mapsto (\det(a)^{-1/n} a, \det(a))$. 
\end{proof}

The following is well-known but maybe not the easiest to find. We include a proof for the reader's convenience. 

\begin{lemma}\label{L.nilpotent0} If $G$ is nilpotent, then there is a non-identity homomorphism from $G$ into its center.
\end{lemma}

\begin{proof} Let $n+1$ be the nilpotency class of $G$. Thus we have $G_0=G$, $G_{k+1}=[G_k,G]$ for $k\leq n$,  so that $G_n$ is nontrivial abelian, including in the center, and $G_{n+1}=\{e\}$. 
If $G$ is abelian then the assertion is trivial, hence we may assume $n\geq 1$.  We will prove that for every $c\in G_{n-1}$ the mapping $x\mapsto [x,c]$ from $G$ into $G$ is a homomorphism. 
Let $c\in G_{n-1}$. Then for every $x\in G$ the commutator $[x,c]$ belongs to the center.
Fix $c,x$, and $y$ in~$G$. Then, repeatedly using the fact that all commutators of the form  $[z,c]$ commute,    we have the following (extra brackets inserted for readability)
\begin{align*}
	[x,c][y,c]&= xcx^{-1} c^{-1} (ycy^{-1} c^{-1})\\
	&= x (ycy^{-1} c^{-1} )cx^{-1} c^{-1}\\
	&= x ycy^{-1} x^{-1} c^{-1}\\
	&=[xy,c].
\end{align*}
Since $G$ is nonabelian, we can choose a non-central $c$. By the above, $x\mapsto [x,c]$ is a group homomorphism from $G$ into $Z(G)$.  Since $G_n$ is nontrivial and it is generated by commutators $[x,c]$ for $x\in G$ and $c\in G_{n-1}$, we can choose $c\in G_{n-1}$ so that the range of the homomorphism $x\mapsto [x,c]$ is nontrivial.  
\end{proof} 

Finally, we treat the reader to a \emph{surprise} example. 

\begin{example} For any $n \geq 2$, neither the \emph{Artin braid group on $n$ strands} (denoted $B_n$) nor the \emph{pure braid group on $n$ strands} (denoted $P_n$) recognize coordinates. For $n \geq 2$, both $B_n$ and $P_n$ are indecomposable (i.e., see \cite[Proposition 4.2]{paris2004artin}), yet they both admit nontrivial maps to their respective centers. We refer the reader to \cite{kassel2008braid} as a basic reference. For $n \geq 2$, the group $B_n$ is defined as follows:
\[
B_n = \left\langle \sigma_1, \dots, \sigma_{n-1} \,\middle|\,
\begin{array}{ll}
\sigma_i \sigma_j = \sigma_j \sigma_i & \text{for } |i - j| \geq 2, \\
\sigma_i \sigma_{i+1} \sigma_i = \sigma_{i+1} \sigma_i \sigma_{i+1} & \text{for } 1 \leq i \leq n - 2
\end{array}
\right\rangle.
\]
The half-twist $\Delta$ in $B_{n}$ is defined as,
\[
\Delta = (\sigma_1)(\sigma_2 \sigma_1)(\sigma_3 \sigma_2 \sigma_1) \cdots (\sigma_{n-1} \cdots \sigma_1),
\]
and the center of $B_{n}$ is precisely $\langle \Delta^{2} \rangle \cong \mathbb{Z}$. On the other hand, $B_{n}/[B_n,B_n] \cong \mathbb{Z}$. Hence, $B_{n}$ admits a nontrivial homomorphism to its center and so it does not recognize coordinates.

Additionally, the pure braid group on $n$ strands is the kernel of the surjective homomorphism from $B_{n}$ onto $\Sym_{n}$ generated via $\sigma_i \to (i,i+1)$. For $n \geq 2$, the group $P_{n}$ has nontrivial center, again $Z(P_n) = \langle \Delta^{2} \rangle \cong \mathbb{Z}$. By \cite[Corollary 1.20]{kassel2008braid}, 
\[
    P_{n} /[P_{n},P_{n}] \cong \mathbb{Z}^{\binom{n}{2}}.
\]
Hence $P_{n}$ also admits a nontrivial homomorphism to its center and thus does not recognize coordinates. 
\end{example}

\section{Other structures}\label{S.Other}
As mentioned in the introduction, in \cite[\S 2.2]{de2023trivial} it was pointed out that 
    from the model-theoretic point of view, a \emph{morally} satisfactory proof that a theory recognizes coordinates would proceed by exhibiting a copy of $\cP(\bbN)/\cI$ as well as the projections $\pi_S$, for $S\in \Bool {\bbN} \cI$ inside every reduced product $\prod_{\Fin} \cM_n$ of models of $T$.  In this sense the proofs that linear orders and sufficiently random graphs recognize coordinates (\cite[Proposition~2.7]{de2023trivial}) are unsatisfactory. 
    We therefore give below new proofs, using \Cref{T.A} (3) and explicit $h$-formulas,  that sufficiently random graphs and linear orders recognizes coordinates. We will restrict to linear orders with no maximal element for simplicity.
    First we study coordinate recognition in a particular a commutative magma.

\subsection{Rock Paper Scissors } 
Let $\mathfrak{M}= (\{R,P,S\},\cdot)$ be the commutative magma on the three elements set $\{R,P,S\}$ with the following operation table:
\begin{center}
	$\begin{array}{ | c || c | c | c |}
		\hline
		\cdot & R & P & S  \\
		\hline
		\hline
		R & R & P & R \\ 
		\hline
		P & P & P & S \\ 
		\hline
		S & R & S & S  \\ 
		\hline
	\end{array}$
\end{center}

\begin{proposition}
	Let  $\Sh= \prod_\Ideal{I} \mathfrak{M}$ be a reduced product of the magma $\mathfrak{M}$. Then
	$\Sh$ interprets parameter-freely the Boolean algebra $\Bool{I}{I}$  and the relative support function 
	\[\supp_=: \Sh^2\rightarrow \mathcal{P}(\mathbb{N})/I,   (a,a') \mapsto  \big[\{i\in \Index{I} \mid a_i=a_i'\}\big]_\Ideal{I} .\] 
\end{proposition}

It is relatively easy to define various copies of Boolean algebra $\Bool{I}{I}$. With a small trick, we can interprets the support function. 
\begin{proof}
	To show that $\Sh$ interprets the support function $\supp_=$, it is enough to show by \Cref{T.eq} that the formula $x=t \rightarrow y=t$ is equivalent to an $h$-formula in $\mathfrak{M}$. 
	Fix $t\in \mathfrak{M}$. We denote by $L(t)$ the element $s$ losing against $t$, i.e. such that $s\neq t$ and $s\cdot t = t$. The following shows that $t\mapsto L(t)$ is defined by the $h$-formula:
	\[L(t)=t' \quad  \Leftrightarrow \quad t'\cdot t = t \wedge \forall s' ~s'\cdot t = t \rightarrow t'\cdot s' = s'.\]
	
	We see that $x=t \rightarrow y=t$ is equivalent to 
	\[\tag{*}\label{eq.RPS}(x\cdot L(t))\cdot (y\cdot L(t))=(y\cdot L(t))\]
	Indeed, if  (\ref{eq.RPS}) holds and $x=t$, then we have \begin{align*}
		&\quad (x\cdot L(t))\cdot (y\cdot L(t))=(y\cdot L(t))\\
		\Leftrightarrow &\quad (t\cdot L(t))\cdot (y\cdot L(t))=(y\cdot L(t))\\
		\Leftrightarrow& \quad t\cdot (y\cdot L(t))=(y\cdot L(t)).\\
	\end{align*}
	Then, since  $y\cdot L(t)$ is either $t$ or $L(t)$, we have:
	\[ \quad t\cdot (y\cdot L(t))=(y\cdot L(t))\Leftrightarrow t=(y\cdot L(t)) \Leftrightarrow t=y.\]
	Conversely, assume $x=t \rightarrow y=t$, and we need to show that (\ref{eq.RPS}) holds. If $y=t$, then $(y\cdot L(t))=t$, and since $(x\cdot L(t))$ is either $t$ or $L(t)$, we have
	\[(x\cdot L(t))\cdot (y\cdot L(t))= t = (y\cdot L(t))\]
	and (\ref{eq.RPS}) holds. If $x\neq t$, then $(x\cdot L(t))=L(t)$ and (\ref{eq.RPS}) also clearly holds. 
	
	At the end, we get that the formula $x=t \rightarrow y=t$ is equivalent to the formula:
	\[(\exists s) ~\left[s\cdot t = t \wedge (\forall s') ~(s'\cdot t = t \rightarrow s\cdot s' = s') \wedge (x\cdot s)\cdot (y\cdot s)=(y\cdot s) \right].\]
	
	This is clearly (equivalent to) an $h$-formula since $(\exists s') ~s'\cdot t = t $ always holds.
\end{proof}

By \Cref{T.A}, we have:

\begin{corollary}
	The magma $\mathfrak{M}$ recognizes coordinates.
\end{corollary}

Finally, we notice that every (parameter-free) formulas in $\mathfrak{M}$ are equivalent to a Boolean combination of (quantifier-free) atomic formulas in the language $(\mathfrak{M},\cdot,L)$. It follows that any reduced power $\Sh$ of $\mathfrak{M}$, $\{(\Sh,\cdot),(\Bool{I}{I},\subseteq), \supp_= \}$ is a definable expansion of the magma $
(\Sh,\cdot)$ and by \Cref{C.QuestionDecomp}:

\begin{corollary}
	Any reduced power $\Sh= \prod_\Ideal{I} \mathfrak{M}$ of the magma $\mathfrak{M}$ eliminates quantifiers relative to $\Bool{I}{I}$ in the following interpretable language:
    \[\{(\Sh,\cdot,L),(\Bool{I}{I},\subseteq), \supp_= \}.\]
\end{corollary}

\subsection{Linear orders}

In this paragraph, $\mathcal{L}$ denotes the language $\mathcal{L}=\{<\}$ of orders. Consider $\fC$ the class of linear orders with no larger element, and $\mathcal{N} \coloneqq \prod_\Ideal{I} \mathcal{M}$ a reduced product. 
The relation 
\[\min(x,y)\leq z,\]
can be expressed with an $h$-formula, namely:
\[(\forall w) ~\left[ (w\leq x \wedge w\leq y) \rightarrow w \leq z \right].\]
It is an $h$-formula because for all $x,y$, there is always $w$ such that $(w\leq x \wedge w\leq y)$.
Clearly, $\min(x,y)\geq z$ is equivalent to the $h$-formula $z\leq x \wedge z\leq y$.

	




    
	\begin{proposition}
	    In all structures in $\fC$, the formula  $f=k \rightarrow g=k$ is equivalent to the following $h$-formula:
    \begin{align*}
        \tag{*} \label{eq.LO}(\exists u) &~ \  u>f\wedge u >g \wedge u>k \\
        &\wedge (\forall g') ~ [\min(g',k)=\min(g,g')=\min(g,k)] \rightarrow\\
         &  \quad \qquad \left [(\exists f')~\min(f',k)=\min(f,f')=\min(f,k) \wedge \max(u,f')\geq \max(u,g') \right].  
    \end{align*}
    
	\end{proposition}
    \begin{proof}
            This is an $h$-formula as for all $g,k$,
    \[(\exists g') ~ \min(g',k)=\min(g,g')=\min(g,k)\]
    always holds (we may take $g'=\min(g,k)$).
    We see $f,g,k,u$ as functions from a nonempty set $A$ to a total order $B$ with no top element.
    Assume that $f$ coincides more often with $k$ than $g$ coincides with $k$. An element $g'$ satisfying  $\min(g',k)=\min(g,g')=\min(g,k)$ must be equal to the minimum of $g$ and $k$, except where $g$ and $k$ coincides.  Since $f$ and $k$ coincide more often than $g$ and $k$, then for all such $g'$ we may find an $f'$ such that $\min(f',k)=\min(f,f')=\min(f,k)$, and which is larger than $g'$ on the part where $f$ and $k$ coincide. To express the later, we use an element $u> f,g,k$ and compare $\max(g',u)$ and $\max(f',u)$.
    Conversely, if $g$ and $k$ coincide where $f$ and $k$  don't, then some $g'$ not smaller than $u$, we see that one can't find such a $f'$.
    
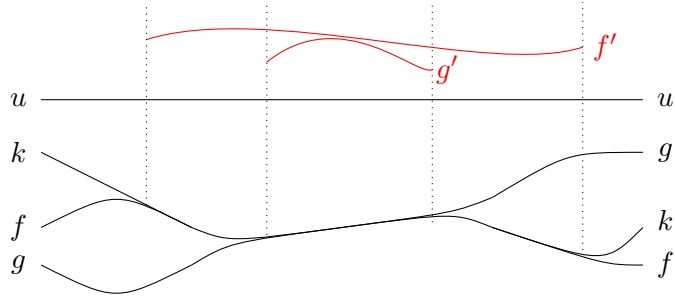
    \begin{center}
        
\begin{tikzpicture}
    \draw (0,0) ..  controls (1,0.5) .. (2,0);
    \draw (0,1) ..  controls (1,0.5) .. (2,0);
    
    \draw (2,0) ..  controls (2.5,-0.2) .. (4,0);
    \draw (2,-0.5) ..  controls (2.5,-0.2) .. (4,0);
    
    \draw (0,-0.5) ..  controls (1,-1) .. (2,-0.5);
    
    \draw (4,0) ..  controls (5.5,+0.2) .. (6,0);
    \draw (4,0) ..  controls (5.5,+0.2) .. (6,0.4);
    
    \draw (6,0.4) ..  controls (7,+1) .. (8,1);
    
    \draw (6,0) ..  controls (7.5,-0.5) .. (8,-0.5);
    \draw (6,0) ..  controls (7.5,-0.5) .. (8,0);

    \draw[red] (1.4,2.5) ..  controls (3,3) and (6,2)  .. (7.2,2.4);
    \draw[red] (7.5,2.4) node{$f'$};

    \draw[red] (3,2.2) ..  controls (4,3) and (5,2)  .. (5.2,2.1);
        \draw[red] (5.4,2.1) node{$g'$};

    \draw (0,1.7) -- (8,1.7);
    \draw[dotted] (1.4,0.3) -- (1.4,3);
    \draw[dotted] (3,-0.1) -- (3,3);
    \draw[dotted] (5.2,0.07) -- (5.2,3);
    \draw[dotted] (7.2,-0.3) -- (7.2,3);
    \draw (8.3,1) node{$g$};
    \draw (8.3,0.1) node{$k$};
    \draw (8.3,-0.5) node{$f$};
    \draw (8.3,1.7) node{$u$};
    \draw (-0.3,-0.5) node{$g$};
    \draw (-0.3,1) node{$k$};
    \draw (-0.3,0) node{$f$};
    \draw (-0.3,1.7) node{$u$};

\end{tikzpicture}

    \end{center}
    Assume now that $\vert A\vert=1$ and identify $B$ with function from $A$ to $B$. The facts above gives the equivalence between $f=k \rightarrow g=k$ and (\ref{eq.LO}).
    \end{proof}

	By \Cref{T.eq} and \Cref{T.A}, it follows in particular:
    
    \begin{corollary}
        Reduced products $\mathcal{M} \coloneq \prod_\Ideal{I}\cM_i$ of total orders with no top element interprets (uniformly) the boolean algebra $\Bool{I}{I}$ and the relative support function $\supp_=: \cM^2 \rightarrow \Bool{I}{I}$, and the class of such orders recognizes coordinates.
    \end{corollary}

We can deduce from Feferman--Vaught theorem a language where a reduced product of dense linear orders (without endpoints) eliminates quantifiers.

\begin{corollary}
	Reduced products $\mathcal{M}:=\prod_{\mathcal{I}} \mathcal{M}_i$ of models of DLO eliminates quantifiers relative to the boolean alegbra $\Bool{I}{I}$ in the interpretable language 
    \[\{(\mathcal{M},\leq) , (\Bool{I}{I},\subseteq), \supp_=, \supp_\leq\}\]
    where for all $a,b\in \mathcal{N}$:
	\[\supp_=(a,b) \coloneqq \big[\{i \mid a_i=b_i\}\big]_\mathcal{I}\]
	and 
	\[\supp_\leq (a,b) \coloneqq \big[\{i \mid a_i\leq b_i\}\big]_\mathcal{I}\]
\end{corollary}
\begin{proof}
	We have quantifier elimination in DLO (the theory of dense linear orderings without the endpoints) in the language of pure orders, therefore all formulas are equivalent to a Boolean combination of formulas of the form $x\geq y$ and $x=y$, which are satisfiable $h$-formulas. We can conclude by \Cref{C.QuestionDecomp}.
\end{proof}

\subsection{Sufficiently random graphs}
In this paragraph, $\mathcal{L}$ denotes the language $\mathcal{L}=\{R\}$ of (undirected and simple) graphs. We recall first the definition of sufficiently random graphs:
\begin{definition}[ {\cite[Definition 2.20]{de2023trivial}}]
    A graph $G$ is sufficiently random if it has at least three vertices and if for all pairwise distinct vertices $u$, $v_1,v_2$, there is a vertex $w$ distinct from $u$ such that $\neg uRw \wedge  v_1Rw \wedge  v_2Rw$ holds. The theory of sufficiently random graphs will be denoted $T_{\textrm{SRG}}$. 
\end{definition}

Notice that the class $\fC$ of models of $T_{\textrm{SRG}}$ is full in the sense of \Cref{Def.Full}.

We claim that the formula $x=z \vee y=z$ is equivalent to 
\begin{align*}
    \phi(x,y,z) \colon ~  (\forall d) ~ dRx \wedge dRy \rightarrow dRz .\end{align*}
Indeed, the direct implication is obvious. To prove the converse, assume that $z\notin \{x,y\}$ holds. Then, by sufficient randomness, there is a vertex $d$ which is connected to $x$ and $y$ and not to $z$. This contradicts the formula $\phi(x,y,z)$ above.  

The formula $\phi(x,y,z)$ is clearly equivalent  to an $h$-formula modulo $T_{\textrm{SRG}}$, as 
$(\exists d) ~ dRx \wedge dRy$ holds in every sufficiently random graph. 

We can now use the following fairly general lemma.  
\begin{lemma}
 Assume that $x=z \vee y=z$ is equivalent to an $h$-formula $\phi(x,y,z)$ modulo some theory $T$. Then $(\forall v) ~ \phi(x,v,z) \rightarrow \phi(y,v,z)$ is 
    an $h$-formula  equivalent modulo $T$ to $x=z \rightarrow y=z$.
\end{lemma} 
\begin{proof}
    The fact that $\forall v ~ \phi(x,v,z) \rightarrow \phi(y,v,z)$ is (equivalent to) an $h$-formula is clear, as  for all $x$ and $z$, the sentence $(\exists v) ~\phi(x,v,z)$ holds in every sufficiently random graph.
To prove the converse implication, assume that $x=z \rightarrow y=z$ holds. If $x\neq z$, then for every vertex $v$, $\phi(x,v,z)$ implies $z=v$ and therefore  it implies $\phi(y,v,z)$. We may therefore assume $x=z$, which implies $x=y=z$.  In this case, both $\phi(x,v,z) $ and $\phi(y,v,z) $ hold for all $v$, proving the converse implication. 
It remains to prove that $(\forall v) ~ \phi(x,v,z) \rightarrow \phi(y,v,z)$ implies $x=z \rightarrow y=z$. Assuming in addition $x=z$,  it suffices to show that $y=z$. Towards this, fix any $v\neq z$. Then both $\phi(x,v,z)$ and $\phi(y,v,z)$ hold, and the latter implies $y=z$.
\end{proof}
By \Cref{T.eq} and \Cref{T.A}, we have the following corollary :
\begin{corollary}
     Reduced products $\cM \coloneq \prod_\Ideal{I}\cM_i$ of sufficiently random graphs interprets (uniformly) the Boolean algebra $\Bool{I}{I}$ and the support function $\supp_=: \cM \rightarrow \Bool{I}{I}$, and the class such graphs recognizes coordinates.
\end{corollary}

Now, let us consider $T_{RG}$, the theory of random graphs. It is of course a completion of  $T_{\textrm{SRG}}$ and it eliminates quantifiers in the language $\mathcal{L}=\{R\}$. By \Cref{C.QuestionDecomp}, we have:
\begin{corollary}
    Reduced products $\mathcal{G}=\prod_\mathcal{I}\mathcal{G}_i$ of random graphs eliminates quantifiers relative to $\Bool{I}{I}$ in the interpretable language 
    \[\{(\mathcal{G},R),\Bool{I}{I}, \supp_= , \supp_R\},\]

        where 
    \[\supp_=(a,b) \coloneq \big[\{i \mid a_i=b_i\}\big]_\mathcal{I}\]
    and 
    \[\supp_R(a,b) \coloneq \big[\{i \mid a_iRb_i\}\big]_\mathcal{I}.
    \]
    
\end{corollary}

\subsection{Connected rings}
In this paragraph, we highlight some connections between coordinate recognition and the work of D'Aquino and Macintyre in~\cite{DAM23}, and of Derakhshan and Macintyre (see e.g.  \cite{DM22}).
Recall that a ring $R$ is connected if $0$ and $1$ are the only idempotent elements in $R$ and $0 \neq 1$. In this paragraph, we denote $\fC$ the class of connected rings.

\begin{lemma}
Modulo the theory of connected rings, the formula  \[x=0\rightarrow y=0\]
is equivalent to the $h$-formula 
\[(\forall z) [(z^2=z\wedge zx=0 )\rightarrow zy=0].\]
\end{lemma}
We leave the proof to the reader. By \Cref{T.eq.groups}, we have:

\begin{corollary}
    Reduced products $\prod_{\Ideal{I}} R_i$ of connected rings interpret (uniformly) the Boolean algebra $\Bool{I}{I}$ and the support map 
    \[\supp: \prod_{\Ideal{I}} R_i \rightarrow \Bool{I}{I}; ~a \mapsto \big[\{~ i \mid a_i \neq 0\} \big]_\Ideal{I},\]
    and the class of connected rings recognizes coordinates.
\end{corollary}

This corollary, at least in its version for products $R \coloneq \prod_{\Index{I}} R_i$, is explicit in \cite{DAM23} and \cite{DM22}, once the Boolean algebra $\cP(\Index{I})$ is identified as the set of idempotents of $R$. This plays an important role, for example, in \cite[Corollary 7.2]{DM22} -- which states a quantifier elimination result for the rings of adeles -- and in \cite[Theorem 2.10]{DAM23} -- which gives a list of five axioms for the class of products of connected rings. 
Connected rings were discussed in \cite[\S 4.1]{de2023trivial} where they were called rings with no nontrivial idempotents,  and the result relevant for recognising coordinates was attributed to S. Burris (\cite{burris1979sheaf}). The authors of \cite{de2023trivial} were not aware of \cite{DAM23} and \cite{DM22}.
Axiomatisation of classes of (reduced) products will not be investigated in this paper.

\section{Limiting examples}\label{S.Limiting}
In this section we collect results that together suggest that the problem of characterizing classes of groups that recognize coordinates is nontrivial. 

\subsection{The failure of compactness}\label{S.FailureOfCompactness}
We have already seen that each one of~$\Sym_3$ and $\SL(2,5)$ recognizes coordinates but $\{\Sym_3,\SL(2,5)\}$ does not (Theorem~\ref{T.example:DNR}). In this section we prove the following.

\begin{theorem}\label{T.Limiting}
	There is a family $\fD$  of perfect groups that  does not recognize coordinates, but every finite subset of $\fD$ does. 
	All groups in $\fD$ are indecomposable and perfect. In particular, there is  no nontrivial homomorphism from a group in $\fD$ into the center of a group in $\fD$. 
\end{theorem}

Once proved, this  will show that the converse of Theorem~\ref{Fact:construction} is false  and dampens any hope that there is a simple characterization of when a class of groups recognizes coordinate.  It also shows that  the class of all perfect groups does not recognize coordinates (cf. Proposition~\ref{P.Perfect}).  The proof of this theorem is given at the end of this section, as a consequence of Proposition~\ref{P.Q/Z} below and  \cite[Theorem~4]{muranov2007finitely}.  

Let $\cw(H)$ denote the commutator width of a group $H$ (see \S\ref{S.Perfect}). 
In \cite[Theorem~1.1]{nikolov2004commutator}, Nikolov constructed a sequence of finite perfect groups~$H_n$ such that $\lim_{n\to \infty} \cw(H_n)=\infty$. We do not know whether Nikolov's groups satisfy the condition ($\dagger$) of Proposition~\ref{P.Perfect}. If infinitely many of them do, then this (together with Proposition~\ref{P.Q/Z}) would imply a failure of compactness 
for the notion of  recognizing coordinates and a failure of the converse to Theorem~\ref{Fact:construction}.

\begin{proposition} \label{P.Q/Z} 
		There is a family $\fC=\{G_n\mid n\in \bbN\}$ of quasisimple groups of commutator width 1 such that every abelian group admits a nontrivial homomorphism into the center of $\prod_{\Fin} G_n$. 
\end{proposition}

\begin{proof} For any $n\geq 2$, there are, by Dirichlet's theorem, infinitely many primes in the arithmetic sequence $(k\cdot n!+1)_{k\geq 1}$. We can therefore choose an increasing sequence of primes $(p_n)_{n\geq 2}$ such that $n!$ divides $p_n-1$ for all $n$. In particular $m!$ divides $p_n-1$ for all $n>m$.
	Let $G_n:=\SL(p_n-1,p_n)$.  By  Proposition~\ref{C.SL(n,F)},  the class $\fC=\{G_n\mid n\in \bbN\}$ recognizes coordinates. Also, since quasisimple groups satisfy ($\dagger$) of Proposition~\ref{P.Perfect}, this proposition implies that every finite subfamily of $\fD$ recognizes coordinates.

	We will now prove that $\bbQ/\bbZ$ embeds into $\prod_{\Fin} G_n$. 
	Since the multiplicative group of $F_{p^n}$ is the cyclic group $\bbZ/(p_n-1)\bbZ$,  every scalar matrix in $\GL(p_n-1,p_n)$ has determinant equal to $1$, and therefore the center of $\SL(p_n-1,p_n)$ is isomorphic to $\bbZ/(p_n-1)\bbZ$. Therefore, for every fixed $m\geq 2$ and all but finitely many $n$ we have that $Z(G_n)$ includes an isomorphic copy of $\bbZ/m!\bbZ$. 
	
	Let  ($o(a)$ denotes the order of the element $a$)
	\[
\textstyle 	Q:=\{(a_n)\in \prod_n Z(G_n)\mid (\forall m)(\forall^\infty n) o(a_n)\geq m\}. 
	\]
	Since for every $m$ and all but finitely many $n$,   $Z(G_n)$ is a finite cyclic group whose order is a multiple of~$m!$, $Q$ is a nontrivial (even uncountable) group. The image of $Q$ under the quotient map $\prod_n G_n\to \prod_{\Fin} G_n$ is divisible and we can recursively choose an isomorphic copy of $\bbQ/\bbZ$ inside it. This copy is clearly included in the center of $\prod_{\Fin} G_n$.

It remains to note that every nontrivial abelian group admits a nontrivial homomorphism into $\bbQ/\bbZ$ and every singly generated abelian group has a nontrivial homomorphism into $\bbQ/\bbZ$. The second part is trivial, and  the first part is an immediate consequence of Baer's theorem that divisible groups are injective in the category of abelian groups (e.g., \cite[Theorem~21.1]{fuchs2015abelian}). 	
	 This completes the proof.
\end{proof}

\begin{proof}[Proof of Theorem~\ref{T.Limiting}] 
	We first need to construct a family $\fD$ of perfect groups that does not recognize coordinates, but every finite subset of~$\fD$ does.

	By \cite[Theorem~4]{muranov2007finitely} for every $n$ there is a perfect, simple group $K_n$ such that $n+1\leq \cw(K_n)\leq 2n+2$.  All of these groups satisfy condition ($\dagger$) of Proposition~\ref{P.Perfect}.  Using groups $G_n=\SL(p_n-1,p_n)$ constructed in Proposition~\ref{P.Q/Z}, let   
	\[
	\fD:=\{G_n,K_n\mid n\in \bbN\}.
	\] 
	All groups in $\fD$ are perfect and of finite commutator width, hence every finite subset of $\fD$ recognizes coordinates by Proposition~\ref{P.Perfect}. (We also know that $\{K_n\mid n\in \bbN\}$ recognizes coordinates by Theorem~\ref{T.RecognizesCoordinates} \ref{1.C}, but this result does not apply to the non-simple groups $G_n$. Thus each one of $\{G_n\mid n\in \bbN\}$ and $\{K_n\mid n\in \bbN\}$ recognizes coordinates, but their union does not.)
	
	It remains to verify that $\fD$ does not recognize coordinates. 
	Since $\cw(K_n)\to \infty$ as $n\to \infty$, neither of the groups $\prod_n K_n$ and $\tilde K:=\prod_{\Fin} K_n$ is perfect. Therefore the latter group has a nontrivial abelian quotient, $A:=\tilde K/[\tilde K,\tilde K]$. 
	
Proposition~\ref{P.Q/Z} implies that there is a nontrivial homomorphism from~$A$ into $Z(\prod_{\Fin} G_n)$. 	The argument from the proof of the first part of Theorem~\ref{Fact:construction} gives an automorphism of $\prod_{\Fin} (G_n\times K_n)$ that does not respect coordinates.

	Finally, all groups in $\fD$ are perfect, hence none of them has an abelian quotient and in particular it does not admit a nontrivial homomorphism into the center of any other group. 
\end{proof}

\subsection{The failure of a weak converse to Theorem~\ref{Fact:construction}}\label{S.Failure}

Theorem~\ref{T.Limiting} implies that  the converse to Theorem~\ref{Fact:construction}  is false.  In the original draft of this paper we asked whether a weak converse to Theorem~\ref{Fact:construction}, asserting that every indecomposable group $G$ that does not admit nontrivial homomorphism into its center recognizes coordinates, holds. This was quickly answered  in the negative by Forte Shinko (\cite{shinko2025recognzing}) 
who proved that the unrestricted wreath product $ \bbZ\wr \bbQ$ (i.e., the semidirect product of $\bbZ^\bbQ$ and $\bbQ$) admits no nontrivial homomorphism into its center,  but its ultrapower does.  This group is clearly infinite, and we do not know whether every finite indecomposable group that does not admit a nontrivial homomorphism into its center recognizes coordinates and even whether the class of all finite groups with this property recognizes coordinates.  
A potential angle of attack to resolve the former question may be via Sylows theorems. The reason why the proof that $\Sym_3$ recognizes coordinates (and Proposition~\ref{P.2}) requires an additional effort is because it has a unique (hence normal) 3-Sylow subgroup. A finite indecomposable  group necessarily has non-unique $p$-Sylow subgroup for some $p$, and a generalization of Proposition~\ref{P.2} (taking $Z(G)=\{e\}$ into the account) may show that such $G$ recognizes coordinates. 

Shinko's proof shows that the property `$G$ does not admit a nontrivial homomorphism into $Z(G)$' is not first-order. 
We remark that the other property figuring in the statement of Theorem~\ref{Fact:construction}, of being \emph{indecomposable}, is not first-order either. 

\begin{example}\label{Ex.(in)deecomposable}
	There is an indecomposable group $G$ whose ultrapower $G^\cU$ is decomposable.  In particular, neither the class of decomposable groups  nor the class indecomposable groups is axiomatizable.   
	
	Take $G=\bbZ$ and $H=\bbZ^{\cU}$, an ultrapower of $\bbZ$ associated with a nonprincipal ultrafilter $\cU$ on $\bbN$. Then $H$ has a nontrivial maximal divisible subgroup.  Let 
	\[
	K=\{(a_n)\in \bbZ^{\bbN}\vert (\forall k\geq 2)(\forall^\cU n) k|a_n\}.
	\] 
	Then the image of $K$ under the quotient map from $\bbZ^\bbN$ onto $\bbZ^\cU$ is clearly a nontrivial divisible subgroup. By \cite[Section 4, Theorem 2.5]{fuchs2015abelian} every abelian group has a maximal (under inclusion) divisible subgroup and this subgroup is a direct summand. Since the elements of the diagonal copy of  $\bbZ$ in $\bbZ^\cU$ are not divisible, this gives a decomposition of $\bbZ^\cU$ into two nontrivial direct summands. 
\end{example}

\subsection{\texorpdfstring{$|\mathcal{L}|$-}{|L|-}compactness} 
In Theorem~\ref{T.Limiting} we have seen that compactness fails for the notion of recognizing coordinates. 
The following gives some compactness-like result (as common, the cardinality $|\calL|$ of a language $\calL$ is the cardinality of the set of its sentences).

\begin{proposition} \label{P.CountableCompactness} Suppose that $\fC$ is a class of structures of the same language $\calL$. Then $\fC$ recognizes coordinates if and only if every subset  of $\fC$  of cardinality $|\calL|$  recognizes coordinates. 
\end{proposition}

\begin{proof}
	To prove the nontrivial direction, suppose that every subset $\fD$ of $\fC$ of cardinality $\lambda$ recognizes coordinates. By Theorem~\ref{T.A}, 
	the formula $x =x'\to y= y'$ is equivalent to an $h$-formula in $\Th(\fD)$.  We claim that there is an $h$-formula $\varphi$ such that for every $\fD\subseteq \fC$ of cardinality $\lambda$, $\varphi$ defines the support in reduced products. Assume otherwise, and for every $h$-formula $\varphi$ fix  $\fD_\varphi\subseteq \fC$ of cardinality $\lambda$ such that $x=x'\to y=y'$ is not equivalent to $\varphi$ in $\Th(\fD_\varphi)$. Then $\fD=\bigcup_\varphi \fD_\varphi$ has cardinality $\lambda$, hence by the assumption some $h$-formula $\psi$  is equivalent to $x=x'\to y=y'$ in $\Th(\fD)$. However, $\Th(\fD)\subseteq \Th(\fD_\psi)$, contradiction. 
\end{proof}

\begin{example}\label{Ex.incompactness}
	For every regular cardinal $\lambda$ there are a language $\calL$ of cardinality $\lambda$ and a class $\fC$ of $\calL$-structures that does not recognize coordinates, but every $\fD\subseteq \fC$ of smaller cardinality recognizes coordinates. In particular, if $\lambda$ is uncountable then every reduced product $\prod_{i\in \bbN} \cM_i/\cI$ of structures in~$\fC$ recognizes coordinates but some reduced product of structures in~$\fC$ does not recognize coordinates.

	Let $\lambda$ be a cardinal and consider the $\mathcal{L}$-language $\{P_\beta : \beta < \lambda\}$, consisting of~$\lambda$ many ternary predicates.
	For $\alpha < \lambda$, denote by $\mathcal{M}_\alpha$ an infinite structure such that for every $\beta<\lambda$ we have 
	\[
	P_\beta(x,y,z)^{\cM_\alpha} \Leftrightarrow
	\begin{cases}
		x = y = z & \text{if } \beta \leq \alpha, \\
		x = z \rightarrow y = z & \text{if } \beta > \alpha. 
	\end{cases}
	\]
	It will suffice to prove that the following holds. 
	\begin{enumerate}
		\item[(a)] The class $\fC \coloneq \{M_\alpha\}_{\alpha < \lambda}$ does not recognize coordinates.
		\item[(b)] Every subclas of cardinality $< \operatorname{cf}(\lambda)$ recognize coordinates.  
	\end{enumerate}
	To prove (a),  consider the product $M := \prod_{\alpha < \lambda} M_\alpha / \mathcal{I}$ where $\mathcal{I}$ is the ideal of subsets of size $< \lambda$. All ternary predicates $P_\beta(x, y, z)$ define on $M$ the equality $x = y = z$; thus, the structure is just an infinite set and does not recognize coordinates. 
	
	To prove (b), consider $\fD\subseteq \fC$ of size $\mu < \operatorname{cf}(\lambda)$. Then, for some $\beta < \lambda$, the class is included in $\{M_\alpha\}_{\alpha < \beta}$, and it suffices to show that this latter class recognizes coordinates. This follows from \Cref{T.A}, as $P_\beta(x, y, z)$ is an $h$-formula defining the support function in any reduced product. The second part of (b) follows immediately. 
	
	Then (b) implies that the class $\{M_\alpha\}_{\alpha < \mu}$ recognizes coordinates in $\mu$ for every $\mu < \operatorname{cf}(\lambda)$ and concludes the proof. 
\end{example}

\subsection{Interpreting \texorpdfstring{$\Bool{I}{I} $ }{P(I)/I} is not enough}
In \Cref{T.eq.groups}, we saw that a group recognizes coordinates if and only if all reduced product $\prod_\Ideal{I} G$ interprets the support function for every ideal $\Ideal{I}$. It is however not sufficient that all reduced product $\prod_\Ideal{I} G$ interprets the Boolean algebra $\Bool{I}{I}$.
In this paragraph, we use a little variation of \Cref{T.RecognizingCoordinates} to show that reduced products of the group of quaternion $Q_8$ interprets the Boolean algebra; however $Q_8$ does not recognize coordinates by \Cref{T.example:DNR}.

\begin{proposition}\label{P.interpretingBool}
	Let $G$ be a non-abelian group with center $\mathcal{Z}$ such that for all $a\in G\setminus \mathcal{Z}$, $C(a^G)=\mathcal{Z}$. Then for every ideal $\Ideal{I}$ on an index set $\Index{I}$, the restricted product $\mathcal{M}=\prod_\Ideal{I} G$ interprets the support modulo $\mathcal{Z}(\mathcal{M})$, that is, the function 
    \[\mathcal{M} \rightarrow \Bool{I}{I}, ~a=(a_i) \mapsto \big[\{i \in \Index{I} \mid a_i\notin \mathcal{Z}\}\big]_\Ideal{I}.\]
    
\end{proposition}

\begin{proof}
	On $\mathcal{M}$, consider the binary relation $\trianglelefteq$:
	\[x \trianglelefteq y \quad  \Leftrightarrow \quad  (\forall w) ~\left(wx \neq xw  \rightarrow (\exists u) ~w^{u}y \neq yw^{u} \right).\]
	The relation $x \trianglelefteq y$ says that any element who does not commute with $x$ has a conjugate which does not commute with $y$. This is clearly a preorder: for all $x,y,z \in \mathcal{M}$ $x \trianglelefteq y\trianglelefteq z$. Consider the associated equivalence relation $x\sim y$ and the quotient $B:=\mathcal{M}/\sim$.  We need to prove the following:
	\begin{claim}
		For all $x,y\in \mathcal{M}$, $x\sim y$ if and only if $x$ and $y$ have the same support modulo $\mathcal{Z}(\mathcal{M})$. In particular $(B,\trianglelefteq)$ can be identify with $(\Bool{I}{I},\subseteq)$
	\end{claim}

	Consider two elements $x=(x_i)_i$ and $y=(y_i)_i$ who does not have the same support modulo $\mathcal{Z}(G)$. We may assume: 
	\[J:= \{i\in \Index{I} \mid y_i\notin \mathcal{Z} \wedge x_i\in \mathcal{Z} \} \notin \Ideal{I}.\]
	For all $i\in J$, let $a_i$ be an element in $G\setminus \mathcal{Z}$ which doesn't commute with  $y_i$. We set $w=(w_i)$ with 
	\[w_i=\begin{cases}  a_i,  & \text{ if } i \in J, \\
		e, & \text{otherwise.}
	\end{cases}\]
	Then $w\in \mathcal{M}$ and clearly $w$ doesn't commute with $y$ but all conjugates of $w$ commute with $x$, and therefore $y \not\trianglelefteq x $. Conversely, assume that $x$ and $y$ have the same support modulo $\mathcal{Z}(\mathcal{M})$ and let $z=(z_i)_i\in \mathcal{M}$ which does not commute with $x$. Then, \[J:=\{ i \in \Index{I} \mid z_ix_i \neq x_iz_i  \} \notin\mathcal{I}.\]
	In particular, for $i\in J$, $z_i$ and $y_i$ are not in $\mathcal{Z}$. Then by assumption, for all $i\in J$, there is $h_i$ such that $z_i^{h_i}$ doesn't commute with $x_i$ for $i\in J$. If $i\notin J$, set $h_i=e$ and let $h=(h_i)\in \mathcal{M}$. It follows that $z^h$ doesn't commute with $x$ in $\mathcal{M}$. Therefore  $x \trianglelefteq y$ and by symmetry, $x \sim y$. This concludes the proof.
\end{proof}

\begin{corollary} \label{C.Q8} The quaternions 
	$Q_8:=\left\langle a,b \, \middle| \, a^4=e, b^2=a^2, ba=a^{-1}b\right\rangle$ have the property that $\Bool II$ is definable in reduced product $\prod_\Ideal I Q_8$ but do not recognize coordinates. 
\end{corollary}

\begin{proof}
	The quaternions satisfy the assumptions of Proposition~\ref{P.interpretingBool} but they do not recognize coordinates by Corollary~\ref{T.example:DNR} \ref{DNR.Q8}. 
\end{proof}

Thus \cite[Theorem~7]{de2023trivial} does not apply 
to prove that forcing axioms imply all automorphisms of $\prod_{\Fin} Q_8$ are trivial, and on the other hand results of \cite{de2024saturation} cannot be used to prove that $\prod_{\Fin}Q_8$ is fully saturated by Proposition~\ref{P.stable} below. This appears to leave the possibility that forcing axioms imply all automorpisms of $\prod_{\Fin} Q_8$ are trivial. We show that this is not the case.

\begin{corollary} \label{C.Q8-automorphisms}The reduced power of $Q_8$ associated with $\Fin$ has $2^{2^{\aleph_0}}$ automorphisms, regardless of whether CH holds or not. 
	\qed  \end{corollary}

\begin{proof}
	By the proof of Theorem~\ref{T.example:DNR} \ref{DNR.Q8},  $Q_8$ admits a homomorphism onto its  (nontrivial) center. Therefore the conclusion follows by Lemma~\ref{L.GxH} below. 
\end{proof}

\begin{lemma}\label{L.GxH}
	If $G$ and $H$ are groups such $Z(H)$ is nontrivial and  $G$ admits a homomorphism $f$ onto $Z(H)$, then the reduced power of $G\times H$ associated with $\Fin$ has $2^{2^{\aleph_0}}$ automorphisms, provably in ZFC. 
\end{lemma}

	\begin{proof}
	By \cite[Theorem~1]{de2024saturation}, $\prod_{\Fin} Z(H)$ is, being stable,  saturated. It therefore has $2^{2^{\aleph_0}}$ automorphisms.  Also,  
	\[
	\Psi\colon (a_n)/\Fin\to f(a_n)/\Fin,
	\]
	defines a surjective homomorphism from $\prod_{\Fin} G$ onto $\prod_{\Fin}Z(H)=Z(\prod_{\Fin}H)$. 
	Since $\prod_{\Fin} (G\times H)\cong \prod_{\Fin} G\times \prod_{\Fin} H$, and $\Psi$ defines an endomorphism $\tilde \Phi$ of this group whose range is $Z(\prod_{\Fin} H)$. 
	For every automorphism $\Phi$ of $\prod_{\Fin} Z(H)$ we there have a unique automorphism of $\prod_{\Fin}(G\times H)$ defined by $a\mapsto a\tilde\Phi(a)$. 
\end{proof}

\section{Concluding remarks}

\subsection{An abstract criterion for recognizing coordinates}
We give an abstraction (and a generalization) of Theorem~\ref{T.RecognizingCoordinates} \eqref{2.RecognizingCoordinates}. The proof is analogous (the relation $R(x,y)$ corresponds to `$x$ and $y$ are in the same conjugacy class' and $S(x,y)$ to $xy=yx$).  This also gives a template for abstracting other results from \S\ref{S.Recognizing}.  

In the following, $h$-definable is short for `definable by an $h$-formula', and by `uniformly' we mean that the same formula works in all models of $T$.

\begin{proposition}
	Suppose that $T$ is a theory such that all models $M$ of $T$ satisfy the following.
	\begin{enumerate}
		\item There is a uniformly $h$-definable element $e$. 
		\item There is a uniformly $h$-definable binary relations $R$ and $S$ with the following properties for all $x,y$ in $M$.  
		\begin{enumerate}
			\item $R(e,x)$ $\Leftrightarrow$ $R(x,e)$ $\Leftrightarrow$ $x=e$. 
			\item  $S(e,x)$ and $S(x,e)$. 
			\item If $x\neq e$ and $y\neq e$, then there is $z$ such that $R(x,z)\land \lnot S(z,y)$. 
		\end{enumerate}
	\end{enumerate}
	Then  $T$ recognizes coordinates. 
\end{proposition}

\begin{proof}  
	
	Let $\varphi(x,z)$ be the formula $(\forall t)(R(z,t)\to S(t,x))$. Since $R(z,e)$ holds in all models of $T$, this is equivalent to an  $h$-formula.  Moreover, since  $\varphi(x,e)$ is true in every model of $T$, 
	\begin{equation}\label{eq.x=e.y=e}
		(\forall z)(\varphi(y,z)\lthen \varphi(x,z)),
	\end{equation}
	is equivalent to an $h$-formula.  By the assumption that $x\neq e$ implies there is $z$ such that $\lnot S(x,z)$, $\varphi(x,z)$ is equivalent to stating that $x=e$ or $z=e$. Therefore the displayed formula is equivalent to $x=e\to y=e$. Hence the relative support function $\supp_=$ is interpretable by an $h$-formula, and Lemma~\ref{L.RecognizingCoordinates}  implies the desired conclusion. 
\end{proof}

\subsection{Isomorphisms between (non-reduced) products}\label{S.Products}

Here we make some connections between recognizing coordinates and recognizing coordinates in products and direct sums. The following is the analog of Definition~\ref{D.recognizes-coordinates} in the context of products. 

\begin{definition}\label{D.product.recognizes-coordinates}
	An isomorphism $\Phi$ between products $\cM=\prod_i \cM_i$ and $\cN=\prod_i \cN_j $ is \emph{isomorphically coordinate respecting} if  there is a bijection $\pi\colon \Index J\to \Index I$ such that for all $j\in \Index J$ there is an isomorphism $\varphi_j\colon \cM_{\pi(j)} \to \cN_{j}$ such that 
	\[
	\Phi((a_i)_{i\in \Index I})=(\varphi_i(a_{\pi(i)})),
	\]
	for all $(a_i)_{i\in \Index I}$ in $\cM$. 
	
	A first-order theory $T$ is said to \emph{recognize coordinates in products} if every isomorphism between products of models of $T$ is isomorphically coordinate respecting. 
	
	More generally, if $\fC$ is a class of structures of the same language (not necessarily axiomatizable), then $\fC$ is said to \emph{recognize coordinates} if for every isomorphism between products of structures from $\fC$ is isomorphically coordinate respecting. 

\end{definition}
We remark that because products are a kind of reduced product, the following theorem holds trivially. 

\begin{theorem}
	Every class of structures that recognizes coordinates, in particular every class of groups listed in Theorem~\ref{T.RecognizesCoordinates} recognizes coordinates in products. 
\end{theorem}

Recognizing coordiantes in products is closely related to the classical Renek--Krull--Schmidt--Azumaya theorem that we now discuss, following \cite{kurosh1956theory}. In the original context of groups (or groups with operators---that is, groups with additional operations) this theorem asserts that if $G$ is a group which has a (finite) principal series of normal subgroups, then any two decompositions of $G$ into direct product of indecomposable factors are centrally isomorphic (\cite[p. 120]{kurosh1956theory}; see also its strengthening `The Fundamental Theorem' \cite[p. 114]{kurosh1956theory}).\footnote{Note that the assumptions are stated in terms of the product group, and not in terms of the indecomposable factors  as in our case.}
In other words, if $G=\prod_{i<m} G_i$ has principal series and $\Phi\colon \prod_{i<m} G_i\to \prod_{j<n} H_j$ is an isomorphism where all $G_i$ and all $H_j$ are indecomposable, then $m=n$, there is a bijection $\pi\colon n\to m$, and there are isomorphisms $f_j\colon H_j\to G_{f(j)}$ such that  $\Psi((a_i)_{i<m}):=(f_j(a_{\pi(j)})_{j<n})) $ defines an isomorphism which satisfies that $\Phi(a)\Psi(a)^{-1}$ is in $ Z(G)$ for all $a\in G$.

Our requirements on the factor groups are more stringent, since our assumptions on classes of groups in  Theorem~\ref{T.RecognizesCoordinates} imply that no nontrivial homomorphism from $G$ into its center exists. On the other hand, to the best of our knowledge, our result is the first extension of the 
Renek--Krull-Schmidt--Azumaya theorem to arbitrary infinite products. Along these lines,  Azumaya's theorem (see e.g., \cite{facchini2003krull}) and  the results of \cite{crawley1964refinements} are about infinite direct \emph{sums} of modules and arbitrary algebraic structures, respectively. 

It is curious that admitting nontrivial homomorphisms into the center gives a limiting example in this type of theorem. In \cite[p. 81]{kurosh1956theory} Kurosh gives an example of indecomposable groups $A$, $B$, $C$ and $D$ such that $A\times B\cong C\times D$ but neither $A$ nor $B$ is isomorphic to $C$ or to $D$. Each one of these groups has the center isomorphic to $\bbZ$ (actually, $D$ is isomorphic to $\bbZ$) and admits a homomorphism onto $\bbZ$. An even more interesting class of examples, showing that even the number of indecomposable factors is not an isomorphism invariant,  is given in \cite{baumslag1975direct}. 

\begin{remark} The main theorem of \cite{bidwell2008automorphisms} records a version of recognizing coordinates with respect to \emph{finite products} of finite groups. More explicitly, if $\fC = \{G_1,\dots ,G_n\}$ such that 
\begin{enumerate}
    \item For $i \leq n$, each $G_i$ is a finite group. 
    \item For $i \leq n$, each $G_i$ is indecomposible. 
    \item For $i, j \leq n$, there does not exist a nontrivial homomorphism from $G_i$ into the center of $G_j$.
\end{enumerate}
Then the class $\fC$ recognizes coordinates with respect to finite products. We remark that if $\fC$ is a finite class of finite groups with the properties above, then the proof from \cite{bidwell2008automorphisms} can be extended to show that any automorphism between direct sums of groups from $\fC$ is isomorphically coordinate respecting. A priori, the assumption of finiteness of the groups cannot be removed. At a critical juncture in the proof, one needs to use that any injective homomorphism between groups in $\fC$ is surjective. 
\end{remark}

Finally, we provide a example which demonstrates how the collection of automorphisms in the context of direct sums can appear wildly different than the collection of automorphisms in the context of reduced products.

\begin{example}  Let $p(n)$, for $n\in \bbN$, be distinct primes and let $G_n=\bbZ/p(n)\bbZ$  Then every automorphism of $\bigoplus_n G_n$ is trivial, but $\prod_{\Fin} G_n$ has nontrivial automorphisms (in ZFC). 
	
	For the former, note that in $\bigoplus_n G_n$ an element $g$ has order $p(j)$ if and only if $\supp(g)=\{j\}$. Therefore every automorphism of $\bigoplus_n G_n$ sends $G_n$ to itself. 
	
	For the latter, since all $G_n$ are abelian, the structure $\prod_{\Fin} G_n$ is  an abelian group. Its theory is therefore stable and it is saturated by \cite{de2024saturation}.  It therefore has $2^{2^{\aleph_0}}$ automorphisms, while clearly there are only $2^{\aleph_0}$ trivial automorphisms. 
\end{example}

\subsection{Rigidity corollaries}\label{S.RigidityCorollaries}
 As pointed out in the introduction, part of the motivation for this paper comes from the study of rigidity of quotient structures (see \cite{farah2022corona} for the current state of the art). We fix a language $\calL$ throughout. 
 If $\cM_i$, for $i\in \bbN$, is a sequence of $\calL$-structures, then an isomorphism between $\cM:=\prod_{\Fin} \cM_i$ and $\cN:=\prod_{\Fin} \cN_i$ is \emph{trivial} if there are a bijection $\pi$ between cofinite subsets of $\bbN$ and isomorphisms\footnote{The official definition requires $f_i$ only to be bijections. In case when the signature is finite, this is equivalent to asking that all $f_i$ be isomorphisms, but not in general;  see \cite[Definition~2.1, Lemma 2.2 (3), Example 2.3]{de2023trivial}. For the sake of brevity, we consider only finite signatures.} $f_i\colon \cM_{\pi(i)}\to \cN_i$ such that the map from $\prod_i \cM_i$ to $\prod_i \cN_i$ defined by 
 \[
 (a_i)_i\mapsto f_i(a_{\pi(i)})_i,
 \]
 lifts it. 

A moment of reflection reveals that every map that has a lifting of this sort is an isomorphism. A bit more work is required to figure out whether every isomorphism has such lifting.  This is reasonably well understood in case of reduced products of countable (possibly finite) structures over $\Fin$, and we concentrate on this case. The Continuum Hypothesis (CH) implies that reduced products of this sort are saturated and are therefore isomorphic if and only if they are elementarily equivalent. In this case, there are $2^{2^{\aleph_0}}$ isomorphisms. Since there are only $2^{\aleph_0}$ trivial isomorphisms, CH implies the existence of nontrivial isomorphisms.  See \cite[\S 6]{farah2022corona} for more on nontrivial isomorphisms. Also note that even the number of automorphisms of $\cP(\bbN)/\Fin$ can be strictly between $2^{\aleph_0}$ and $2^{2^{\aleph_0}}$ (\cite{Ste:Autohomeomorphism}). 

An automorphism $\Phi$ of $\cP(\bbN)/\Fin$ is \emph{trivial} if there is a bijection $\pi$ of cofinite subsets of $\bbN$ (such $\pi$ is called almost permutation) such that $X\mapsto \pi[X]$ lifts $\Phi$. (If one identifies $\cP(\bbN)/\Fin$ with the reduced power of the two-element Boolean algebra, then this is a special case of the general definition of a trivial isomorphism.) Thus the group of trivial automorphisms of $\cP(\bbN)/\Fin$ is the quotient of the semigroup of all almost permutations and its subsemigroup of eventually equal almost permutations. 
 By a seminal result of Shelah (\cite[\S V]{Sh:Proper}), it is relatively consistent with ZFC that all automorphisms of $\cP(\bbN)/\Fin$ are trivial.

 Theorem~\ref{T.Rigidity} below uses $\MA$ and $\OCAT$, consequences of the Proper Forcing Axiom commonly used in proofs of rigidity of quotient structures since the seminal paper \cite{Ve:OCA}; see \cite[\S 7.3]{farah2022corona} for additional background. These axioms are independent of ZFC. 
Part~\eqref{1.T.Rigidity} of the following implies Corollary~\ref{C.Rigidity} by Theorem~\ref{T.RecognizesCoordinates}. 
 
 \begin{theorem} \label{T.Rigidity} Assume $\OCAT$ and $\MA$. 
 	Suppose that $\fC$ is a class of groups that recognizes coordinates  and that $\cG:=\prod_{\Fin}\cG_i$ and $\calH:=\prod_{\Fin}\calH_i$ are reduced products of countable  or finite structures in $\fC$. 
 	\begin{enumerate}
 		\item  	\label{1.T.Rigidity} Then every  isomorphism between $\cG$ and $\calH$ is trivial.
\item \label{3.T.Rigidity} The automorphism group of $\cG$ is the semidirect product of $\prod_{\Fin} \Aut(\cG_i)$ and the group of trivial automorphisms of $\cP(\bbN)/\Fin$ associated with almost permutations $\pi$ such that $\cG_i\cong \cG_{\pi(i)}$ for all $i$. 
  	\end{enumerate}
In particular, if $G$ is any ground that recognizes coordinates, then the automorphism group of $\prod_{\Fin} G$ is isomorphic to the  semidirect product of $\prod_{\Fin} \Aut(\cG_i)$ and the group of trivial automorphisms of $\cP(\bbN)/\Fin$.
 \end{theorem}

\begin{proof}\eqref{1.T.Rigidity} is an immediate consequence of \cite[Theorem~7]{de2023trivial}. 
	The remaining two claims follow. 
	\end{proof}

The following diagonalization argument combined with the Feferman--Vaught theorem and saturation to produce reduced products that are nontrivially isomorphic (and typically, by rigidity results, nonisomorphic if forcing axioms assumed) is an application of Ghasemi's trick (\cite[Lemma~4.5]{brian2024conjugating}, also \cite[Lemma~5.2]{ghasemi2016reduced}). 

\begin{corollary}\label{C.Ghasemi}
	Suppose that $\fC$ is an infinite  class of groups that recognizes coordinates. Then there are groups $G_i$, for $i\in \bbN$, in $\fC$ such that for all infinite $X$ and $Y$ in $\cP(\bbN)$ for which $X\Delta Y$ is infinite the assertion $\prod_{\Fin} G_i\rs X\cong \prod_{\Fin} G_i\rs Y$ is independent from ZFC. 
\end{corollary}
	
	\begin{proof}
		Since $\fC$ is infinite, we can choose $G_i$ so that the theories of $G_i$ converge, in the sense that every sentence $\varphi$ of the language of the theory of groups either holds in all but finitely many $G_i$ or it holds in only finitely many of the $G_i$.
	Then the Feferman--Vaught theorem implies  $\prod_{\Fin} G_i\rs X\equiv \prod_{\Fin}G_i$ for all infinite $X\subseteq \bbN$. Thus CH implies that $\prod_{\Fin}G_i\cong \prod_{\Fin}G_i\rs X$, as these are elementarily equivalent saturated models. 
	
	On the other hand, since $\fC$ recognizes coordinates, by Theorem~\ref{T.Rigidity} $\OCAT$ and $\MA$ imply that every isomorphism between $\prod_{\Fin}\cG_i\rs X$ and $\prod_{\Fin} \cG_i\rs Y$  is associated with an almost bijection $\pi\colon X\to Y$ such that $\cG_i\cong \cG_{\pi(i)}$ for all $i\in \dom(\pi)$. Since $\cG_i\not\cong \cG_j$ for all $i\neq j$, such $\pi$ exists if and only if $X\Delta Y$ is finite. 
	\end{proof}
	
	There is no known `dividing line' for theories that recognize coordinates, and our results from \S\ref{S.FailureOfCompactness} suggest that the line, even if it exists, is rather rugged. 	
	By \cite[Theorem~1]{de2024saturation}, if the theory of a reduced product $\cM$ over $\Fin$  is stable and all $\cM_i$ have cardinality not greater than $2^{\aleph_0}$, then $\cM$ is fully saturated and therefore has $2^{2^{\aleph_0}}$ automorphisms. 
	 The class of stable groups that are reduced products is not very interesting---all such groups are abelian. 
	
	\begin{proposition}\label{P.stable}
		Suppose that $\cP(\bbN)/\cI$ is an atomless Boolean algebra and $\prod_\cI G_n$ is stable, or even NIP. Then the set $\{n\mid G_n$ is not abelian$\}$ belongs to $\cI$. 
	\end{proposition}
	
	\begin{proof}
		Assume otherwise and consider the formula $\varphi(x,y)$, `$xy=yx$'. 
		In each $G_n$ there are $a_n$ and $b_n$ such that $\varphi(G_n,a_n)$ and $\varphi(G_n,b_n)$ are distinct and have nonempty intersection (it contains $e_n$). By \cite[Theorem~2.10 (3)]{de2024saturation}, the theory of $G$ is not stable.
	\end{proof}

It is not difficult to see that stability is not a necessary condition for the existence of $2^{2^{\aleph_0}}$ nontrivial automorphisms of a reduced product over $\Fin$. We even have a natural example. 
By Corollary~\ref{C.Q8} and Corollary~\ref{C.Q8-automorphisms}, the reduced power of the quaternion group $Q_8$ has $2^{2^{\aleph_0}}$ automorphisms in ZFC although $\cP(\bbN)/\Fin$ is interpretable in it (and in particular its theory is unstable). 
As the referee pointed out, an even more natural example can be obtained as follows. If $G$ is a group that recognizes coordinates (e.g., any nonabelian  simple group) and $H$ is a nontrivial abelian group, then the reduced product of $G\oplus H$ is isomorphic to the direct sum of reduced products of $G$ and of $H$, and it has $2^{\aleph_0}$ automorphisms because  the second component does. On the other hand, the theory of the first component is not stable.  

\subsection{Continuous logic}
This paper is concerned with recognizing coordinates in classical, discrete, logic.  We conclude with a few words on the study of recognizing coordinates in the setting of continuous logic.   Metric analog of \cite[Theorem~7]{de2023trivial} was proved in \cite{de2024metric}. It asserts that   the usual forcing axioms imply that coordinate-respecting functions between  reduced products of separable metric structures are trivial. 
This result, together with the usual forcing axioms,  was applied in \cite{de2024automorphism} and in \cite{vignati2025rigidity} to prove rigidity results analogous to those of \S\ref{S.RigidityCorollaries} for universal sofic groups and Higson coronas, respectively. 

\begin{problem}
	State and prove the analog of Theorem 1 for continuous logic. 
\end{problem}

Solving this problem would require the analog of  Palyutin's theory of $h$-formulas in continuous logic,    developed in \cite{fronteau2023produits}.

	\subsection{Questions}
We conclude this paper with a few open problems. The first batch is  concerned with possible weak converses to Theorem~\ref{Fact:construction} (see also \S\ref{S.Failure} for related limiting examples).

\begin{question} 
		Does the class of all indecomposable groups with  trivial center recognize coordinates?  
\end{question}
	
\begin{question} Suppose that $G$ is a group with any of the following properties. Can one conclude that  it recognizes coordinates? 
\begin{enumerate}
	\item $G$ is indecomposable with trivial center. 
	\item $G$ is indecomposable and finite, with trivial center. 
	\item $G$ is indecomposable and finite and it does not admit a nontrjvial homomorphism into its center. 
\end{enumerate}		
\end{question}
	As pointed out earlier, Forte Shinko proved that the unrestricted wreath product $ \bbZ\wr \bbQ$ (i.e., the semidirect product of $\bbZ^\bbQ$ and $\bbQ$) is an indecomposable group that admits no nontrivial homomorphism into its center,  but its ultrapower does.

There are two different sources of coordinate recognition in groups, one given by the two parts of  Theorem~\ref{T.RecognizingCoordinates} and the other given by Proposition~\ref{P.2}. 
Each one of these sources comes in infinitely many varieties, the first indexed by prime numbers and 0 (the latter case corresponding with the domains, as in Corollary~\ref{C.Montse}) and the other indexed by a `length' (typically a commutator length, but see also the proof of Proposition~\ref{P.Z/2Z*Z/2Z}).  
	
	\begin{question} \label{Q.AnythingElse} Is there a class of groups that recognizes coordinates for a reason not covered by the results described in the previous paragraph? 
	\end{question}

		Theorem~\ref{T.A}  implies that if the answer to Question~\ref{Q.AnythingElse}  is positive, then the reason for coordinate recognition has to be coded by an $h$-formula.

	\bibliographystyle{plain}
	\bibliography{library}
	
\end{document}